\documentclass[11pt]{scrartcl}

\usepackage{amsmath}
\usepackage{amssymb}
\usepackage{amsthm}
\usepackage{cases}
\usepackage{commath}
\usepackage{enumerate}
\usepackage{xcolor}
\usepackage{graphicx}
\usepackage{float}
\usepackage{bm}

\usepackage{float}
\usepackage{algorithm}
\usepackage{algorithmic}

\newcommand{\curl}{\mathop{\mathbf{curl}}\nolimits}
\newcommand{\di}{\mathop{\mathrm{div}}\nolimits}
\newcommand{\R}{\mathbb{R}}
\newcommand{\N}{\mathbb{N}}
\newcommand{\nd}{n_p}
\newcommand{\dd}{d}

\newcommand{\Dt}{\tau}
\newcommand{\Ds}{\kappa}
\renewcommand{\norm}[2]{\|#1\|_{#2}}
\renewcommand{\O}{\Omega}

\newcommand{\rL}{\mathrm{L}}
\newcommand{\rH}{\mathrm{H}}

\newcommand{\partialt}{d_t}
\newcommand{\partials}{d_s}

\newcommand{\bv}[1]{{\boldsymbol #1}}

\newcommand{\balpha}{\bv{\alpha}}
\newcommand{\btheta}{\bar{\theta}}
\newcommand{\wtheta}{\widetilde{\theta}}

\newcommand{\bbalpha}{\bar{\bv{\alpha}}}

\newcommand{\balphaI}{\balpha_{int}}
\newcommand{\hD}{\widehat{D}}
\newcommand{\bX}{\mathbf{X}}
\newcommand{\bx}{\boldsymbol{x}}
\newcommand{\bvarphi}{\boldsymbol{\varphi}}
\newcommand{\brho}{\boldsymbol{\rho}}

\newcommand{\hbx}{\widehat{\bx}}
\newcommand{\bvel}{\mathbf{f}}
\newcommand{\rmvel}{\mathrm{f}}
\newcommand{\hbvel}{\widehat{\rmvel}}
\newcommand{\eps}{\varepsilon}
\newcommand{\tnu}{\tilde{\nu}}
\newcommand{\calH}{\mathcal{H}}
\newcommand{\calU}{\mathcal{U}}
\newcommand{\calV}{\mathcal{V}}
\newcommand{\calJ}{\mathcal{J}}
\newcommand{\calJs}{\mathcal{F}}
\newcommand{\calC}{\mathcal{C}}
\newcommand{\Ccone}{\mathcal{A}}
\newcommand{\wbP}{\widetilde{\bv{P}}}
\newcommand{\hbP}{\widehat{\bv{P}}}

\newcommand{\bB}{\bv{B}}
\newcommand{\bF}{\bv{F}}
\newcommand{\bmag}{\bv{m}}
\newcommand{\bH}{\bv{H}}
\newcommand{\bM}{\bv{M}}
\newcommand{\mH}{\mathbb{H}}

\newcommand{\weakto}{\rightharpoonup}
\DeclareMathOperator*{\argmin}{arg\,min}

\newcommand{\pair}[1]{\left\langle #1 \right\rangle}

\newcommand{\abssec}[1]{\noindent\normalsize {\bfseries #1\quad }\ignorespaces}
\renewenvironment{abstract}{\abssec{Abstract}}{\par\vspace{.1in}}
\newenvironment{keywords}{\abssec{Key Words}}{\par\vspace{.1in}}
\newenvironment{AMSMOS}{\abssec{AMS subject
  classification}}{\par\vspace{.1in}}

\theoremstyle{plain}
\newtheorem{theorem}{Theorem}[section]

\newtheorem{lemma}[theorem]{Lemma}

\theoremstyle{definition}
\newtheorem{remark}[theorem]{Remark}

\numberwithin{equation}{section}

\begin{document}

\title{Optimizing the Kelvin force in a moving target subdomain
\thanks{The work of  H. Antil has been partially supported by NSF grants DMS-1109325 and DMS-1521590.
 R.H. Nochetto has been partially supported by NSF grants DMS-1109325 and DMS-1411808 and
  P. Venegas  has been  supported by NSF grant DMS-1411808 and CONICYT  scholarship (Chile).}
}

\author{Harbir Antil \thanks{Department of Mathematical Sciences. George Mason University, Fairfax, VA 22030, USA. Email: {\tt hantil@gmu.edu}. }
\and
Ricardo H. Nochetto \thanks{Department of Mathematics and Institute for Physical
Science and Technology, University of Maryland College Park, MD 20742, USA. Email: {\tt rhn@math.umd.edu}. }
\and
Pablo Venegas  \thanks{GIMNAP, Departamento de Matem\'atica, Universidad del B{\'\i}o B{\'\i}o, Chile. Email: {\tt pvenegas@ubiobio.cl}.}
}

\maketitle

\begin{abstract}
In order to generate a desired Kelvin (magnetic) force in a target subdomain moving along a prescribed trajectory, we propose a minimization problem with a tracking type cost functional. We use the so-called dipole approximation to
realize the magnetic field, where the location and the direction of the magnetic sources are assumed to be  fixed. The magnetic field intensity acts as the control and exhibits limiting pointwise constraints. We address two specific problems: the first one corresponds to a fixed final time whereas the second
one deals with an unknown force to minimize the final time. We prove existence of solutions and deduce local uniqueness provided that a second order sufficient condition is valid. We use the classical backward Euler scheme for time discretization. For both problems we prove the $\rH^1$-weak convergence of this semi-discrete numerical scheme. This result is motivated by $\Gamma$-convergence and does not require second order sufficient condition. If the latter holds then we prove $\rH^1$-strong local convergence. We report computational results to assess the performance of the numerical methods. As an application, we study the control of magnetic nanoparticles as those used in magnetic drug delivery, where the optimized Kelvin force is used to transport the drug to a desired location.
\end{abstract}


\begin{AMSMOS}
65D05,    %%   Interpolation
49J20,    %%   Optimal control problems involving partial differential equations
49M25,    %%   Discrete approximations
65M12,    %%   Stability and convergence of numerical methods
65M60    %%   Finite elements, Rayleigh-Ritz and Galerkin methods, finite methods
\end{AMSMOS}

%------------------------------------------------
%----  Introduction
%------------------------------------------------

\section{Introduction}
\label{sec:magnetic}

It is well-known that the magnetic field exerts a force on magnetic materials such as magnetic
nanoparticles (MNPs).  This principle has been widely exploited.
For instance, MNPs under the action of external magnetic field are used in medical sciences:
as contrast agents  to enhance the contrast in MRI \cite{CRIP2006,SLZ2008}, as carriers for
targeted drug delivery \cite{LAB2001,MCN1963}, to treat cancer and tumor cells in magnetic
hyperthermia \cite{JWFJHF1993,MUNNWTJ2011}, in gene therapy \cite{D2006} and in magnetized
stem-cells \cite{SKL2008}, among others. The application of magnetic force is not
restricted to the medical sciences; these forces are relevant in %applications of this control problem include
magnetic tweezers \cite{DKDK2003,HJBTF2003}, lab-on-a-chip systems that include
magnetic particles or fluids \cite{GLL2010,LHPVRG2006},  magnetofection \cite{FX2012,SBRHGKHBP2005} or
separation of particles \cite{ZDVTX2013,LYL2014}, just to name a few.

To understand how a magnetic field can manipulate MNPs, we need to recall that a magnetic
field gradient is required to exert a force at a distance, such a magnetic force is given by \cite{R1997}:
\begin{equation}\label{eq:force_1}
\bF=(\bmag \cdot\nabla)\bH,
\end{equation}
where $\bmag$ is the magnetic dipole moment and $\bH$ is the magnetic field.
To address the computation of \eqref{eq:force_1} we  analyze  the
magnetic force acting on a point-like magnetic dipole.
In the case of a magnetic nanoparticle suspended in a weakly
diamagnetic  medium such as water, the total moment on the particle
can be written as $\bmag  = V_m\bM$, where $V_m$ is the volume of the
particle and $\bM$ is its volumetric magnetization.
In diamagnetic and paramagnetic materials, the relation between $\bM$ and
$\bH$ is linear, which in turn
is given by $\bM = \Delta \chi \bH$, where $\Delta \chi =  \chi_p -  \chi_m$
is the effective susceptibility, namely, the difference in magnetic susceptibility
between the magnetic particle, $\chi_p$, and its surrounding buffer or medium, $\chi_m$.
Furthermore, provided there are no time-varying electric fields
or currents in the medium, we can apply the Maxwell equation
$\curl \bH = 0$  so that equation \eqref{eq:force_1} becomes:
\begin{equation}\label{eq:F}
\bF=\dfrac{V_m\Delta\chi}{2}\nabla|\bH|^2
\end{equation}
where we have used the identity:
$\nabla(\bH \cdot \bH) = 2\bH \times (\curl \bH) + 2(\bH \cdot \nabla)\bH = 2(\bH\cdot \nabla)\bH$.
A similar expression has been considered, for instance, in a simplified version of
ferrohydrodynamics equations \cite{R1997,R2002}.

The immediate and fundamental difficulty in correlating the magnetic field with the
physically observable forces exerted on the elementary magnetic entities (for instance,
nanoparticle or ferrofluid) is that the magnetic field intensity $\bH$ is not parallel
to the magnetic force $\bF$. Instead, it  may take any direction relative to $\bF$
depending on the spatial distribution of the magnetic field sources (or field gradients).

Due to the physics of magnetic fields and forces,
the majority of magnetic application
systems have been designed to pull in or attract therapeutic particles
to target regions (see Figure~\ref{fig:push} (left)). It is, however, also possible
to use two or more magnets to  ``magnetically inject" nanoparticles
(see Figure~\ref{fig:push} (center and right)) particles \cite{SDR2010} or
to fully manipulate microrobots in wireless micromanipulation \cite{KAKBSN2010}.
\begin{figure}[!h]
\centering
\includegraphics[width=0.3\textwidth]{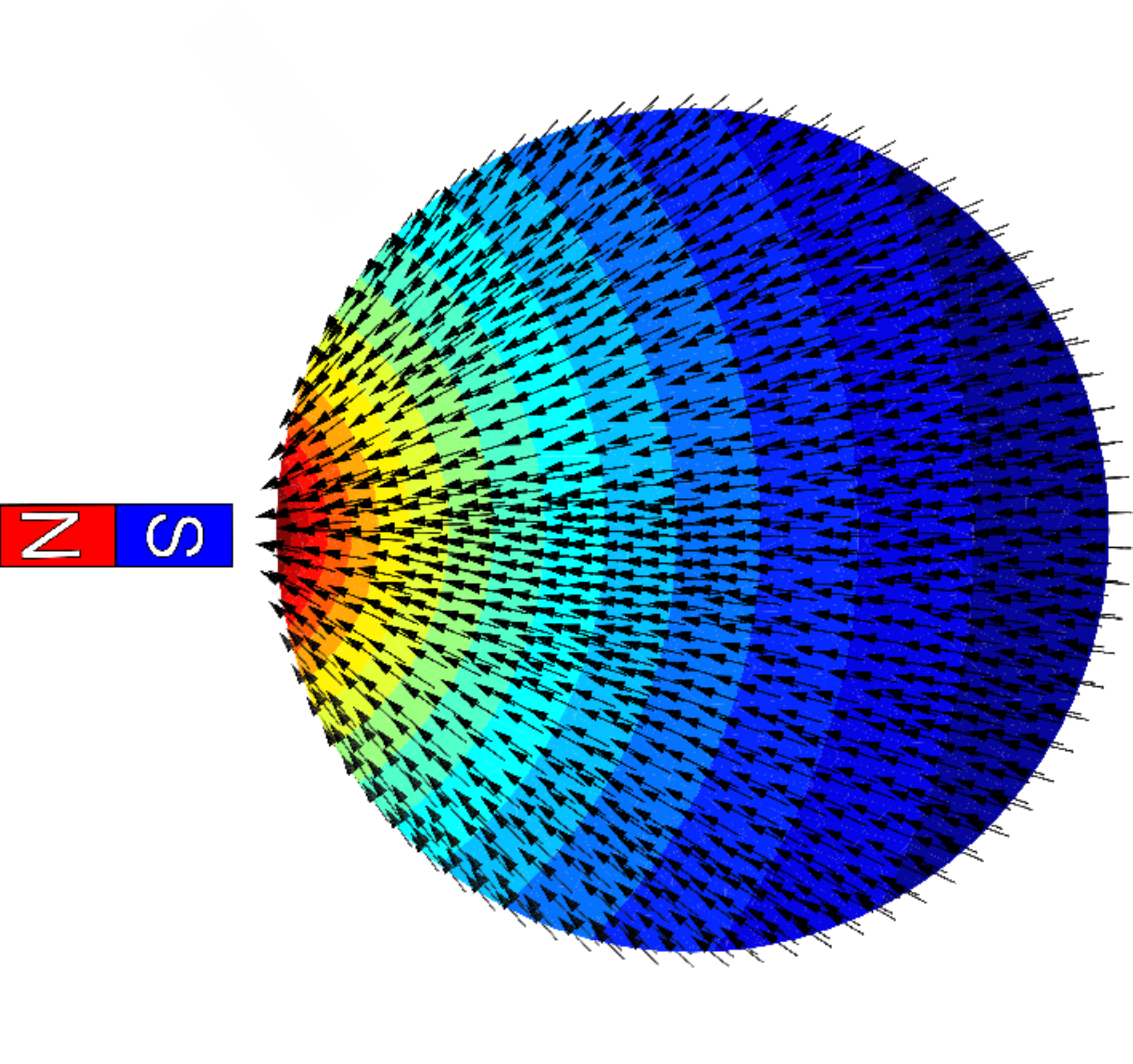}\qquad
\includegraphics[width=0.25\textwidth]{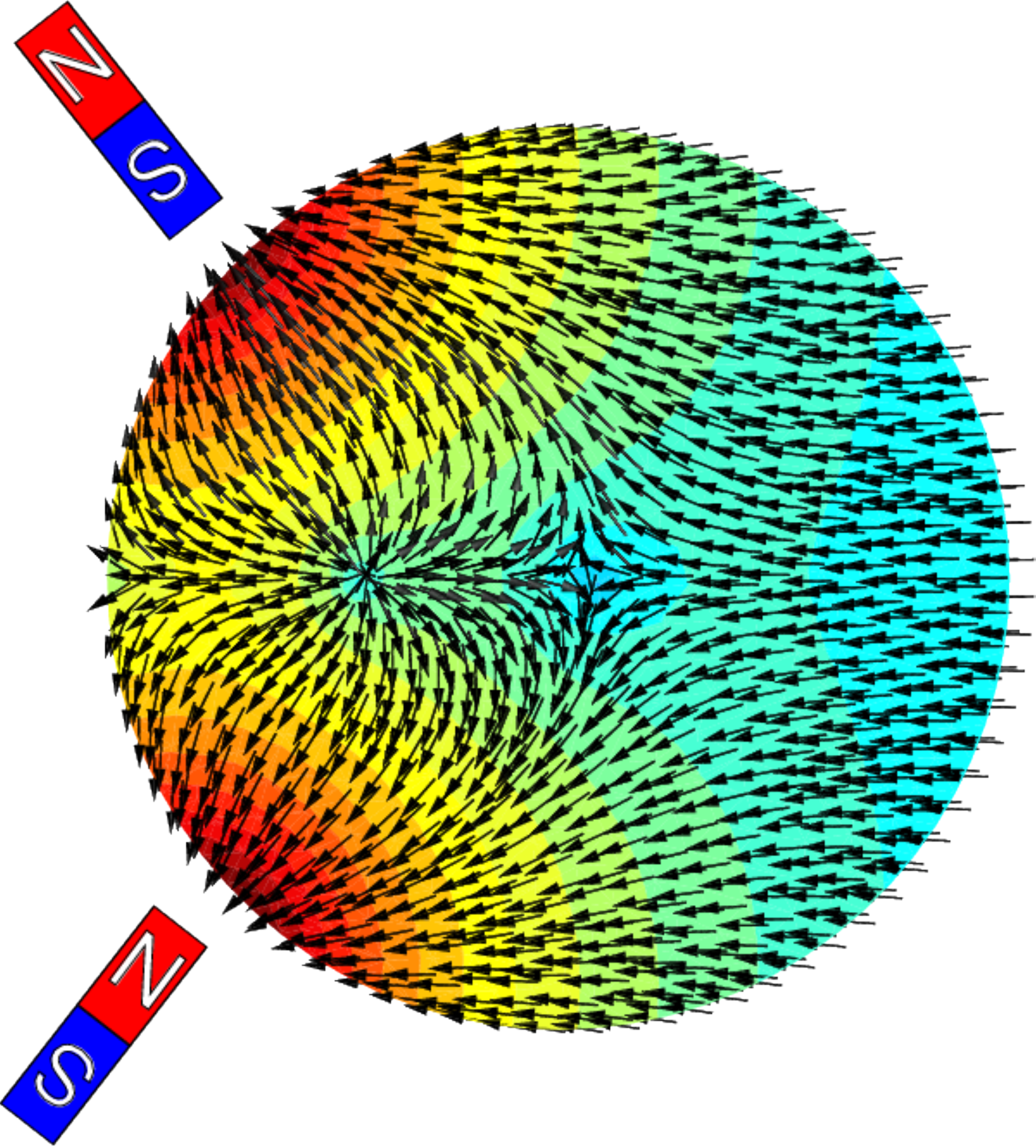}\qquad
\includegraphics[width=0.25\textwidth]{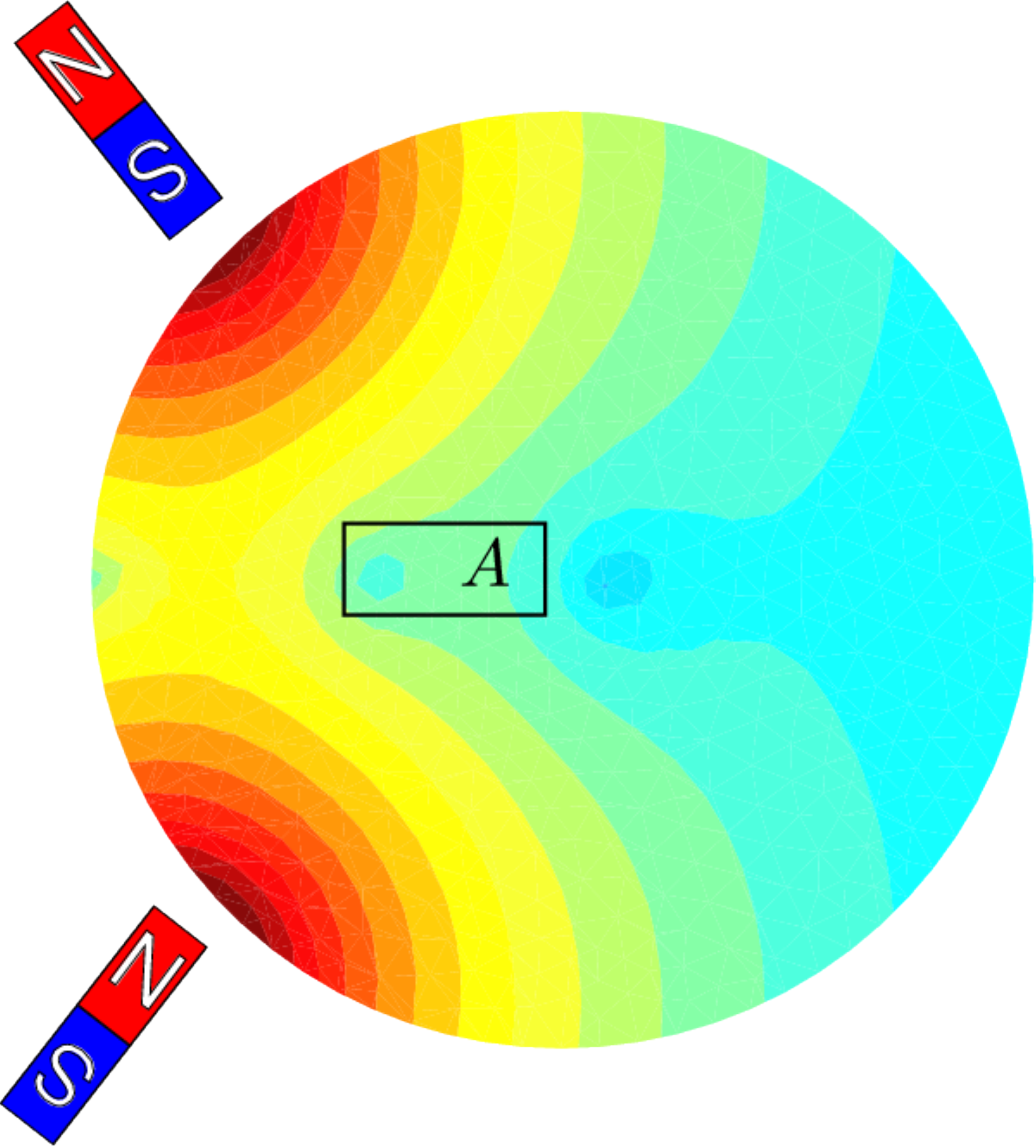}
\caption{The panels show the magnetic force $\bF$, with arrows indicating the normalized
 force direction generated using one and two permanent magnets.
 The background color indicates the magnetic force magnitude  $|\nabla|\bH|^2|$ on a log scale.
In particular, the light and dark colors correspond to low and high values of $|\nabla|\bH|^2|$.
The left panel shows that using a permanent magnet one can generate $\bF$
which can be used to pull in or attract particles to a region of interest.
In contrast, using two permanent magnets we can generate $\bF$ that enables us to
``magnetically inject" particles. On the right  panel (same as in the center without arrows),
region $A$ contains a zone where the arrows point to the right and thus enable pushing
away from the magnets.}
\label{fig:push}
\end{figure}
As we notice, the interaction of these dipoles generates a zone
where the magnetic force is pointing outward (see region A, Figure~\ref{fig:push} (right)).
This field can push particles away from the dipole positions.
Indeed, the success of the aforementioned  applications highly depend
on the accurate control of the magnetic force.
Such a control will enable us to better understand the existing and explore new applications of MNPs,
ferrofluids and magnetic force-based models in general.
In this paper we focus on a key question: how to approximate a desired magnetic force $\bvel$
by actuating a configuration of magnetic field sources whose location and direction are fixed.
We aim to achieve this goal by studying a minimization problem
\[
\min_{\bF} \int_0^T \| \bF - \bvel \|_{\rL^2(D)}^2 \ dt\quad
\mbox{for } T>0 \mbox{  and } D\subset \R^d, d=2,3.
\]
A good approximation of  $\bvel$ computed by the minimization problem
enables us to manipulate, for instance, MNPs with the diversity of applications that this entails.

We begin this paper in Section 2 by discussing the mathematical formulation of two minimization problems.
In Section~\ref{s:fft}, we assume that the final time $T$ and the vector field $\bvel$
are given. In Section~\ref{s:mft} we study the problem with an unknown vector field so as to minimize the final time $T_F$. To tackle this, we
replace time with arc length. For both problems, we prove global existence and local uniqueness
of minimizers, the latter provided a second order sufficient condition holds. Section~\ref{s:disc} is devoted to the numerical analysis of the time-discrete
problems. We obtain first a $\rH^1$-weak convergence, which is motivated by $\Gamma$-convergence. In addition,
by assuming a second-order sufficient condition we prove $\rH^1$-strong local convergence for the first problem. In Section~\ref{s:numerics} we report  numerical tests that assess the ability of both approaches to approximate spatially uniform vector fields in a moving subdomain and the performance of the corresponding numerical schemes. We conclude with an  example which illustrates
how  the optimal magnetic force, generated by the minimization problem, can be used in magnetic drug targeting.

%We begin this paper by discussing the mathematical formulation of  a minimization problem in Section 2.
%For this problem we assume that the final time and the vector field $\bvel$ are given.
%We prove existence and local uniqueness of minimizers.
%Section 3 is devoted to the numerical analysis of the semi-discrete in time problem arising from
%a backward Euler time-discretization. A $\rH^1$-weak convergence result is obtained by applying
% $\Gamma$-convergence theory and, by assuming the second-order sufficient condition, a
% $\rH^1$-strong local convergence result is proved. In Section 4 we study the problem with an unknown final time.
%{To tackle this,} we replace the time with the arc length.
%  For this new problem we prove the global existence, local  uniqueness and $\Gamma$-convergence
%  of a discrete approximation.
%In Section 5 we report  numerical tests in order to assess the
%validity  of both approaches and the proposed numerical schemes.
%We conclude with an  example which illustrates how  the optimal
% magnetic force, generated by the minimization problem, can be used in magnetic drug targeting.

%------------------------------------------------
%----  Minimization Problem
%------------------------------------------------

\section{Minimization and control}

Let $\O\subset \R^\dd $, $\dd=2,3$
be an open bounded domain and $T>0$ be the final time.
With $D_t$ we denote a time dependent domain
that deforms smoothly and is strictly contained in $\Omega$ for every $t\in[0,T]$.
Our goal is then to approximate a vector field $\bvel \in [\rL^{2}(0,T;\rL^2(\Omega))]^\dd$ by the
 so--called Kelvin
 force (cf.~\eqref{eq:F}). We consider magnetic sources outside $\Omega$, then, from
 the  Maxwell equations it follows that the magnetic field $\bH$ satisfies:
\begin{align}\label{eq:control_h}
\curl \bH&=\bv{0}, \quad   \di \bH=0 \quad \mbox{in } \Omega,
\end{align}
where the last equation follows by assuming a linear relation between the magnetic induction
$\bB$ and the magnetic field $\bH$. The magnetic field generated by a current distribution
and a permanent magnet  can be modeled by the Biot-Savart law, which is a magnetostatic approximation.
However, for simplicity, in our case we consider a dipole approximation to the magnet
source (see Figure~\ref{fig:3dipoles}), which provides a concise and easily tractable representation
 of the magnetic field (see \cite{PA2013} for a quantification  of the error associated
with the dipole approximation). This approximation is commonly used
 for localization of objects in applications ranging from medical imaging to
military. It is also extensively used in real-time control of magnetic devices 
in medical sciences \cite{FKA2010,MA2011,NKA2010}.
\begin{figure}
\begin{center}
\includegraphics[width=0.22\linewidth]{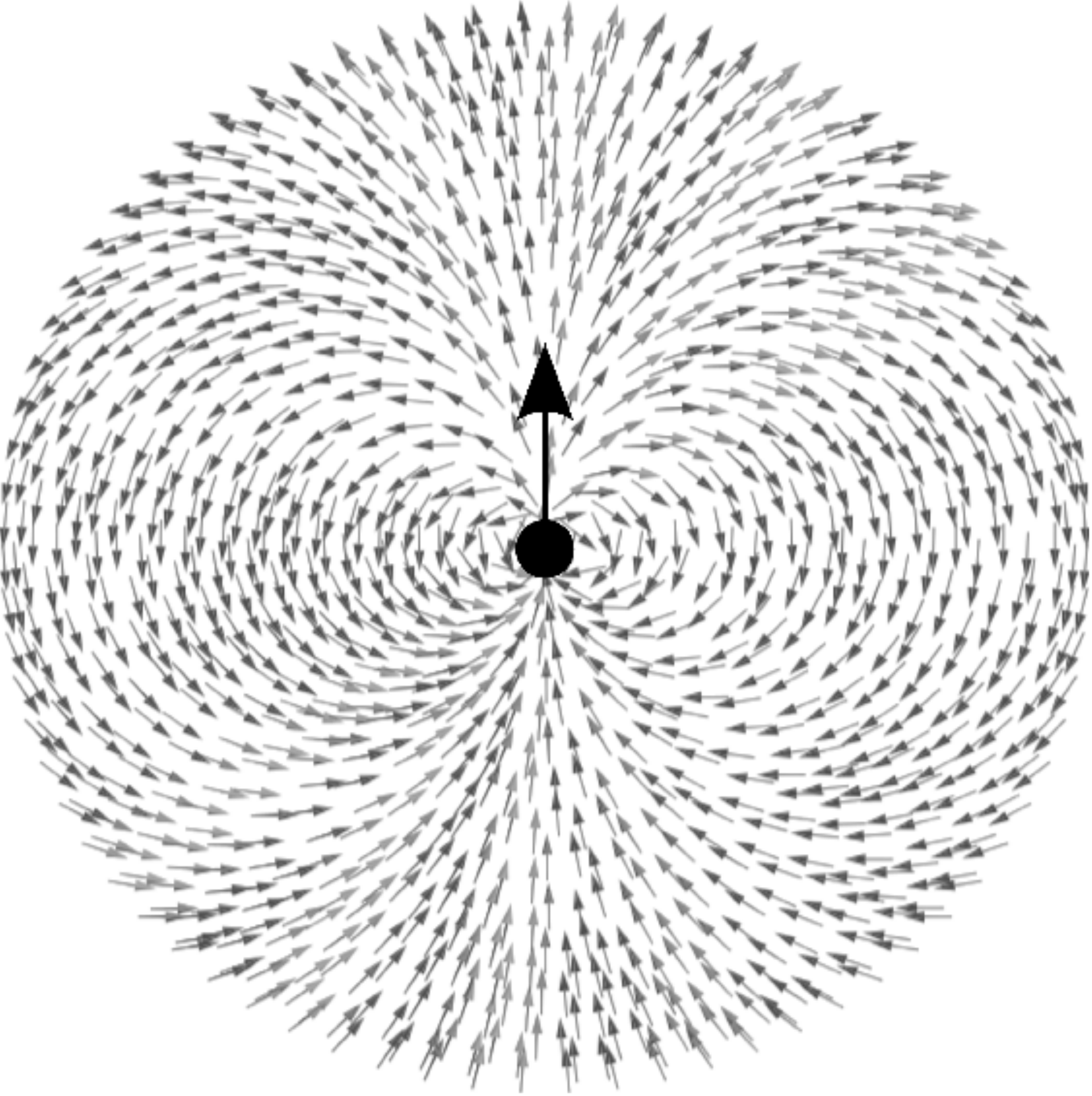}\qquad
\includegraphics[width=0.22\linewidth]{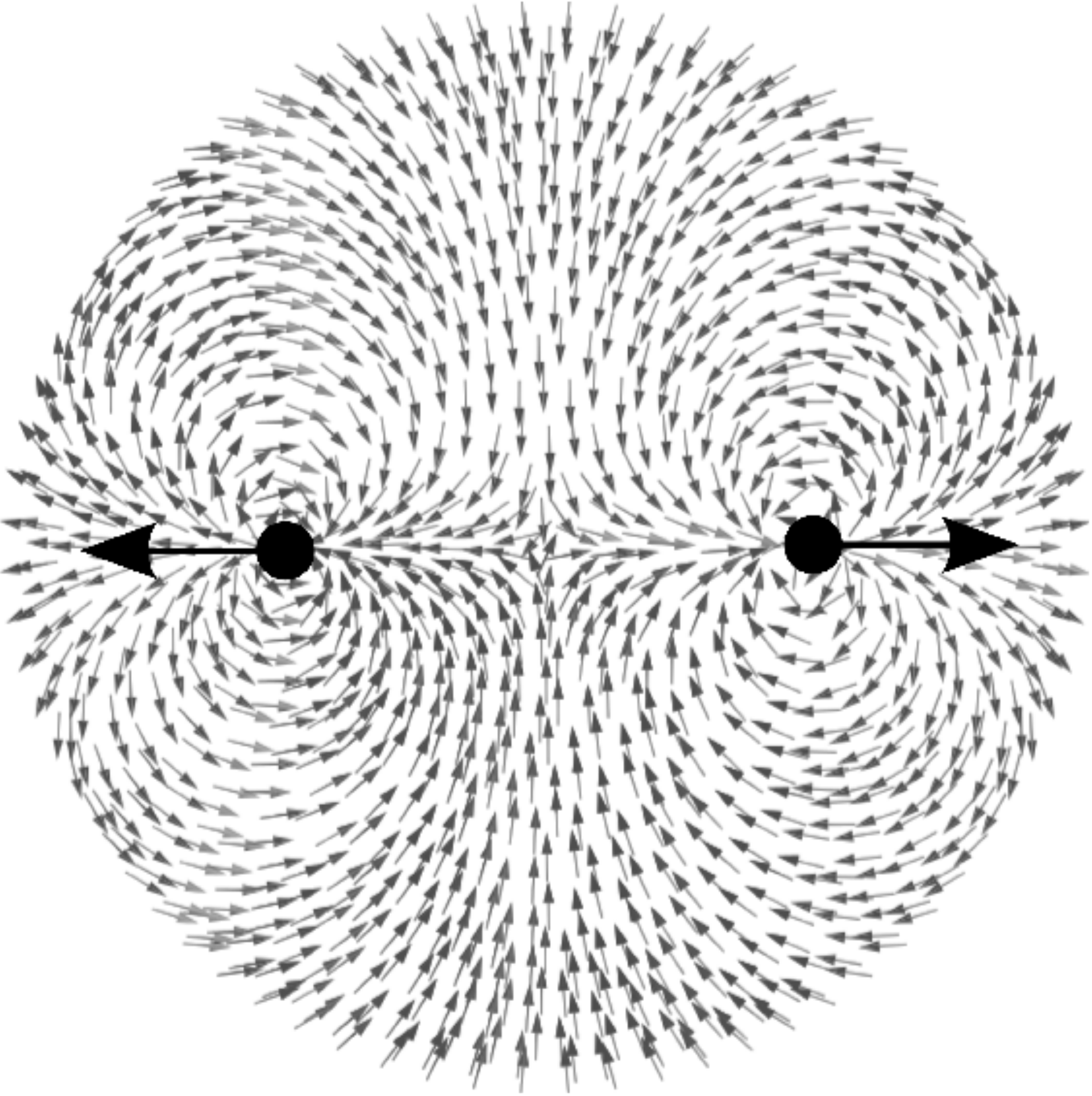}\qquad
\includegraphics[width=0.22\linewidth]{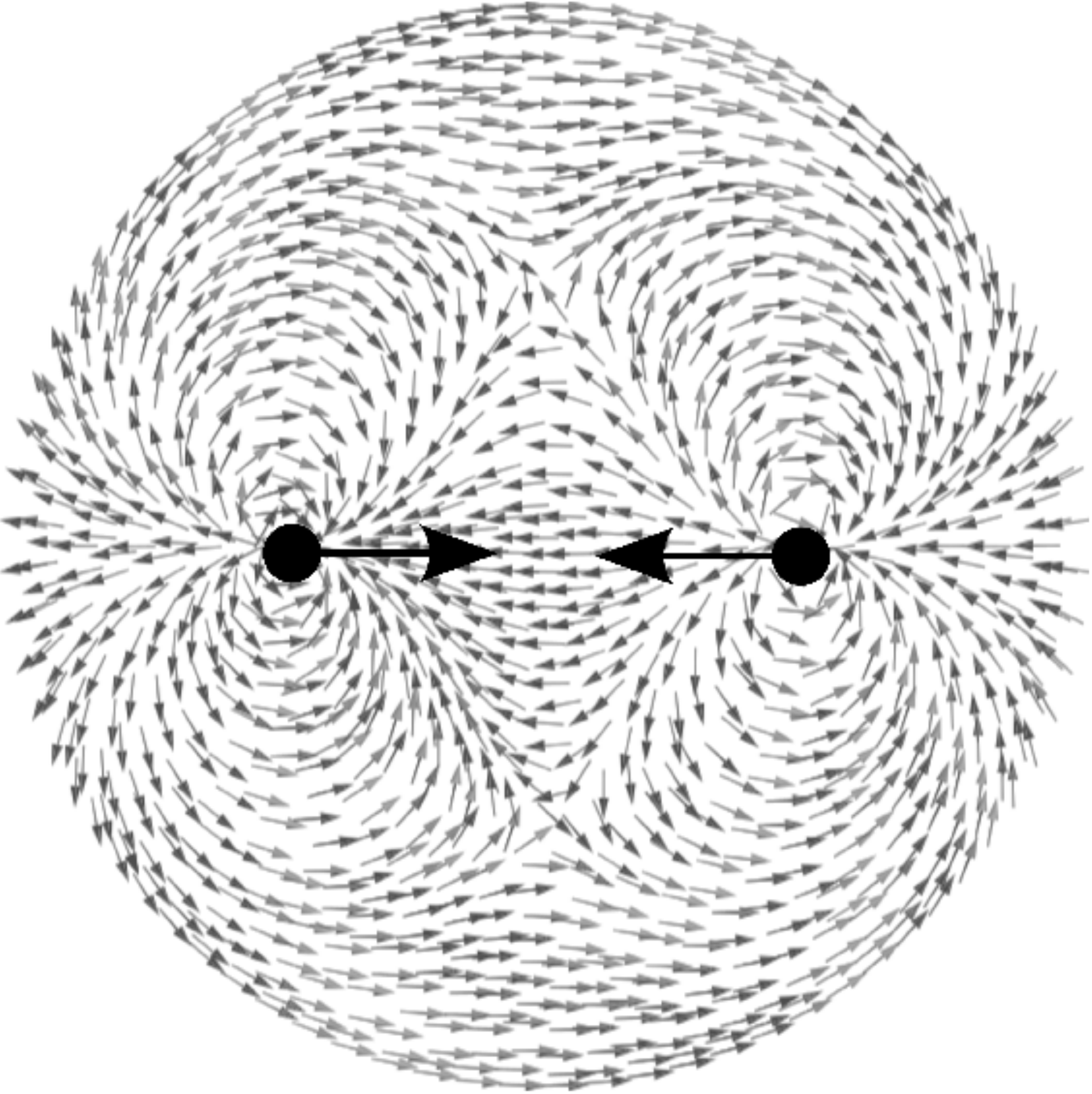}
\caption{These panels depicts the magnetic field $\bH$ (represented by small arrows)
generated using the dipole approximation \eqref{eq:h_mag}
for unit intensities ($\alpha_i=1$, $i=1,\ldots, \nd$).
The $i^{th}$ dipole is located at $\bx_i \in \Omega$  %, which lies outside the circular region $\Omega$
with direction $\widehat{\bv{d}}_i$. Left panel: single dipole with $\widehat{\bv{d}}_1 = (0,1)$ and $\bx_1 = (0.0)$.
Center panel: two dipoles with $\widehat{\bv{d}}_1 = (-1,0)$,  $\widehat{\bv{d}}_2 = (1,0)$ and
$\bx_1 = (-0.5,0)$, $\bx_2 = (0.5,0)$. Right panel: two dipoles with $\widehat{\bv{d}}_1 = (1,0)$,
 $\widehat{\bv{d}}_2 = (-1,0)$ and $\bx_1 = (-0.5,0)$, $\bx_2 = (0.5,0)$.  For presentation purposes,
 we consider the normalized magnetic field on each example.
}
\label{fig:3dipoles}
\end{center}
\end{figure}
In addition,  fixed location and direction of magnetic dipoles but variable intensity field,
 is considered in several applications such as microrobots micromanipulation \cite{KAKBSN2010}
 or  magnetic drug targeting \cite{SLDS2013}.
 With such a configuration it is possible  to reduce the numbers of free variables
 (fixed location and direction of dipoles) and the hardware involved.

Since we are interested in the feasibility of the field approximation  rather than the optimality of the
prescribed location, from now on we will  assume that the magnetic field is modeled by the superposition
of a fixed number $\nd$ of dipoles, namely
\begin{equation}\label{eq:h_mag}
\bH(\bx,t)=\sum_{i=1}^{\nd}\alpha_i(t)\del{\dd\dfrac{(\bx-\bx_i)(\bx-\bx_i)^\top}{|\bx-\bx_i|^2}-\mathbb{I}}\dfrac{\widehat{\bv{d}_i}}{|\bx-\bx_i|^\dd}=
\sum_{i=1}^{\nd}\alpha_i(t)\bH_i(\bx),
\end{equation}
where $\mathbb{I}\in \R^{\dd\times \dd}$ is the identity matrix. In addition,
$\widehat{\bv{d}_i}\in\R^\dd$ and $\bx_i\in\R^\dd\setminus\overline{\O}$, $i=1,\ldots,\nd$, denote fixed
unit vectors and dipole positions, respectively (see Figure~\ref{fig:domain}).
It is straightforward to show that the magnetic field given by \eqref{eq:h_mag}
satisfies \eqref{eq:control_h}.

\subsection{Problem 1: Fixed final time}\label{s:fft}
 With this configuration in mind, we introduce the  minimization problem
\begin{subequations}\label{eq:PJ1}
\begin{align}\label{eq:control_problem_t}
\min_{\balpha\in \calH_{ad} } \calJ(\balpha)
\end{align}
with
\begin{equation}
\calJ(\balpha):=\dfrac{1}{2}\int_0^T\|\nabla|\bH(\balpha)|^2-\bvel\|^2_{\rL^2(D_t)}dt
+\dfrac{\lambda}{2}\int_0^T|\partialt\balpha|^2dt.
\end{equation}
\end{subequations}
Here $\balpha(t):=(\alpha_1(t),\ldots,\alpha_{\nd}(t))^\top\in\R^{\nd}$ denotes the
vector of magnetic field intensities, and  $\lambda >0$ is the cost of control.
A nonzero $\lambda$
in \eqref{eq:control_problem_t} will enforce a smooth evolution of the intensities.
Moreover, the larger the value of $\lambda$, the smoother is this evolution.
For given  constant vectors $\balpha_0,\balpha_*, \balpha^*\in\R^{\nd}$, we
seek $\balpha$  in the following admissible convex set:
\begin{equation}\label{eq:ad_set}
\calH_{ad} := \left\{\balpha \in [\rH^1(0,T)]^{\nd} :  \balpha(0)=\balpha_0  \quad \mbox{and}
\quad  \balpha_* \leq \balpha(t)\leq \balpha^*,\,\, \forall t\in [0,T]\right\}.
\end{equation}
Notice that the feasibility of the dipole approximation can be studied with the first term of
$\calJ$. Applications of \eqref{eq:PJ1} include control of drug concentration, introducing
particles containing nucleic acids into target cells (magnetofection), and to separate magnetic
materials from a nonmagnetic liquid medium.

The resulting magnetic force arising as a solution to \eqref{eq:PJ1} can be used to
 control drug concentration, cell, or to separate
magnetic materials from a nonmagnetic liquid medium, among other applications.
\begin{remark}\rm
Changing the number of dipoles, their positions or magnetic field directions leads to different configurations for
the minimization problem \eqref{eq:PJ1}. Although the computed force may vary depending on the configuration, the
following mathematical analysis remains the same.
%The minimization problem \eqref{eq:h_mag}--\eqref{eq:control_problem_t} can be studied for different configurations
% for instance,  changing the number of dipoles, their positions and magnetic field directions. Although the computed
%force may vary depending on the configuration, the following mathematical analysis remains the same.
\end{remark}

\begin{figure}[ht]
\begin{center}
\includegraphics[height=0.25\linewidth]{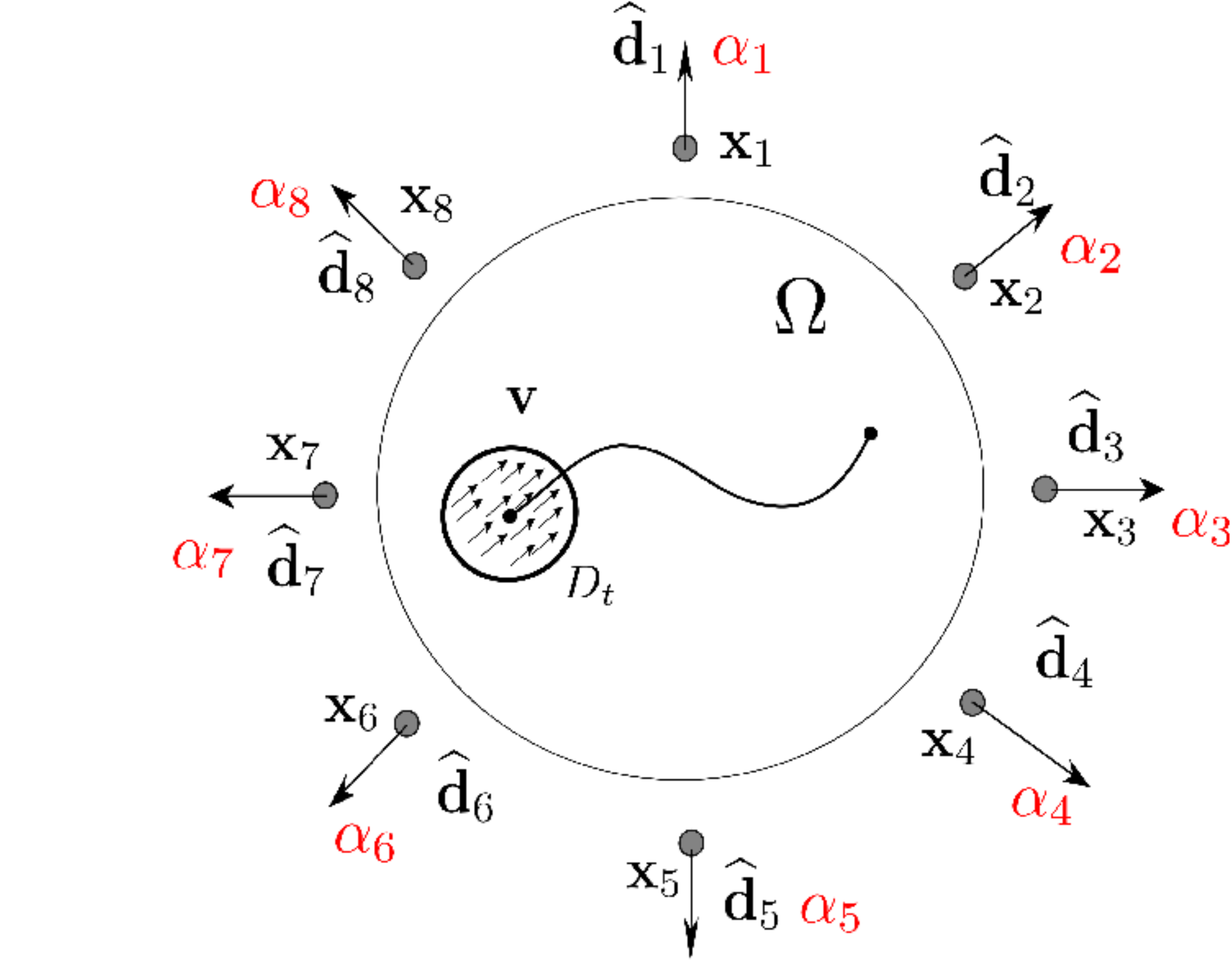}
\includegraphics[height=0.25\linewidth]{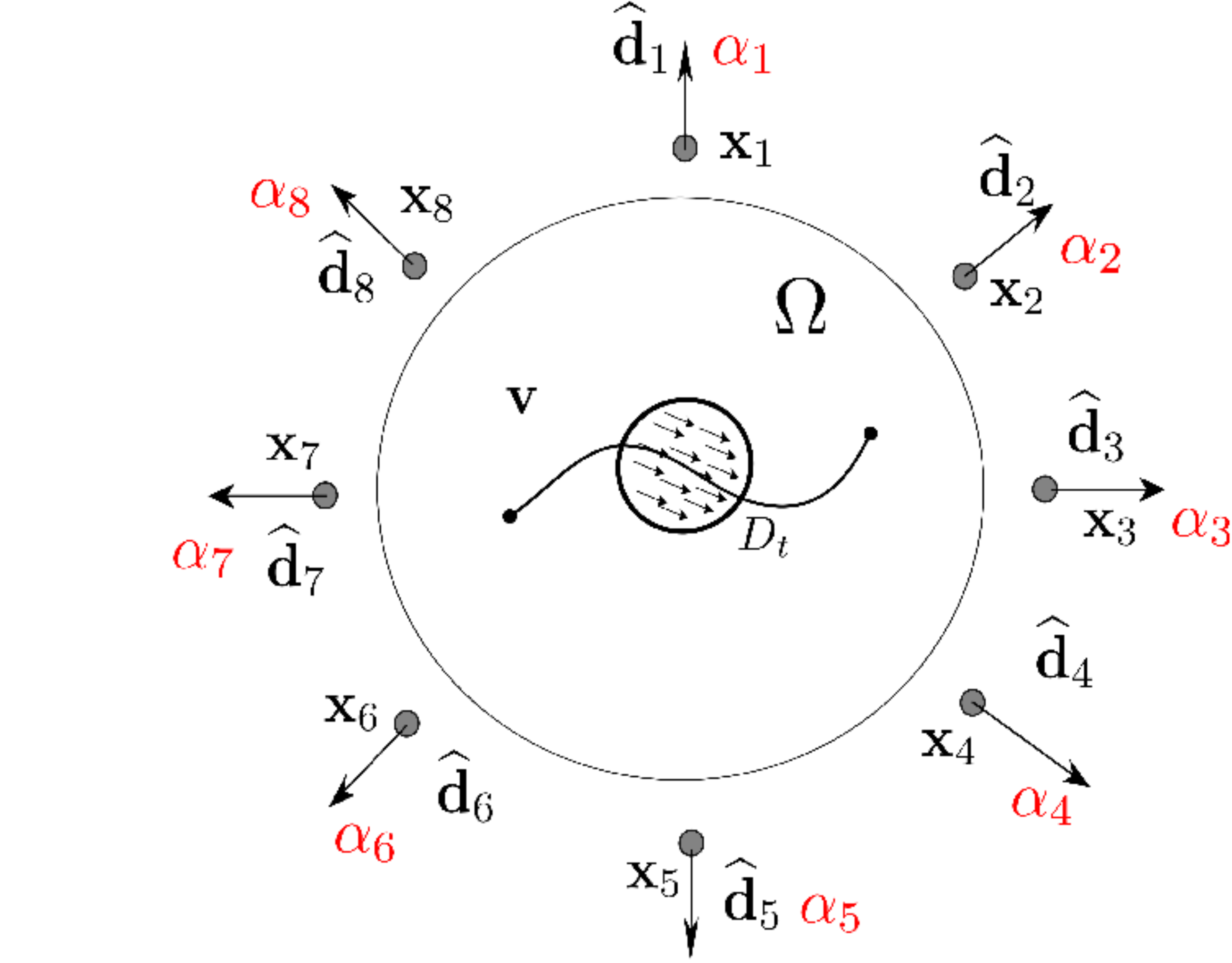}
\includegraphics[height=0.25\linewidth]{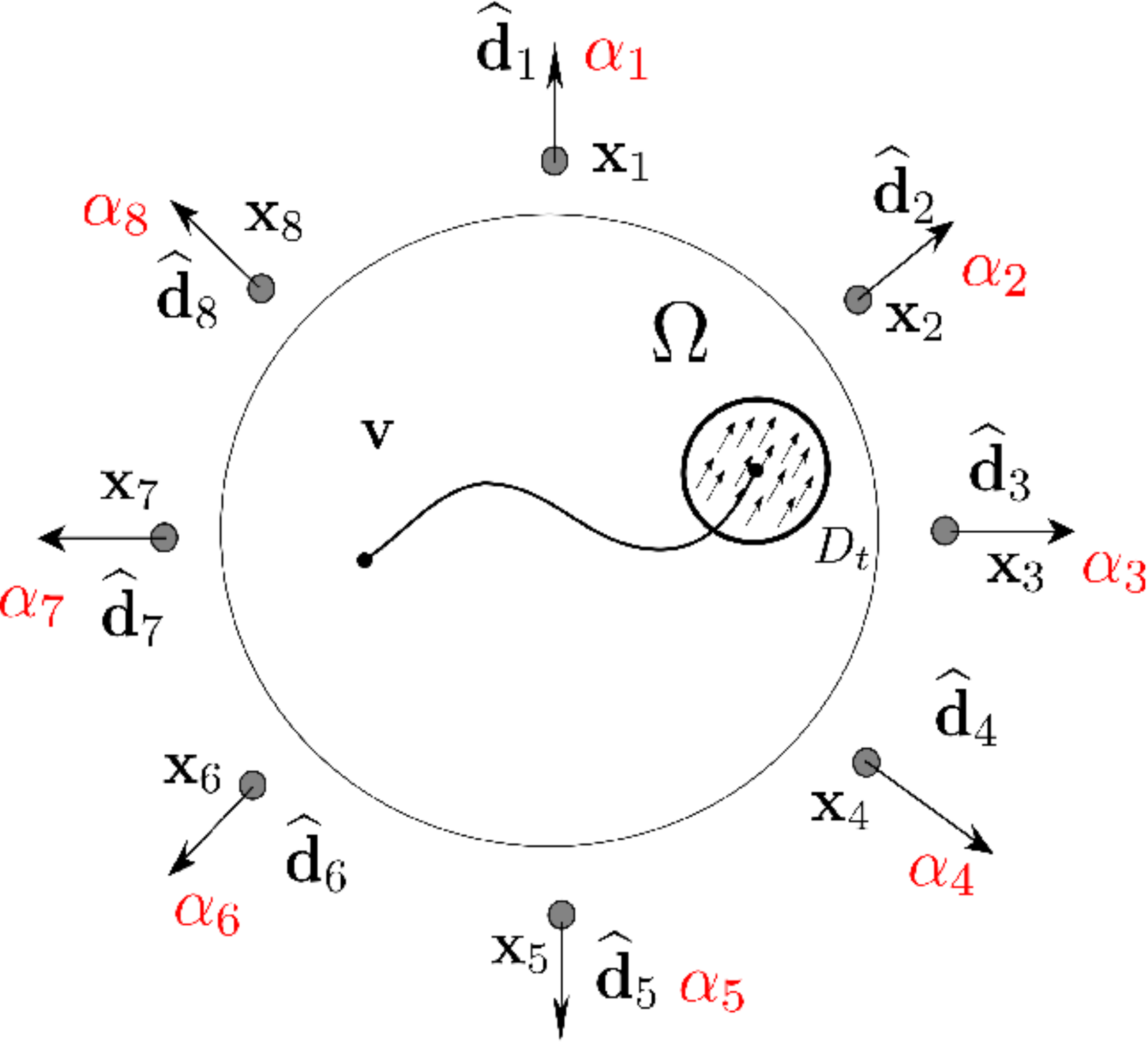}
\caption{
Configuration of $\nd=8$ dipoles surrounding a computational domain $\Omega\subset \R^2$ and moving domain $D_t$.
 The target vector field $\bvel$, represented with arrows, is shown only in $D_t$. Each dipole is characterized by
 its position $\bx_i$ (represented by a dot outside $\O$), direction $\bv{d}_i$
(represented by an arrow)
and the magnitude $\balpha_i$, for $i=1,\ldots,\nd$. The evolution of a circular moving domain $D_t$
 is shown for different times: $t = 0$ (left), $t\in (0,T)$ (center) and $t = T$ (right).
 The center of $D_t$ moves along a curve represented by the solid line inside $\O$.
}
\label{fig:domain}
\end{center}
\end{figure}
In view of \eqref{eq:h_mag}, we can rewrite $\bH$ as
\[
\bH(\bx,t)=\sum_{i=1}^{\nd}\alpha_i(t)\bH_i(\bx)=\mH(\bx)\balpha(t)
\]
where $\mH=\left(\bH_1\, \bH_2\,\ldots\,\bH_{\nd}\right)\in [C^{\infty}(\overline{\Omega})]^{\dd\times {\nd}}$. Thus
\begin{align*}
\nabla|\bH(\bx,t)|^2&=\nabla\left(\balpha(t)^\top\mH(\bx)^\top\mH(\bx)\balpha(t)\right) \\
&=\begin{pmatrix}\balpha(t)^\top\partial_{x_1}(\mH(\bx)^\top\mH(\bx))\balpha(t)\\
\vdots
\\
\balpha(t)^\top\partial_{x_\dd}(\mH(\bx)^\top\mH(\bx))\balpha(t)\end{pmatrix}
=\begin{pmatrix}\balpha^\top\bv{P}_1\balpha\\
\vdots
\\
\balpha^\top\bv{P}_\dd\balpha\end{pmatrix}
\end{align*}
with $\bv{P}_i=\partial_{x_i}(\mH^\top\mH)\in [C^{\infty}(\overline{\Omega})]^{{\nd}\times {\nd}}, i=1,\ldots,\dd.$
If $\bvel:=(\rmvel_{1},\ldots, \rmvel_{\dd})^\top$, then from the above equation we notice that
\eqref{eq:control_problem_t} reduces to
\begin{align}\label{eq:control_problem_2}
\min_{\balpha\in \calH_{ad} } \calJ(\balpha), \quad   \calJ(\balpha)=\dfrac{1}{2}\int_0^T
\left(\sum_{i=1}^\dd \|\balpha^\top\bv{P}_i\balpha-\rmvel_{i}\|^2_{\rL^2(D_t)}
+\lambda |\partialt\balpha|^2 \right) dt.
\end{align}
Next, we focus on the mathematical analysis of the minimization problem \eqref{eq:control_problem_2},
which is nonconvex.
We first embark on a journey to show the existence and local uniqueness of a solution to \eqref{eq:control_problem_2}.
For notational simplicity, from now on we denote by $\mathrm{V}$ both a Banach space $\mathrm{V}$ and
the Banach tensor product  $\mathrm{V}^{\nd}$.
\begin{theorem}[existence of minimizers]\label{thm:existence_optimal_control}
There exist at least one solution $\bbalpha$ to the minimization problem \eqref{eq:control_problem_2}.
\end{theorem}
\begin{proof}
We apply  the direct method of the calculus of variations. Given that $\calJ$
is bounded below by zero, we deduce that $j = \inf_{\balpha \in \calH_{ad}} \calJ(\balpha)$ is finite.
We can thus construct a minimizing sequence $\{ \balpha_n \}_{n\in \mathbb{N}}$ such that
\[
    j = \lim_{n\rightarrow \infty} \calJ(\balpha_n).
\]
As the sequence $\{ \balpha_n \}_{n\in \mathbb{N}}$ is uniformly bounded in $\calH_{ad}\subset \rH^1(0,T)$,
we can extract a (not relabeled) weakly
convergent subsequence $\{ \balpha_{n} \}_{n\in \N}$ such that
\begin{align}\label{eq:ex_1}
    \balpha_{n} \weakto \bar{\balpha} \quad \mbox{in } \rH^1(0,T),
                \qquad \bar{\balpha} \in \calH_{ad}.
\end{align}
Moreover, according to \cite[Theorem~9.16]{B11} we have
\begin{align}\label{eq:ex_2}
    \balpha_{n} \to \bar{\balpha} \quad \mbox{in }
C([0,T]).
\end{align}
%
%Here $\bbalpha$ is the candidate for a minimizer.
To show the optimality of $\bar{\balpha}$, we first consider \eqref{eq:ex_2} to get that
\begin{align*}
\int_0^T\|\balpha_n^\top\bv{P}_i\balpha_n-\rmvel_{i}\|^2_{\rL^2(D_t)}
\to \int_0^T\|\bar{\balpha}^\top\bv{P}_i\bar{\balpha}-\rmvel_{i}\|^2_{\rL^2(D_t)}  \qquad i=1,\ldots,\dd.
\end{align*}
This and the fact that the last term in $\calJ$
is weakly lower semicontinuous (see~\cite[Theorem 2.12]{T10}) yields
\begin{align*}
\min_{\balpha\in \calH_{ad}} \calJ(\balpha) = \liminf_{n\to\infty} \calJ(\balpha_n)
  \geq\calJ(\bar{\balpha}),
\end{align*}
which concludes the proof.
\end{proof}

We now state the first order optimality condition. This follows by standard arguments
(see \cite[Lemma~2.21]{T10}) in view of the fact that
$\calJ : \rH^1(0,T) \rightarrow \mathbb{R}$ is Fr\'echet differentiable.
\begin{lemma}[first order optimality condition]\label{var_ineq}
If $\bbalpha\in \calH_{ad}$  denotes an optimal control, given
by Theorem~\ref{thm:existence_optimal_control}, then the first order necessary optimality condition satisfied by
$\bbalpha$ is
\[
\calJ'(\bbalpha)\pair{\balpha - \bbalpha}\geq 0\qquad \forall \balpha\in \calH_{ad}
\]
where, for $\delta\balpha=(\balpha - \bbalpha)$ we have
\[
\calJ'(\bbalpha)\pair{\delta\balpha}=\int_0^T \left(\sum_{i=1}^\dd \int_{D_t}\left(\bbalpha^\top\bv{P}_i\bbalpha-\rmvel_{i}\right)
\left(2\bbalpha^\top\bv{P}_i\delta\balpha\right) d\bx +\lambda \partialt\bbalpha^\top \partialt\delta\balpha \right) dt.
\]
\end{lemma}

Since $\calJ$ is nonconvex, it is customary (cf. \cite[Section~4.10]{T10}) to assume that $\bbalpha$ is a nondegenerate local minimizer, namely that
there exists $\omega >0$ such that
\begin{equation}\label{eq:calJ_quad}
\calJ''(\bbalpha)\pair{\delta\balpha,\delta\balpha} \geq\omega |\delta\balpha|_{\rH^1(0,T)}^2 \qquad \forall \delta\balpha\in \Ccone(\bbalpha)
\end{equation}
where
\begin{multline*}
\calJ''(\bbalpha)\pair{\delta\balpha,\delta\balpha} \\
=2\int_0^T\int_{D_t}\sum_{i=1}^\dd\left( \left(\bbalpha^\top\bv{P}_i\bbalpha-\rmvel_{i}\right)(\delta\balpha)^\top\bv{P}_i(\delta\balpha)
+2\left( \bbalpha^\top\bv{P}_i \delta\balpha\right)^2\right)+\lambda\int_0^T|\partialt(\delta\balpha)|^2
\end{multline*}
and
$$
\Ccone(\balpha):=\left\{ \bv{h} \in \rH^1(0,T): \bv{h}(0) = \boldsymbol{0}, \ \balpha+ \zeta\bv{h}\in \calH_{ad},  \ \forall 0 \le \zeta \le 1 \right\}
$$
is the set of admissible variations of $\balpha \in \calH_{ad}$. This ensures local uniqueness as we show now.
\begin{lemma}[local uniqueness]\label{lemma:loc_unique}
If $\bbalpha\in \calH_{ad}$ solves \eqref{eq:control_problem_2} and satisfies \eqref{eq:calJ_quad},
then there exist positive constants $\nu$ and $\hat{C}$ such that for all
$\balpha\in\calH_{ad}$ with $\norm{\balpha-\bbalpha}{\rL^2(0,T)}\leq \nu$
\begin{equation}\label{eq:quad_J}
\calJ(\balpha)\geq\calJ(\bbalpha)+\hat{C}\norm{\balpha-\bbalpha}{\rH^1(0,T)}^2.
\end{equation}
\end{lemma}
\begin{proof}
It is easy to see that $\calJ : \rH^1(0,T) \rightarrow \mathbb{R}$ is twice continuously Fr\'echet differentiable.
Thus, from Taylor's theorem we have that there exists $\xi \in (0,1)$ such that for all $\delta\balpha := \balpha - \bbalpha \in \Ccone(\bbalpha)$
\begin{align*}
\calJ(\balpha)&=\calJ(\bbalpha)+\calJ'(\bbalpha)\pair{\delta\balpha}+\dfrac{1}{2}\calJ''(\bbalpha)\pair{\delta\balpha,\delta\balpha} \\
             & \quad +\dfrac{1}{2}\left(\calJ''(\bbalpha+\xi(\delta\balpha))-\calJ''(\bbalpha)\right)\pair{\delta\balpha,\delta\balpha}  .
\end{align*}
Next, we estimate the last term on the right-hand side of the above equation.
With this in mind, for  $\balpha_1:=\bbalpha+\xi(\delta\balpha)$ and $\balpha_2:=\bbalpha$
we consider
\begin{align*}
&\Big(\calJ''(\balpha_1)-\calJ''(\balpha_2)\Big)\pair{\delta\balpha,\delta\balpha}\\
&=2\int_0^T\int_{D_t}\sum_{i=1}^\dd\Bigg( \left(\balpha_1^\top\bv{P}_i\balpha_1-\balpha_2^\top\bv{P}_i\balpha_2\right)(\delta\balpha)^\top\bv{P}_i\delta\balpha \\
&\quad +2\left(\left(\balpha_1^\top\bv{P}_i\delta\balpha\right)^2-\left(\balpha_2^\top\bv{P}_i\delta\balpha\right)^2\right)\Bigg) \\
&=2\int_0^T\int_{D_t}\sum_{i=1}^\dd \left(\balpha_1^\top\bv{P}_i (\balpha_1-\balpha_2)+\balpha_2^\top\bv{P}_i(\balpha_1-\balpha_2)\right)(\delta\balpha)^\top\bv{P}_i\delta\balpha\\
&\quad+4\int_0^T\int_{D_t}\sum_{i=1}^\dd\left((\balpha_1-\balpha_2)^\top\bv{P}_i\delta\balpha\right)\left((\balpha_1+\balpha_2)^\top\bv{P}_i\delta\balpha\right) .
\end{align*}
A simple application of Cauchy-Schwarz inequality in conjunction with the
embedding $\rH^1(0,T) \subset C([0,T])$ leads to
\begin{align*}
\Big(\calJ''&(\balpha_1)-\calJ''(\balpha_2)\Big)\pair{\delta\balpha,\delta\balpha}\\
\leq& 2\norm{\balpha_1-\balpha_2}{\rL^2(0,T)}\norm{\delta\balpha}{C(0,T)}^2 |D_t|
\Big(\sum_{i=1}^{\dd}\norm{P_i}{\rL^{\infty}(\O)}^2\Big)
\left(\norm{\balpha_1}{\rL^2(0,T)}+\norm{\balpha_2}{\rL^2(0,T)}\right)\\
&+4\norm{\balpha_1-\balpha_2}{\rL^2(0,T)}\norm{\delta\balpha}{C(0,T)}^2 |D_t|
\Big(\sum_{i=1}^{\dd}\norm{P_i}{\rL^{\infty}(\O)}^2\Big)
\norm{\balpha_1+\balpha_2}{\rL^2(0,T)}\\
\leq& C_1\norm{\balpha_1-\balpha_2}{\rL^2(0,T)}\norm{\delta\balpha}{\rH^1(0,T)}^2
\end{align*}
where we used the fact that $D_t$ is a subset of $\Omega$ for all $t \in [0,T]$.
From the above inequality and \eqref{eq:calJ_quad} it follows that
\begin{align*}
\calJ(\balpha)\geq&\calJ(\bbalpha)+\calJ'(\bbalpha)\pair{\delta\balpha}+
\dfrac{\omega}{2}|\delta\balpha|_{\rH^1(0,T)}^2-\dfrac{C_1}{2}
\norm{\delta\balpha}{\rL^2(0,T)}\norm{\delta\balpha}{\rH^1(0,T)}^2.
\end{align*}
Then, the assertion follows from the first-order optimality condition
(cf. Lemma~\ref{var_ineq}) and the norm equivalence in $\Ccone(\bbalpha)$
provided that  $\norm{\delta\balpha}{\rL^2(0,T)}\leq \nu$ for
$\nu$ small enough.
\end{proof}
\begin{remark}[sufficient condition for \eqref{eq:calJ_quad}]\rm
The second derivative of $\calJ(\bbalpha)$ in the direction
 $\delta \balpha \in \Ccone(\bbalpha)$ can be estimated by using the
  Sobolev embedding $\rH^1(0,T)\subset C([0,T])$ and
  $\| \delta\balpha\|_{C([0,T])} \le T^{1/2} |\delta\balpha|_{\rH^1(0,T)}$ as follows
\begin{equation}
\begin{aligned}
\calJ''(\bbalpha)\pair{\delta\balpha,\delta\balpha}
%=2\int_0^T\int_{D_t}\sum_{i=1}^\dd\left( \left(\bbalpha^\top\bv{P}_i\bbalpha-\rmvel_{i}\right)(\delta\balpha)^\top\bv{P}_i(\delta\balpha)
%+2\left((\delta\balpha)^\top\bv{P}_i\bbalpha\right)^2\right)+\lambda\int_0^T|\partialt(\delta\balpha)|^2\\ \nonumber
&\geq -2|D_t|\sum_{i=1}^\dd\Big(\norm{\bbalpha^\top\bv{P}_i\bbalpha}{\rL^{\infty}(0,T;\rL^{\infty}(D_t))}\norm{\bv{P}_i}{\rL^{\infty}(\O)} \\
&\quad\quad\quad\quad\quad\quad\quad+2\norm{\bbalpha}{\rL^{\infty}(0,T)}^2\norm{\bv{P}_i}{\rL^{\infty}(\O)}^2\Big)
\norm{\delta\balpha}{\rL^{2}(0,T)}^2\\
&\quad-2T|\delta\balpha|_{\rH^1(0,T)}^2\sum_{i=1}^\dd\norm{\rmvel_{i}}{\rL^{1}(0,T;\rL^{1}(D_t))} \norm{\bv{P}_i}{\rL^{\infty}(\O)}+\lambda|\delta\balpha|_{\rH^1(0,T)}^2.
\end{aligned}
\end{equation}
The celebrated Poincar\'e inequality $\| \delta\balpha\|_{\rL^2(0,T)} \le T |\delta\balpha|_{\rH^1(0,T)}$ yields
\begin{align*}
\calJ''(\bbalpha)\pair{\delta\balpha,\delta\balpha}
&\geq \bigg(\lambda-2\sum_{i=1}^\dd\Big(3T^2|D_t|\norm{\bv{P}_i}{\rL^{\infty}(\O)}^2\norm{\bbalpha}{\rL^{\infty}(0,T)}^2 \\
&\quad\quad\quad+T|D_t|^{1/2}\norm{\rmvel_{i}}{\rL^{1}(0,T;\rL^{2}(D_t))} \norm{\bv{P}_i}{\rL^{\infty}(\O)}
\Big)\bigg)
|\delta\balpha|_{\rH^{1}(0,T)}^2.
\end{align*}
Clearly, a sufficient condition for \eqref{eq:calJ_quad} to hold is to consider  $\lambda$ sufficiently big which,
however,   is not reasonable for applications. Another condition for such an inequality to hold is take
  either $T$ or $D_t$  small enough.
Practically speaking, this appears in \cite{DKDK2003,HJBTF2003}.
\end{remark}

We have so far approximated a fixed vector field $\bvel$.
It is also meaningful to minimize the final time $T_F$. One way to realize this is by
treating $\bvel$ as an unknown. This is the topic of discussion for the next section.

%------------------------------------------------
%----  Unknown final time
%------------------------------------------------

\subsection{Problem 2: Minimizing the final time}\label{s:mft}
For obvious reasons, an important quantity to account for is
 the time it takes for $D_t$ to arrive at its final destination.
This is, for instance, the scenario of magnetic drug targeting.
This  can be  modeled by adding an unknown
 final time in the cost functional.

Notice that, in the previous section, magnetic force $\bvel$ and the moving domain $D_t$
are not necessarily related. However, if the final time is an unknown,
then a fixed vector field $\bvel$ in \eqref{eq:control_problem_2} is not a meaningful quantity.
We are now interested in the  ``force"   of $D_t$,
an unknown quantity, to be considered as a part of the minimization problem.
To properly handle the new unknown $T_F$, motivated by 
\cite{KS11}, we will reformulate the minimization problem in terms of the arc length.
Such a reformulation enables us to replace the variable final time $T_F$ by a fixed arc length $s_F$.
As an additional modeling approximation, we assume that there exist a curve $\calC$ (sufficiently smooth)
 with end points $(\bx_I)$ and $(\bx_F)$ that lies in $\O$ (see Figure~\ref{fig:curveC}).
We also assume that the displacement of $D_t$ is characterized by
 $\bx_C(t)$ which moves along  $\calC$ with velocity $\partialt \bx_C$, for instance, $D_t=\hD +\bx_C(t)$, where
 $\hD$ is a reference domain.
\begin{figure}[!h]
\begin{center}
\includegraphics[width=0.8\textwidth]{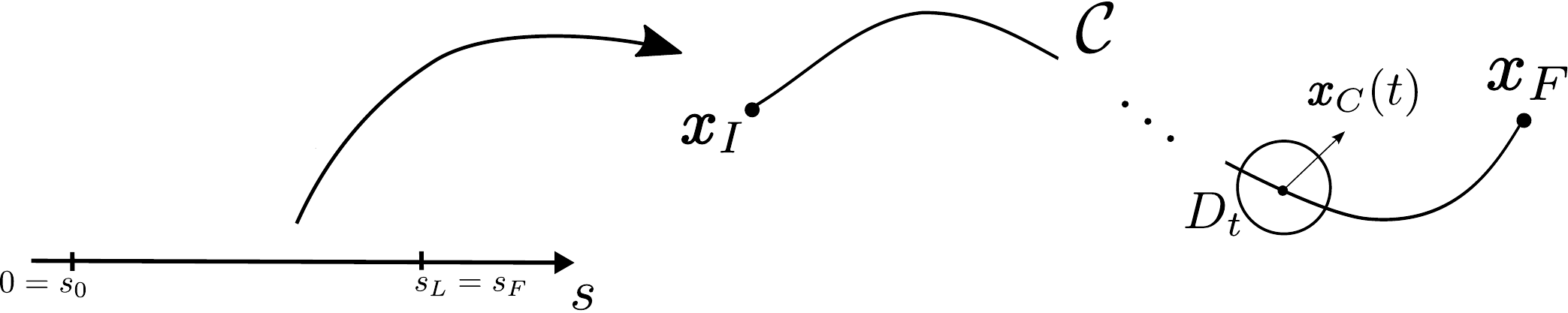}\qquad
\caption{Two dimensional curve $\calC$ and moving domain $D_t$ with barycenter $\bx_C$.
The domain $D_t$ travels along $\calC$ from an initial point $\bx_I$ to a final point $\bx_F$.}
\label{fig:curveC}
\end{center}
\end{figure}
In order to minimize the final time $T_F>0$, i.e, to maximize $\partialt \bx_C$
 we introduce $s \in [0,s_F]$ which represents the arc length parameter of the curve $\calC$.
 We assume that $\bx_C(t)$ starts from $\bx_I$ at $t=0$ and moves along $\calC$ with an arbitrary speed
$\theta(t)>0$, and eventually reaches $\bx_F$ at $t=T_F$.
With this in mind, we define a map $\sigma(\cdot):[0,T_F]\to [0,s_F]$ as
\begin{equation}\label{eq:s_t}
s=\sigma(t)=\int_0^{t}\theta(\tau)d\tau.
\end{equation}
Moreover, we assume that there exists a parametrization $\brho$ of $\calC$ depending
on the arc length $s \in [0,s_F]$ such that $\bx_C (\cdot) = \brho \circ \sigma (\cdot)$.
In the sequel we will assume that $\brho \in C^1[0,s_F]$.
From the above and by setting $\brho'=\partialt\brho$  we arrive at $
\partialt\bx_C(t)=\theta(t)\brho'(\sigma(t))$. We assume that the vector $\partialt\bx_C(t)$ has units similar to $\bF$. As in Problem~1, we consider a tracking type term:
\begin{align}\label{eq:control_problem_tf}
%\int_0^{T_F}\left(\dfrac{1}{2}\|\balpha(\sigma(t))^\top\bv{P}_i\balpha(\sigma(t))-\partialt\bx_C(t)\|^2_{\rL^2(D_{\sigma(t)})}
%+\dfrac{\lambda}{2}|\partialt\balpha(\sigma(t))|^2+ \beta\right)dt,\\
\int_0^{T_F}\left(\dfrac{1}{2} \sum_{i=1}^d\|\balpha(\sigma(t))^\top\bv{P}_i\balpha(\sigma(t))-\partialt\bx_{i,C}(t)\|^2_{\rL^2(D_{\sigma(t)})}
+ \beta\right)dt,
\end{align}
where $\beta>0$ is an additional penalty parameter and $\bx_{C}:=(\bx_{1,C},\ldots, \bx_{\dd,C})^\top$.
The last term on the right hand side of \eqref{eq:control_problem_tf}
is related with the minimization of the final time:
the larger the value of $\beta$, the smaller the value of $T_F$.
Notice that $\balpha:[0,s_F]\to \R$ and the moving domain are
defined in terms of arc length $s$ and are thus independent of $\theta$.
It is easy to see that the intent of \eqref{eq:control_problem_tf} is to
approximate $\partialt\bx_C$ and  minimize the final time $T_F$.
However, as we notice in the previous section, another important consideration of the
minimization problem is to control the  rate of change of the variables involved.
With this in mind, we apply the change of variables \eqref{eq:s_t} in \eqref{eq:control_problem_tf}
and propose the  minimization problem:
\begin{subequations}\label{eq:PJ2}
\begin{align}\label{eq:control_problem_sf}
&\min_{(\balpha,\theta)\in \calU_{ad}\times \calV_{ad} } \calJs(\balpha,\theta)
\end{align}
with
\begin{align}\label{eq:F_s}
\begin{aligned}
%\calJs(\balpha,\theta)
%:=\int_0^{s_F}&\left(\dfrac{1}{2\theta(s)}
%\sum_{i=1}^\dd\|\balpha(s)^\top\bv{P}_i\balpha(s)-\brho_i'(s)\theta(s)\|^2_{\rL^2(D_s)} \dfrac{\theta(s)\lambda}{2}|\partials\balpha(s)|^2+ \dfrac{\beta}{\theta(s)}
%+\dfrac{\eta}{2}|d_s\theta(s)|^2\right)ds\\
\calJs(\balpha,\theta)
:=\int_0^{s_F}\Bigg(\dfrac{1}{2\theta(s)}\sum_{i=1}^\dd\|\balpha(s)^\top\bv{P}_i\balpha(s)-\brho_i'(s)\theta(s)\|^2_{\rL^2(D_s)}\\
+\dfrac{\beta}{\theta(s)} +\dfrac{\lambda}{2}|\partials\balpha(s)|^2 +\dfrac{\eta}{2}|d_s\theta(s)|^2\Bigg)ds
\end{aligned}
\end{align}
\end{subequations}
where, for notational simplicity, we have used $\theta(\cdot):=\theta\circ\sigma^{-1}(\cdot)$.
In the above functional we have considered additional regularization terms:
for $\lambda, \eta > 0$, the last two term in  \eqref{eq:F_s} will
enforce a smooth evolution of the intensities and velocity.
%$d_{tt}\bx_C=\theta^2\brho''+\theta\partials\theta\brho'$
The admissible sets for $\balpha$ and $\theta$ are defined as
\begin{align*}
\calU_{ad} &:= \set{\balpha \in \rH^1(0,s_F) : \balpha(0)=\balpha_0
 \quad \mbox{and} \quad  \balpha_* \leq \balpha(s)\leq \balpha^*,\,\, \forall s\in [0,s_F]  } , \\
\calV_{ad} &:= \set{\theta \in \rH^1(0,s_F) : \theta(0)=\theta_0
\quad \mbox{and} \quad  0<\theta_* \leq \theta(s)\leq \theta^*,\,\, \forall s\in [0,s_F]  }.
\end{align*}
Notice that, for fixed upper and lower bounds on the dipole intensities ($\balpha^*$ and $\balpha_*$, respectively)
the maximum value of the magnetic force is fixed, thus, the maximum ``velocity" is fixed.
However, from  practical considerations we may impose an upper bound $\theta^*$ lower than a physically reachable
velocity obtained for a given pair $(\balpha_*,\balpha^*)$.
\begin{theorem}[existence of minimizers]\label{thm:existence_optimal_control_sf}
There exists at least one solution $(\bbalpha,\btheta)$ to the control problem \eqref{eq:control_problem_sf}.
\end{theorem}
\begin{proof}
The proof is similar to Theorem~\ref{thm:existence_optimal_control}. Since $\calJs$
is bounded below by zero there exists a minimizing sequence
$\{ (\balpha_n,\theta_n) \}_{n\in \mathbb{N}}$ such that
\[
 j:=\inf_{(\balpha,\theta)\in \calU_{ad}\times \calV_{ad} }
\calJs(\balpha,\theta)= \lim_{n\rightarrow \infty} \calJs(\balpha_n,\theta_n)
\]
which is uniformly bounded in $\calU_{ad}\times \calV_{ad}\subset [\rH^1(0,s_F)]^2$. Hence
 there exists $(\bbalpha,\btheta)\in [\rH^1(0,s_F)]^2$ such that
 $(\balpha_n,\theta_n)\weakto (\bbalpha,\btheta)$ in $\rH^1(0,s_F)$ and
 $ (\balpha_n,\theta_n) \to (\bbalpha,\btheta)$ in $C([0,s_F])$.
We thus arrive at
 \begin{align*}
 \dfrac{1}{\theta_n^{1/2}} \left(\balpha_n^\top\bv{P}_i\balpha_n-\brho_i'\theta_n\right)&\to \dfrac{1}{\btheta^{1/2}} \left(\bbalpha^\top\bv{P}_i\bbalpha-\brho_i'\btheta\right)\quad \mbox{in } \rL^2(0,s_F;\rL^2(D_{s}))\\
\dfrac{\beta}{\theta_n} &\to \dfrac{\beta}{\btheta} \quad \mbox{in } \rL^2(0,s_F).
 \end{align*}
This combined with the weak lower semicontinuity property of
the $\rH^1(0,s_F)$-semi-norm yields
 \[
 \min_{(\balpha,\theta)\in \calU_{ad}\times \calV_{ad} }
 \calJs(\balpha,\theta)\geq\liminf_{n\to\infty}\calJs(\balpha_n,\theta_n)\geq \calJs(\bbalpha,\btheta).
 \]
This completes the proof.
\end{proof}

Before showing a local uniqueness result of problem~\eqref{eq:control_problem_sf},
we state the first order optimality condition which follows from the Fr\'echet differentiability of
$\calJs : \rH^1(0,s_F)\times \rH^1(0,s_F) \rightarrow \mathbb{R}$.
\begin{lemma}[first order optimality condition]\label{var_ineq2}
If $(\bbalpha,\btheta)\in \calU_{ad}\times\calV_{ad}$  denotes an optimal control, given
by Theorem~\ref{thm:existence_optimal_control_sf}, then the first order necessary optimality condition satisfied by
$(\bbalpha,\btheta)$ is
\begin{equation}\label{eq:first_F}
\calJs'(\bbalpha,\btheta)\pair{(\delta\balpha,\delta\theta)}\geq 0\qquad \forall (\balpha,\theta)\in \calU_{ad}\times\calV_{ad}
\end{equation}
where $\delta\balpha=\balpha-\bbalpha$, $\delta\theta=\theta-\btheta$ and
\begin{align*}
\calJs'(\bbalpha,\btheta)&\pair{(\delta\balpha,\delta\theta)} \\
    =&\int_0^{s_F}\int_{D_s}\Bigg(\sum_{i=1}^\dd\left(\dfrac{\bbalpha^\top\bv{P}_i\bbalpha}{\btheta^{1/2}}
    -\brho_i'\btheta^{1/2}\right)\left(2\dfrac{\bbalpha^\top\bv{P}_i\delta\balpha}{\btheta^{1/2}}
    -\dfrac{\bbalpha^\top\bv{P}_i\bbalpha \delta\theta}{2\btheta^{3/2}}- \dfrac{\brho_i'\delta\theta}{2\btheta^{1/2}}\right)\\
    &-\dfrac{\beta\delta\theta}{\btheta^{2}} +\lambda\partials\bbalpha^\top\partials\delta\balpha
    +\eta\partials\btheta\partials\delta\theta\Bigg)d\bx ds.
\end{align*}
\end{lemma}
Given that the functional $\calJs$ is non-convex, like in Problem~1, the
 following second order sufficient condition guarantees that $(\bbalpha,\btheta)$ is a nondegenerate local minimizer:
  for all $(\balpha,\theta)\in \calU_{ad}\times\calV_{ad}$, there exists
 $\tilde{\omega}>0$ such that $\forall (\delta\balpha,\delta\theta) \in \mathcal{B}(\bbalpha,\btheta)$
\begin{equation}\label{eq:local_unig_0}
\calJs''(\bbalpha,\btheta)\pair{(\delta\balpha,\delta\theta), (\delta\balpha,\delta\theta)}\geq \tilde{\omega}\left(\norm{\delta\theta}{\rH^1(0,s_F)}^2
+\norm{\delta\balpha}{\rH^1(0,s_F)}^2\right) ,
\end{equation}
where
\begin{align*}
\mathcal{B}(\balpha,\theta):=\Big\{ (\bv{h},\widehat\theta) \in [\rH^1(0,T)]^2 &: (\bv{h}(0), \widehat\theta(0)) = (0,0), \\
&(\balpha, \theta)+ \zeta (\bv{h},\widehat\theta)\in \calU_{ad} \times \calV_{ad},  \ \forall  \zeta \in [0,1] \Big\} ,
\end{align*}
is the set of admissible variations of $(\balpha,\theta) \in \calU_{ad} \times \calV_{ad}$.
Consequently, the following lemma, yields the local uniqueness of problem~\eqref{eq:control_problem_sf}.
\begin{lemma}[local uniqueness]\label{lemma:loc_unique_s}
If $(\bbalpha,\btheta)\in \calU_{ad}\times\calV_{ad}$ satisfy the first and second-order
optimality conditions \eqref{eq:first_F} and \eqref{eq:local_unig_0}, then there exist $\tnu$ and $\hat{C}$ such that
\begin{equation}\label{eq:quad_Js}
\calJs(\balpha,\theta)\geq\calJs(\bbalpha,\btheta)+\hat{C}
\left(\norm{\balpha-\bbalpha}{\rH^1(0,s_F)}^2+\norm{\theta-\btheta}{\rH^1(0,s_F)}^2\right)
\qquad\forall (\balpha,\theta) \in\calU_{ad}\times \calV_{ad}
\end{equation}
 with $\norm{\balpha-\bbalpha}{\rH^1(0,s_F)}+\norm{\theta-\btheta}{\rH^1(0,s_F)}\leq \tnu.$
\end{lemma}
\begin{proof}The estimate \eqref{eq:quad_Js} follows by using the same techniques as in Lemma~\ref{lemma:loc_unique}, namely,
Taylor's expansion, first and second-order optimality conditions \eqref{eq:first_F} and \eqref{eq:local_unig_0},
and the following inequality: for $(\balpha,\theta)$ and $(\balpha_i,\theta_i)\in\calU_{ad}\times\calV_{ad}$, $i=1,2$,
\begin{multline*}
\left(\calJs''(\balpha_1,\theta_1)-\calJs''(\balpha_2,\theta_2)\right)\pair{(\delta\balpha,\delta\theta),(\delta\balpha,\delta\theta)}\\
\leq C\left(\norm{\balpha_1-\balpha_2}{\rH^1(0,s_F)}+\norm{\theta_1-\theta_2}{\rH^1(0,s_F)}\right)
\left(\norm{\delta\balpha}{\rH^1(0,s_F)}^2+\norm{\delta\theta}{\rH^1(0,s_F)}^2\right)
\end{multline*}
where $\delta\balpha = (\balpha-\bbalpha)$, $\delta\theta = (\theta-\btheta)$.
We omit the details for brevity.
\end{proof}
%
%\begin{remark}\rm
%If the curve $\calC$ is piecewise linear, namely, there exits $0=s_0<s_1<\ldots<s_{L-1}<s_L=s_F$ such
%that $\brho'(s)=\partial\brho^i=(\partial\rho^i_1,\partial\rho^i_2)\in \R^2$ for $s\in [s_{i-1},s_i)$, $i=1,\ldots,L$
%(see Figure~\ref{fig:Ccurve}), then we can rewrite \eqref{eq:control_problem_sf} as follows
%%
%\begin{align*}
%\calJs(\balpha,\theta)=&\sum_{j=1}^L\int_{s_{j-1}}^{s_{j}}
%\!\!\left(\dfrac{1}{2\theta}\sum_{i=1}^\dd\|\balpha^\top\bv{P}_i\balpha
%-\partial{\brho_i}^j\theta\|^2_{\rL^2(D_t(s))}+ \dfrac{\beta}{\theta}
%+\dfrac{\lambda}{2}|\partials\balpha|^2+
%\dfrac{\eta}{2}|\partials\theta|^2\right)ds.
%\end{align*}
%%
%\begin{figure}[!h]
%\begin{center}
%\includegraphics*[height=3cm]{imag/curvaC2b}
%\caption{Two dimensional piecewise linear curve $\calC$ defined on a partition
%$0=s_0,s_1,\ldots,s_{L-1}, s_L=s_F$ of the interval $(0,s_F)$. The constant
%slopes of $\calC$ are given by $\partial \brho^i$, $i=1,\ldots,L$.
%}
%\label{fig:Ccurve}
%\end{center}
%\end{figure}
%%
%\end{remark}

%------------------------------------------------
%----  Discrete problem
%------------------------------------------------

\section{Discretization}\label{s:disc}
This section is devoted to the numerical approximation of the
minimization problems \eqref{eq:control_problem_2} and \eqref{eq:control_problem_sf}.
For simplicity, we first introduce a parametrization of the moving domain ($D_t$ and $D_s$)
in terms of a fixed domain $\hD$.
After rewriting the minimization problem in terms of $\hD$,  we will introduce a discrete
formulation and prove $\rH^1$-weak convergence of its solution to a minimizer
of the continuous problem. Such a proof is motivated by $\Gamma$-convergence theory \cite{Maso93,Braides2013}.
%by applying $\Gamma$-convergence theory.
%An important consideration in this analysis is the smoothness of the continuous
%solutions.
Moreover, for problem
\eqref{eq:control_problem_sf} we prove the $\rH^1$-strong  convergence
 of the discrete problem to a local minimizer of \eqref{eq:control_problem_2}.
% Then, using Lemma~\ref{lemma:loc_unique}
%

\subsection{Problem 1: Fixed final time}\label{s:dfft}
Given that the domain $D_t$ changes with time, suitable assumptions are
needed in order to define a time discrete approximation of the minimization problem.
  With this in mind,
we define a reference domain
$\hD\subset \R^\dd$ and a map $\bX: [0,T] \times \overline{\hD}\to\overline{\O}$, such that for all $t\in [0,T]$
\begin{align*}
\bX(t,\cdot):&\overline{\hD}\to \overline{D}_t\\
      &\hbx\to \bx = \bX(t,\hbx),
\end{align*}
is a one-to-one correspondence which satisfies $\bX(t,\hD)=D_t$. For simplicity, we assume
\[
\bX(t,\hbx)=\bvarphi(t)+\psi(t)\hbx,
\]
where
\begin{equation}\label{eq:psi_reg}
\bvarphi:[0,T]\to \R^\dd,\,\quad \psi:[0,T]\to (0,+\infty),\quad \bvarphi, \psi \in \rH^1(0,T)
\end{equation}
are functions such that
$\bvarphi(0)=\boldsymbol{0}$ and $\psi(0)=1$, namely $\hD=D_{0}$.
In case $\psi(t)=\widehat{\psi}\in \R^+$ for all $t\in [0,T]$, $\bvarphi$
can be viewed as a parameterization of a desired path that the scaled domain $\widehat{\psi}\widehat{D}$
traverses from an initial position to a final position.
 However, in applications such as magnetic drug targeting,  a general function
 $\psi(t)$ may be needed to control drug spreading.
 Therefore, for $i=1,\dots, \dd$
\begin{multline*}
\int_0^T\int_{D_t}\left(\balpha(t)^\top\bv{P}_i(\bx)\balpha(t)-\rmvel_{i}(t,\bx)\right)^2 d\bx\, dt\\
=
\int_0^T\int_{\hD}\left(\balpha(t)^\top\bv{P}_i(\bX(t,\hbx))\psi(t)^{\dd/2}\balpha(t)-\rmvel_{i}(t,\bX(t,\hbx))\psi(t)^{\dd/2}\right)^2 d\hbx\, dt,
\end{multline*}
because $\det(\nabla_{\hbx}\bX(t,\hbx))=\psi(t)^\dd$.
Then, we rewrite  $\calJ$ as
\begin{equation}\label{eq:calJ_D}
\calJ(\balpha)=\dfrac{1}{2}\sum_{i=1}^\dd\int_0^T\|\balpha^\top\hbP_i\balpha-\hbvel_i\|^2_{\rL^2(\hD)}
+\dfrac{\lambda}{2}\int_0^T|\partialt\balpha|^2=\calJ^1(\balpha)+\calJ^2(\balpha)
\end{equation}
with $\hbP_i(t,\hbx):=\bv{P}_i(\bX(t,\hbx))\psi(t)^{\dd/2}$
and $\hbvel_{i}(t,\hbx)=\rmvel_{i}(t,\bX(t,\hbx))\psi(t)^{\dd/2}$, $i=1,\dots, \dd$.

Next, we introduce a time discretization of the Problem 1 upon using $\calJ$ as defined in \eqref{eq:calJ_D}.
Let us fix $N\in \N$ and let $\Dt := T/N$ be the time step.
Now, for $n = 1,\ldots, N$, we define $t^n := n\Dt$,
$\hbP^n_i=\hbP_i(t^n)$ and $\hbvel^n_{i}$ to be
\begin{equation}\label{eq:discr_v}
\hbvel^n_{i}(\cdot)=\dfrac{1}{\tau}\int_{t^{n-1}}^{t^n}\hbvel_{i}(t,\cdot)dt,\qquad i=1,\ldots,\dd,
\end{equation}
which in turn allows us  to incorporate a general $\bvel$.
Then we consider a time discrete version of \eqref{eq:control_problem_2}: given
the initial condition $\balpha_0=:\bbalpha_\tau(0)$ (cf.~\eqref{eq:ad_set}),
find $\bbalpha_\tau \in \calH^\tau_{ad}$ solving
\begin{align}
\label{eq:control_problem_t_disc}
\bbalpha_\tau=\argmin_{\balpha_\tau\in\calH^\tau_{ad}} \calJ_\tau(\balpha_\tau),
\quad \calJ_\tau(\balpha_\tau)=\calJ^1_\tau(\balpha_\tau)+\calJ^2_\tau(\balpha_\tau),
\end{align}
where
\begin{align*}
\calJ^1_\tau(\balpha_\tau)=
\Dt\sum_{n=1}^N\dfrac{1}{2}\sum_{i=1}^\dd\|(\balpha^n_\tau)^\top
\hbP^n_i\balpha^n_\tau-\hbvel^n_{i}\|^2_{\rL^2(\hD)} ,
\quad
\calJ^2_\tau(\balpha_\tau) =
\Dt\sum_{n=1}^N\dfrac{\lambda}{2\Dt^2}|\balpha^n_\tau-\balpha^{n-1}_\tau|^2\\
\end{align*}
and
$$
\calH^\tau_{ad} := \{\balpha_\tau \in \rH^1(0,T) : \balpha_\tau|_{[t^{n-1},t^n]}\in \mathbb{P}^1, \ n = 1, \dots, N\} \cap \calH_{ad}.
$$
Hereafter $\mathbb{P}^1$ is the space of polynomials of degree at most 1.
Moreover, by applying the same arguments of Theorem~\ref{thm:existence_optimal_control}, it follows that
there exists $\bbalpha_\tau\in \calH^\tau_{ad}$ solution to problem \eqref{eq:control_problem_t_disc}.

%------------------------------------------------
%----  Convergence
%------------------------------------------------
%
%\subsection{Convergence}\label{sec:convergence}
%
  Notice  that $\calJ$ is a non-convex functional, so convergence of discrete
minimizers  to a continuous one is not immediate.
The convergence of the  discrete scheme is the content of the next result which is motivated by $\Gamma$-convergence.
\begin{theorem}[Problem 1: convergence to global minimizers]\label{thm:Gconv}
The family of global minimizers $\left\{\bbalpha_{\tau}\right\}_{\tau>0}$ to
\eqref{eq:control_problem_t_disc} is uniformly bounded in  $\rH^1(0,T)$
and it contains a subsequence that converges weakly to
$\bbalpha$ in $\rH^1(0,T)$, a global solution to the minimization problem \eqref{eq:control_problem_2},
and $\lim_{\tau\to 0}\calJ_\tau(\bbalpha_\tau)=\calJ(\bbalpha)$.
\end{theorem}
\begin{proof}
We proceed is several steps.
\begin{enumerate}[1.-]
\item \textit{Boundedness of $\left\{\bbalpha_\tau\right\}_{\tau>0}$ in $\rH^1(0,T)$:}
 This follows immediately from the fact that $\bbalpha_{\tau}$ minimizes
$\calJ_{\tau}$ and $\lambda>0$: given that the constant function  $\balpha_0(t)=\balpha_0$  belongs to $\calH^\tau_{ad}$, we have
\[
\calJ_{\tau}(\bbalpha_{\tau})\leq \calJ_{\tau}(\balpha_0)\leq C\left(|\balpha_0|^4+\|\bvel\|^2_{\rL^2(0,T;\rL^2(\O))}\right).
\]
This implies the existence of a (not relabeled) weakly convergent subsequence such that
$\bbalpha_\tau \weakto \bbalpha$  in $\rH^1(0,T)$ and $\bbalpha \in \calH_{ad}$.
It remains to prove that $\bbalpha$ solves \eqref{eq:control_problem_2} and
$\lim_{\tau\to 0}\calJ_\tau(\bbalpha_\tau)=\calJ(\bbalpha)$.
\item \textit{Lower bound inequality}:
We show that
\begin{equation}\label{eq:Lower_J2}
\calJ(\balpha)\leq \underset{\tau\to 0}{\lim \inf}\calJ_\tau(\balpha_\tau).
\end{equation}
for all $\left\{\balpha_\tau\right\}_{\tau>0}\subset \calH^\tau_{ad}$
converging to $\balpha$ weakly in $\rH^1(0,T)$. Consequently
 $\balpha_\tau\to\balpha$ in strongly in $\rL^2(0,T)$ for a
subsequence (not relabeled). If $\overline{\Pi}_\tau\balpha_\tau$
 is the piecewise constant interpolant of $\balpha_\tau$, then
\begin{equation*}
    \| \overline{\Pi}_\tau\balpha_\tau - \balpha \|_{\rL^2(0,T)}
    \le \| \overline{\Pi}_\tau\balpha_\tau - \balpha_\tau \|_{\rL^2(0,T)}
      + \| \balpha_\tau - \balpha \|_{\rL^2(0,T)} \rightarrow 0,
\end{equation*}
because $\{ \partialt\balpha_\tau \}_{\tau>0}$ being bounded implies
$\| \overline{\Pi}_\tau\balpha_\tau - \balpha_\tau \|_{\rL^2(0,T)} \rightarrow 0$.
In view of the smoothness of $\hbP_{i}$, if $\overline{\hbP}_i$ denotes the piecewise
constant interpolation in time then  $\overline{\hbP}_i \rightarrow \hbP_i$ in $\rL^2(0,T; \rL^2(\widehat{D}))$,
 $i = 1, \dots, d$. Collecting these results and using that $\| \balpha_\tau\|_{\rL^\infty(0,T)}\leq \balpha^*$,
 we readily obtain
\begin{equation*}
\overline{\Pi}_\tau\balpha_\tau^\top\overline{\hbP}_i\overline{\Pi}_\tau\balpha_\tau\to\balpha^\top
\hbP_i\balpha\quad \mbox{in  } \rL^2(0,T; \rL^2(\widehat{D})) ,
 \end{equation*}
which in conjunction with the regularity of $\bvel$ leads to
%Additionally, owing to the regularity of $\bvel$ we arrive at
%
\begin{equation}\label{eq:Low_bnd_1}
 \calJ^1_\tau(\balpha_\tau)\to\calJ^1(\balpha).
\end{equation}
On the other hand, given that $\partialt\balpha_\tau$ converges weakly to
 $\partialt\balpha$ in $\rL^2(0,T)$,
from the weak lower semi-continuity of the semi-norm it follows that
\begin{equation}\label{eq:Low_bnd_2}
\calJ^2(\balpha)\leq \underset{\tau\to 0}{\lim \inf}\calJ^2_\tau(\balpha_\tau).
\end{equation}
From \eqref{eq:Low_bnd_1} and \eqref{eq:Low_bnd_2} we conclude
\begin{equation*}
\calJ(\balpha)\leq \underset{\tau\to 0}{\lim \inf}\calJ^1_\tau(\balpha_\tau)+\underset{\tau\to 0}{\lim \inf}\calJ^2_\tau(\balpha_\tau)\leq \underset{\tau\to 0}{\lim \inf}\calJ_\tau(\balpha_\tau).
\end{equation*}
\item \textit{Existence of a recovery sequence}:
Let $\balpha\in\calH_{ad}$ be given. Then, the piecewise linear Lagrange interpolant $\Pi_{\Dt}\balpha$ of $\balpha$,
belongs to $\calH_{ad}^\tau$. Since $\partialt(\Pi_{\Dt}\balpha) \rightarrow \partialt\balpha$ in $\rL^2(0,T)$
and $\Pi_\tau\balpha^\top_{\Dt}\overline{\hbP}_i\Pi_{\Dt}\balpha_\tau\to\balpha^\top\hbP_i\balpha$
in $\rL^2(0,T;\rL^2(\hD))$ because $\|\Pi_{\Dt}\balpha\|_{\rL^\infty(0,T)}\leq \balpha^*$, we obtain
\[
 \underset{\tau\to 0}{\limsup}\calJ_\tau(\Pi_{\Dt}\balpha)\le \underset{\tau\to 0}{\limsup}\calJ^1_\tau(\Pi_{\Dt}\balpha) + \underset{\tau\to 0}{\limsup}\calJ^2_\tau(\Pi_{\Dt}\balpha)\leq\calJ(\balpha).
\]
\item \textit{$\bbalpha$ is a global minimizer for Problem 1:} We need to show
\begin{equation}\label{balpha_min}
 \calJ(\bv{v})\geq \calJ(\bbalpha) \qquad \forall \bv{v}\in \calH_{ad} .
\end{equation}
From step 3 there exists $\{ \bv{v}_\tau \}_{\tau>0}$ such that
$\bv{v}_\tau \rightharpoonup \bv{v}$ in $\rH^1(0,T)$ and
\[
\calJ(\bv{v})
\ge\underset{\tau\to 0}{\lim \sup}\calJ_\tau(\bv{v}_\tau)
\geq \underset{\tau\to 0}{\liminf} \calJ_\tau(\bv{v}_\tau)\geq \underset{\tau\to 0}{\liminf} \calJ_\tau(\bbalpha_\tau)\geq \calJ(\bbalpha)
\]
where we have used that $\bbalpha_\tau$ is a global minimizer for
$\calJ_\tau$ together with \eqref{eq:Lower_J2}.
\item  \textit{Convergence:} Since $\bbalpha_\tau$ is a global minimizer we deduce
$\calJ_\tau(\bbalpha_\tau)\leq \calJ_\tau(\Pi_\tau\bbalpha)$, whence
applying first step 3 and next step 2 we see that
\[
\underset{\tau\to 0}{\lim\sup}\calJ_\tau(\bbalpha_\tau)\leq\underset{\tau\to 0}{\lim\sup}\calJ_\tau(\Pi_\tau\bbalpha)\leq \calJ(\bbalpha)\leq \underset{\tau\to 0}{\liminf} \calJ_\tau(\bbalpha_\tau).
\]
This implies $\lim_{\tau\to 0}\calJ_\tau(\bbalpha_\tau)=\calJ(\bbalpha)$. In addition
 $\left\{\bbalpha_{\tau}\right\}_{\tau>0}$ converges $\rL^2$-strongly and
 $\rH^1$-weakly to $\bbalpha$, a global minimizer of \eqref{eq:control_problem_2}.
\end{enumerate}
This concludes the proof.
\end{proof}

\begin{remark}[quadrature]\rm
In general, numerical integration has to be used to
compute the space integrals in \eqref{eq:control_problem_t_disc}, leading to another
approximation error. For simplicity, we have assumed that we can evaluate the integrals exactly.
\end{remark}

Given that $\calJ$ is non-convex, the solution to \eqref{eq:control_problem_2} might not be unique.
However, the minimizers are locally unique if we assume the second-order optimality condition
\eqref{eq:calJ_quad} (cf.~Lemma~\ref{lemma:loc_unique}).
We also notice that the previous convergence result, \thmref{thm:Gconv}, does not
guarantee the convergence to one of these local minimizers but to a global one.
To rectify this we  follow \cite{CT02a}. Let us first assume that $\bbalpha$ is a
local unique solution to \eqref{eq:control_problem_t}.
For a fixed $\eps>0$, we  construct a family $\left\{\balpha^\eps_\tau\right\}_{\tau>0}$
upon solving the minimization problem
\begin{equation}\label{eq:control_problem_t_disc_eps}
\balpha^\eps_\tau=\argmin_{\balpha_\tau\in\calH^{\tau,\eps}_{ad}} \calJ_\tau(\balpha_\tau)
\end{equation}
in an $\eps-$neighborhood of  $\bbalpha$ denoted by $\calH_{ad}^{\tau,\eps}=\left\{\balpha_\tau\in \calH_{ad}^{\tau}: \norm{\Pi_{\Dt}\bbalpha-\balpha_\tau}{\rL^2(0,T)}\leq \eps\right\}$.

Next, we prove that
$\{ \balpha^\eps_\Dt \}_{\Dt > 0}$ forms the local solution to the problem \eqref{eq:control_problem_t_disc}.
First, we recall the definition of $\nu$ from \eqref{eq:quad_J} and define
$\bar{\eps}:=\dfrac{\nu}{2}$. It follows from the definition of $\Pi_\Dt$
that there exist a $\Dt_0 > 0$  such that for every $\tau \le \Dt_0$ we have
$\norm{\Pi_{\Dt}\bbalpha-\bbalpha}{\rL^2(0,T)} \le \dfrac{\nu}{2}$. Then,
for a given $\balpha \in \calH_{ad}^{\tau,\eps}$, $\eps\le \bar{\eps}$,
and $\tau \le  \tau_0$ we obtain
\[
\norm{\balpha-\bbalpha}{\rL^2(0,T)}\leq \norm{\Pi_{\Dt}\bbalpha - \balpha}{\rL^2(0,T)}
 + \norm{\Pi_{\Dt}\bbalpha-\bbalpha}{\rL^2(0,T)} \leq \nu.
\]
Finally, in view of \eqref{eq:quad_J} we arrive at
\begin{equation}\label{eq:calJ_quadUeps}
\calJ(\balpha)\geq\calJ(\bbalpha)+\hat{C}\norm{\balpha-\bbalpha}{\rH^1(0,T)}^2\qquad\forall \balpha\in \calH_{ad}^{\tau,\eps}.
\end{equation}
Such a quadratic behavior enables us to prove the following convergence
result for the family $\left\{\balpha^\eps_\tau\right\}_{\tau>0}$ solving \eqref{eq:control_problem_t_disc_eps}.
\begin{lemma}[Problem 1: convergence to local minimizers] \label{lem:aux_conv}
Let $\bbalpha \in \calH_{ad}$ be a local unique minimizer of \eqref{eq:PJ1}.
If $\eps\leq \bar{\eps}$ and $\{ \balpha_\tau^\eps \}_{\Dt > 0}$ solves
\eqref{eq:control_problem_t_disc_eps}, then there exists a $\Dt_0 > 0$ such that for all
$\Dt \leq \Dt_0$, $\calJ_{\tau}(\balpha_\tau^\eps) \rightarrow \calJ(\bbalpha)$
and $\norm{\balpha_\tau^\eps-\bbalpha}{\rH^1(0,T)}\to 0$ as $\Dt \rightarrow 0$.
Moreover, $\balpha^\eps_\tau$ is a local solution of \eqref{eq:control_problem_t_disc},
i.e
\begin{equation}\label{eq:local_ineq}
\calJ_{\tau}(\balpha_\tau^\eps) \leq \calJ_{\tau}(\balpha_\tau)
\end{equation}
for all $\balpha_\tau\in \calH_{ad}^\tau$ such that $\norm{\balpha_\tau^\eps-\balpha_\tau}{\rL^2(0,T)}\leq \eps/2$.
\end{lemma}
\begin{proof} Notice that, because of the regularity assumption \eqref{eq:psi_reg}
on $\bvarphi$ and $\psi$, it is straightforward to prove that
\begin{gather}
\begin{aligned}
|\calJ_{\tau}(\balpha_\tau)-\calJ(\balpha_\tau)|&\to 0\qquad \forall \balpha_\tau \in \calH_{ad}^\tau,\label{eq:cv_1} \\
\calJ(\Pi_{\Dt}\balpha)&\to\calJ(\balpha)\qquad \forall \balpha \in \calH_{ad} . %\label{eq:cv_2}.
\end{aligned}
\end{gather}
On the other hand, from \eqref{eq:calJ_quadUeps} we obtain
\begin{equation}\label{eq:ineq_1}
\begin{aligned}
\calJ_{\tau}(\balpha_\tau^\eps)
&\geq\calJ(\bbalpha)+ (\calJ_{\tau}(\balpha_\tau^\eps)-\calJ(\balpha_\tau^\eps))+\hat{C}\norm{\balpha_\tau^\eps-\bbalpha}{\rH^1(0,T)}^2 \\
&\geq \calJ(\bbalpha)- |\calJ_{\tau}(\balpha_\tau^\eps)-\calJ(\balpha_\tau^\eps)|.
\end{aligned}
\end{equation}
Moreover, given that $\Pi_{\Dt}\bbalpha\in\calH_{ad}^{\tau,\eps}$,
from the optimality of $\balpha_\tau^\eps$  it follows that
\begin{align}\label{eq:ineq_2}
\calJ_\tau(\balpha_\tau^\eps)\leq&\calJ_\Dt(\Pi_{\Dt}\bbalpha)
= \calJ(\bbalpha) + \left(\calJ_\tau(\Pi_{\Dt}\bbalpha)-\calJ(\Pi_{\Dt}\bbalpha)\right)
+ \left(\calJ(\Pi_{\Dt}\bbalpha)-\calJ(\bbalpha)\right).
\end{align}
Then, from \eqref{eq:cv_1}-\eqref{eq:ineq_2} we obtain
\begin{equation}\label{eq:ineq_2b}
\calJ_\tau(\balpha_\tau^\eps)\to\calJ(\bbalpha) ,
\end{equation}
which is the first assertion. The second assertion $\norm{\balpha_\tau^\eps-\bbalpha}{\rH^1(0,T)}^2 \rightarrow 0$ is now a
trivial consequence of the middle inequality in \eqref{eq:ineq_1}
together with \eqref{eq:cv_1} and \eqref{eq:ineq_2b}.
It remains to show \eqref{eq:local_ineq}.
Let $\balpha_\tau\in \calH_{ad}^\tau$ satisfy $\norm{\balpha_\tau-\balpha_\tau^\eps}{\rL^2(0,T)}\leq \dfrac{\eps}{2}$.
Then for  $\tau$ small enough we arrive at
\[
\norm{\balpha_\tau-\Pi_{\Dt}\bbalpha}{\rL^2(0,T)}\leq\norm{\balpha_\tau-\balpha_\tau^\eps}{\rL^2(0,T)}+
\norm{\bbalpha-\balpha_\tau^\eps}{\rL^2(0,T)}+\norm{\Pi_{\Dt}\bbalpha-\bbalpha}{\rL^2(0,T)}\leq\eps .
\]
Thus, $\balpha_\tau$ belongs to $\calH_{ad}^{\tau,\eps}$  and
 \eqref{eq:local_ineq} follows from \eqref{eq:control_problem_t_disc_eps}.
\end{proof}

%------------------------------------------------
%----  Approximation
%------------------------------------------------

\subsection{Problem 2: Minimizing the final time}
Like in Section~\ref{s:dfft}, in order to handle the first term in \eqref{eq:F_s},
we consider a parametrization of the moving domain $D_s$   but now defined in terms of the arc length $s$, namely
\[
\bX(s,\hbx)=\brho(s)+\widetilde{\psi}(s)\hbx \qquad \forall \hbx\in \hD, \quad 0\leq s\leq s_F,
\]
where $\widetilde{\psi}$ belongs to $\rH^1(0,s_F)$ and $\brho\in C^1[0,s_F]$. Then, the functional $\calJs$ in \eqref{eq:F_s} reduces to
\begin{align*}
\calJs(&\balpha,\theta)=\calJs^1(\balpha,\theta)+\calJs^2(\theta)+\calJs^3(\alpha)
+\calJs^4(\theta)
\end{align*}
with
\begin{align*}
\calJs^1(\balpha,\theta)&=\int_0^{s_F}\dfrac{1}{2\theta(s)}\sum_{i=1}^\dd
\|\balpha(s)^\top\wbP_i\balpha(s)-\brho_i'(s)\wtheta(s)\|^2_{\rL^2(\hD)} ds, \\
 \calJs^2(\theta)&=\int_0^{s_F}\dfrac{\beta}{\theta(s)} ds , \quad
 \calJs^3(\balpha)=\int_0^{s_F}\dfrac{\lambda}{2}|\partials\balpha(s)|^2 ds, \quad
\calJs^4(\theta)=\int_0^{s_F}\dfrac{\eta}{2}|\partials\theta(s)|^2 ds
\end{align*}
and $\wbP_i(\hbx,s)=\bv{P}_i(\bX(s,\hbx))\widetilde{\psi}^{\dd/2}(s)$, $i=1,\ldots, \dd$,
$\wtheta(s) = \theta(s) \widetilde{\psi}^{\dd/2}(s)$.
Notice that now the $\rL^2$-norm in $\calJs^1$ is computed on $\hD$ instead of $D_s$.
Then we introduce a space discretization of \eqref{eq:F_s} by taking into
account the previous definition of  $\calJs$. Let us fix $M\in \N$ and set $\Ds := s_F/M$.
Now, for $m = 1,\ldots, M$, we define $s^m := m\Ds$, $\balpha^m:=\balpha(s^m)$,
$\theta^m:=\theta(s^m)$, $\widetilde\theta^m:=\widetilde\theta(s^m)$, and $\wbP^m_i=\wbP_i(s^m),$ $i=1,\ldots, \dd$.
Then we consider the following discrete problem: given  an initial condition
$(\balpha_0,\theta_0)=:(\balpha_\Ds(0),\theta_\Ds(0))$
find a solution $(\balpha_\Ds,\theta_\Ds)\in \calU^\Ds_{ad}\times\calV^\Ds_{ad}$ to
\begin{align}
\label{eq:control_problem_s_disc}
\min_{(\balpha_\Ds,\theta_\Ds)\in\calU^\Ds_{ad}\times\calV^\Ds_{ad}}
\calJs_{\Ds}(\balpha_\Ds,\theta_\Ds),
\quad  \calJs_{\Ds}(\balpha_\Ds,\theta_\Ds)=\calJs_\Ds^1(\balpha_\Ds,\theta_\Ds)
+\calJs_\Ds^2(\theta_\Ds)+\calJs_\Ds^3(\balpha_\Ds)+\calJs_\Ds^4(\theta_\Ds)
%&\calJs_{\Ds}(\balpha_\Ds,\theta_\Ds)
%=\Ds\sum_{m=1}^M\bigg\{\dfrac{1}{2\theta^m_\Ds}
%\sum_{i=1}^\dd\|(\balpha^m_\Ds)^\top\wbP^m_i\balpha^m_\Ds-\brho'_i(s^m)\theta^m_\Ds\|^2_{\rL^2(\hD)}
%\\
%+&\dfrac{\lambda\theta^m_\Ds}{2\Ds^2}|\balpha^m_\Ds-\balpha^{m-1}_\Ds|^2+\dfrac{\beta}{\theta^m_\Ds}
%+\dfrac{\eta}{2\theta^m_\Ds}|\theta^m_\Ds\partial^2\brho^m+
%\dfrac{\theta^m_\Ds}{\Ds}(\theta^m_\Ds-\theta^{m-1}_\Ds)\partial\brho^m|^2\bigg\}\\
%=& \calJs_\Ds^1(\balpha_\Ds,\theta_\Ds)+\calJs_\Ds^2(\theta_\Ds)+\calJs_\Ds^3(\theta_\Ds)
%+\calJs_\Ds^4(\theta_\Ds).\nonumber
\end{align}
where
\begin{align*}
\calJs_\Ds^1(\balpha_\Ds,\theta_\Ds)&=\sum_{m=1}^M\dfrac{\Ds}{2\theta^m_\Ds} \sum_{i=1}^\dd\|(\balpha^m_\Ds)^\top\wbP^m_i\balpha^m_\Ds-\brho'_i(s^m)\wtheta^m_\Ds\|^2_{\rL^2(\hD)} , \quad
   \calJs_\Ds^2(\theta_\Ds)=\sum_{m=1}^M\dfrac{\beta\Ds}{\theta^m_\Ds}, \\
 \calJs_\Ds^3(\balpha_\Ds)&=\Ds\sum_{m=1}^M\dfrac{\lambda}{2\Ds^2}|\balpha^m_\Ds-\balpha^{m-1}_\Ds|^2, \quad
  \calJs_\Ds^4(\theta_\Ds)=\Ds\sum_{m=1}^M\dfrac{\eta}{2\Ds^2}|\theta^m_\Ds-\theta^{m-1}_\Ds|^2 ,
  \end{align*}
and admissible sets
\begin{align*}
\calU^\Ds_{ad} &:= \set{\balpha_\Ds \in \rH^1(0,s_F) : \balpha_\Ds|_{[s^{m-1},s^m]}\in \mathbb{P}^1, m = 1,\ldots, M
}\cap\calU_{ad}, \\
\calV^\Ds_{ad} &:= \left\{\theta_\Ds \in \rH^1(0,s_F) : \theta_\Ds|_{[s^{m-1},s^m]}\in \mathbb{P}^1, m = 1,\ldots, M
\right\}\cap\calV_{ad}.
\end{align*}
Since \eqref{eq:control_problem_s_disc} is a finite dimensional problem, the box constrains
and the continuity of $\calJs_{\Ds}$ implies the existence of a solution
$\left(\bbalpha_\Ds,\btheta_\Ds\right)$ to \eqref{eq:control_problem_s_disc}.
\begin{theorem}[Problem 2: convergence to global minimizers]\label{thm:Gconv_pb2}
The family of global minimizers $\left\{\left(\bbalpha_{\Ds},
\btheta_\Ds\right)\right\}_{\Ds>0}$ to \eqref{eq:control_problem_s_disc} is uniformly bounded
and it contains a subsequence that converges $\rH^1(0,s_F)$-weak to
$ \left(\bbalpha, \btheta\right)$, a global minimizer of \eqref{eq:control_problem_sf},
and $\lim_{\Ds\to 0}\calJs_\Ds(\bbalpha_{\Ds}, \btheta_\Ds)=\calJs(\bbalpha, \btheta)$.
\end{theorem}
\begin{proof}
As in Theorem~\ref{thm:Gconv} we proceed in several steps.
\begin{enumerate}[1.-]
\item \textit{Boundedness of $\left\{(\bbalpha_{\Ds}, \btheta_\Ds)\right\}_{\Ds}$ in
$[\rH^1(0,s_F)]^2$}:
 Given that the constant function  $\left(\balpha_0(s),\theta_0(s)\right)
=(\balpha_0,\theta_0)$  belongs to $\calU^\Ds_{ad}\times\calV^\Ds_{ad}$
and the fact that $\left(\bbalpha_{\Ds}, \btheta_\Ds\right)$ minimizes
$\calJs_{\Ds}$ we have the bound
\[
\calJs_{\Ds}(\bbalpha_{\Ds},\btheta_{\Ds})\leq \calJs_{\Ds}(\balpha_0,\theta_0)
\leq C\left(\dfrac{|\balpha_0|^4}{|\theta_0|}+|\theta_0| +\dfrac{1}{|\theta_0|}\right).
\]
Since $\lambda,\eta >0$, this implies a uniform bound for $(\bbalpha_{\Ds}, \btheta_\Ds)$ in $\rH^1(0,s_F)$ and
the existence of a (not relabeled)  subsequence that converge to $(\bbalpha, \btheta)$
weakly in $[\rH^1(0,s_F)]^2$ and strongly in $[C([0,s_F])]^2$.

\item \textit{Lower bound inequality}: Let us consider a sequence
$\left\{(\balpha_\Ds,\theta_\Ds)\right\}_{\Ds>0}\subset \calU^\Ds_{ad}\times\calV^\Ds_{ad}$ that
converges weakly to $(\balpha,\theta)$ in $[\rH^1(0,s_F)]^2$.  Let
 $\overline{\Pi}_\Ds$ be the piecewise constant interpolation operator
 at the nodes $\{s^m\}_{m>0}$. Then by straightforward computations
 it follows that
\begin{equation*}
    \| \overline{\Pi}_\Ds\balpha_\Ds - \balpha \|_{\rL^2(0,s_f)}
     \rightarrow 0 \quad \mbox{and} \quad \| \overline{\Pi}_\Ds\theta_\Ds - \theta \|_{\rL^2(0,s_f)} \rightarrow 0.
\end{equation*}
Moreover, from the smoothness of $\wbP_{i}$ and $\brho'$,
if $\overline{\wbP}_i$, $\overline{\brho'}$
denote the piecewise constant interpolation in time then
$\overline{\brho'}\to \brho'$ and $\overline{\wbP}_i\rightarrow \wbP_i$ in $\rL^2(0,s_f;\rL^2(\hD))$, for
 $i = 1, \dots, r$. This, together with the box constrains of $(\bbalpha_{\Ds}, \btheta_\Ds)$, leads to
\begin{gather}
\begin{aligned}\label{eq:conv_F12}
\dfrac{\overline{\Pi}_\kappa\balpha_\Ds^\top\overline{\wbP}_i\overline{\Pi}_\kappa\balpha_\Ds}{\overline{\Pi}_\kappa\theta_\Ds^{1/2}}
-\overline{\Pi}_\kappa\wtheta_\Ds^{1/2}\overline{\brho}'
&\to\dfrac{\balpha^\top\wbP_i\balpha}{\theta^{1/2}}-\wtheta^{1/2}{\brho}' \quad \mbox{in  }  \rL^2(0,s_f;\rL^2(\hD))\\
\dfrac{1}{\overline{\Pi}_\kappa\theta_\Ds}&\to\dfrac{1}{\theta}\quad \mbox{in  }  \rL^2(0,s_f) ,
 \end{aligned}
 \end{gather}
which imply
\begin{equation}\label{eq:Low_bnds_1}
\calJs^1(\balpha,\theta)= \underset{\Ds\to 0}{\lim} \calJs^1_\Ds(\balpha_\Ds,\theta_\Ds) \quad \mbox{and}\quad
\calJs^2(\theta)= \underset{\Ds\to 0}{\lim} \calJs^2_\Ds(\theta_\Ds).
\end{equation}
On the other hand, given that $\partials\balpha_\Ds$ and
$\partials\theta_\Ds$  converge weakly to
 $\partials\theta$ in $\rL^2(0,s_F)$ and $\partials\balpha$ , respectively,
 and $\calJs^3(\balpha_\Ds)=\calJs^3_\Ds(\balpha_\Ds)$, $\calJs^4(\balpha_\Ds)=\calJs^4_\Ds(\balpha_\Ds)$,
invoking the weak lower semi-continuity of the $\rH^1(0,T)$ semi-norm we obtain
\begin{equation}\label{eq:Low_bnds_2}
\calJs^3(\balpha)\leq \underset{\Ds\to 0}{\lim \inf}\calJs^3_\Ds(\balpha_\Ds)\quad \mbox{and}\quad
\calJs^4(\theta)\leq  \underset{\Ds\to 0}{\lim \inf}\calJs^4_\Ds(\theta_\Ds).
\end{equation}
Therefore, from \eqref{eq:Low_bnds_1}-\eqref{eq:Low_bnds_2} we conclude that
\begin{align*}
\calJs(\balpha,\theta) &\leq \underset{\Ds\to 0}{\lim \inf}\calJs^1_\Ds(\balpha_\Ds,\theta_\Ds)+\underset{\Ds\to 0}{\lim \inf}\calJs^2_\Ds(\theta_\Ds)\\
 &+\underset{\Ds\to 0}{\lim \inf}\calJs^3_\Ds(\balpha_\Ds)+\underset{\Ds\to 0}{\lim \inf}\calJs^4_\Ds(\theta_\Ds)\leq\underset{\Ds\to 0}{\lim \inf}\calJs_\Ds(\balpha_\Ds,\theta_\Ds).
\end{align*}
\item \textit{Existence of a recovery sequence}: If $(\balpha,\theta)\in\calU_{ad}\times \calV_{ad}$,
is given, then  $(\Pi_{\Ds}\balpha,\Pi_{\Ds}\theta)$ belongs to $\calU_{ad}^\Ds\times\calV_{ad}^\Ds$ where
$\Pi_{\Ds}$ is the piecewise linear Lagrange interpolation operator.
Since
$\partial_s(\Pi_{\kappa}\balpha) \rightarrow \partial_s\balpha$ in $\rL^2(0,s_F)$,
$\partial_s(\Pi_{\kappa}\theta) \rightarrow \partial_s\theta$ in $\rL^2(0,s_F)$,
 we proceed as in Step 2 to obtain the convergence of $\calJs^1_\Ds$ and
$\calJs^2_\Ds$, whence
\begin{align*}
 \underset{\Ds\to 0}{\lim\sup}\calJs_\Ds(\Pi_{\Ds}\balpha,\Pi_{\Ds}\theta) &\leq \underset{\Ds\to 0}{\lim\sup}\calJs^1_\Ds({\Pi}_{\kappa}\balpha,{\Pi}_{\Ds}\theta)
 + \underset{\Ds\to 0}{\lim\sup}\calJs^2_\Ds(\Pi_{\Ds}\theta)\\
& \quad+ \underset{\Ds\to 0}{\lim\sup}\calJs^3_\Ds(\Pi_{\Ds}\balpha) + \underset{\Ds\to 0}{\lim\sup}\calJs^4_\Ds(\Pi_{\Ds}\theta)
\leq\calJs(\balpha,\theta).
\end{align*}
\end{enumerate}
We finally argue as in steps  4 and 5 of Theorem~\ref{thm:Gconv} to conclude that
 $(\bbalpha,\btheta)$ is a global  minimizer of \eqref{eq:PJ2}, i.e.
 \[
 \calJs(\balpha,\theta) \geq \calJs(\bbalpha,\btheta)\quad \forall (\balpha,\theta) \in \calU_{ad}\times\calV_{ad},
 \]
$\lim_{\Ds\to 0}\calJs_\Ds(\bbalpha_\Ds,\btheta_\Ds)=\calJs(\bbalpha,\btheta)$ and
$\left\{\bbalpha_{\Ds},\btheta_\Ds\right\}_{\Ds>0}$
converges $\rL^2$-strongly and $\rH^1$-weakly to $(\bbalpha,\btheta)$. We skip details for brevity.
\end{proof}

%------------------------------------------------
%----  Numerical examples
%------------------------------------------------

\section{Numerical examples}\label{s:numerics}

Let us illustrate the performance of the proposed discrete
schemes  \eqref{eq:control_problem_t_disc} and \eqref{eq:control_problem_s_disc}.
For the first scheme we consider the approximation of a constant and  time dependent
vector field $\bvel$ (space independent)  on $D_t$.
For the second scheme  we compute the optimal time and force by solving $\eqref{eq:control_problem_s_disc}$.
Finally, we use the computed optimal force as  an input to an advection-diffusion
partial differential equation used in magnetic drug targeting (see \cite{GR2005}).
For all examples $\Omega\subset \R^2$ is a ball of unit radius centered at $(0,0)$ and, the dipoles positions and directions are
$\bx_{k+1}=1.2(\cos(k\pi/4),\sin(k\pi/4))$ and $\widehat{\bv{d}}_{k+1}=(\cos(k\pi/4),\sin(k\pi/4)),
k=0,\ldots,7$ ($\nd=8$), respectively (see Figure~\ref{fig:example1_domains} (left)).

To solve the minimization problem, we use a \textsc{Matlab} implementation of the projected BFGS with Armijo  line search \cite{K1999}; alternative strategies such as semi-smooth Newton \cite{HIK2002} can be immediately applied as well. The minimization algorithm is terminated when the $l^2$-norm of the projected gradient
is less or equal to $10^{-5}$.  Besides the initial condition $\balpha_0\in \R^{\nd}$, the minimization algorithm requires an
initial guess for the solution $\bbalpha$.
We propose below Algorithm \ref{alg:algorithm} in order to compute an initial guess  $\balphaI\in \calH^\tau_{ad}$.
Notice that the speed of convergence of the minimization
algorithm is sensitive to the choice of  $\balphaI$.

Given $\balpha_*$ and $\balpha^*$, by $\mbox{Proj}_{[\balpha_*,\balpha^*]}$
 we denote pointwise projection on the interval $[\balpha_*,\balpha^*]$
\[
\mbox{Proj}_{[\balpha_*,\balpha^*]}(\mathbf{x})
=\min\left\{\balpha^*,\max\left\{\mathbf{x},\balpha_*\right\}\right\},
\]
where $\min$ and $\max$ are interpreted componentwise.
The aim of Algorithm~\ref{alg:algorithm} is to  initialize the
actual minimization algorithm used to solve \eqref{eq:control_problem_t_disc}.
Notice that the algorithm above is the so-called finite horizon model predictive or instantaneous optimization 
algorithm \cite{AHNSWsub,HK1999}: the minimization takes place over one time step only.
Under the assumption that solving a one step problem is cheaper, the previous algorithm
provides us a  ``fast" yet ``accurate" initial guess to solve \eqref{eq:control_problem_t_disc}.
\begin{remark}[non-convexity] \rm
  Due to the non-convexity of the cost functional $\calJ$ (and $\calJ_\tau$), we may converge
  to different local minima  depending on the choice
  of the initial guess.
\end{remark}
\begin{algorithm}[H]
\caption{: Initialization algorithm}
\begin{algorithmic}[1]

\STATE $\textbf{Input:} \, \balpha_0,\,\balpha_*,\, \balpha^*,\, \lambda,\,\tau,\, \hD,\, \verb"tol",\, \hbP^n_i,\, \hbvel^n_{i},\, n=1,\ldots, N,\, $
\STATE Set $\mathbf{x}^0:=\balpha_0$
\FOR{$n=1,\ldots,N$}
    \STATE Solve for $\mathbf{x}\in \R^{\nd}$ \\
    $\qquad\displaystyle\underset{\balpha_*\leq \mathbf{x}\leq \balpha^*}{\min_{\mathbf{x}\in\R^{\nd}}}J(\mathbf{x}), \qquad J(\mathbf{x})=
\dfrac{1}{2}\sum_{i=1}^\dd\|\mathbf{x}^\top
\hbP^n_i\mathbf{x}-\hbvel^n_{i}\|^2_{\rL^2(\hD)}
+\dfrac{\lambda}{2\Dt^2}|\mathbf{x}-\mathbf{x}^0|^2$\\
 with termination criterion: $|\mathbf{x}-\mbox{Proj}_{[\balpha_*,\balpha^*]}(\mathbf{x}-\nabla J(\mathbf{x}))| < \verb"tol"$.
\STATE $\balphaI(n\tau)=\mathbf{x}$
\STATE $\mathbf{x}^0 \gets \mathbf{x}$
\ENDFOR
\end{algorithmic}
\label{alg:algorithm}
\end{algorithm}

\subsection{Problem 1: Approximation of a vector field}
\label{sec:example1}%
We consider the discrete minimization problem \eqref{eq:control_problem_t_disc} for two
different choices of $\bvel$ and $D_t$. The admissible set is characterized by upper and lower bounds
$\balpha^*=(2,\ldots,2)\in\R^8$ and $\balpha_*=(-2,\ldots,-2)\in\R^8$, respectively.
The final time is $T=1$
and the reference domain $\hD$ is a ball of radius $0.2$ centered at $(-0.75,0)$.
  We are interested in the approximations of a constant vector field
  $\bvel_1(\bx,t) = (1, 0)^\top$ and a time dependent vector field given
  by $\bvel_2(\bx,t)=(\sin(\pi(1-t)), -\cos(\pi (1-t))^\top$.
  The moving domains $D_{i,t}$ are such that $D_{i,t}=\bx_I(t,\hD)$, with
$\bx_I(t,\hbx)=\bvarphi_i(t)+\hbx$, $i=1,2$.

Here $\bvarphi_1(t)=(t, 0)^\top$
  and  $\bvarphi_2(t)=0.6(\cos(\pi(1-t)), \sin(\pi (1-t))^\top$ represent
  trajectories of the barycenters of $D_{1,t}$ and $D_{2,t}$, respectively as
  shown in  Figure~\ref{fig:example1_domains} (center and right).
\begin{figure}
 \centering
\includegraphics[width=0.3\linewidth]{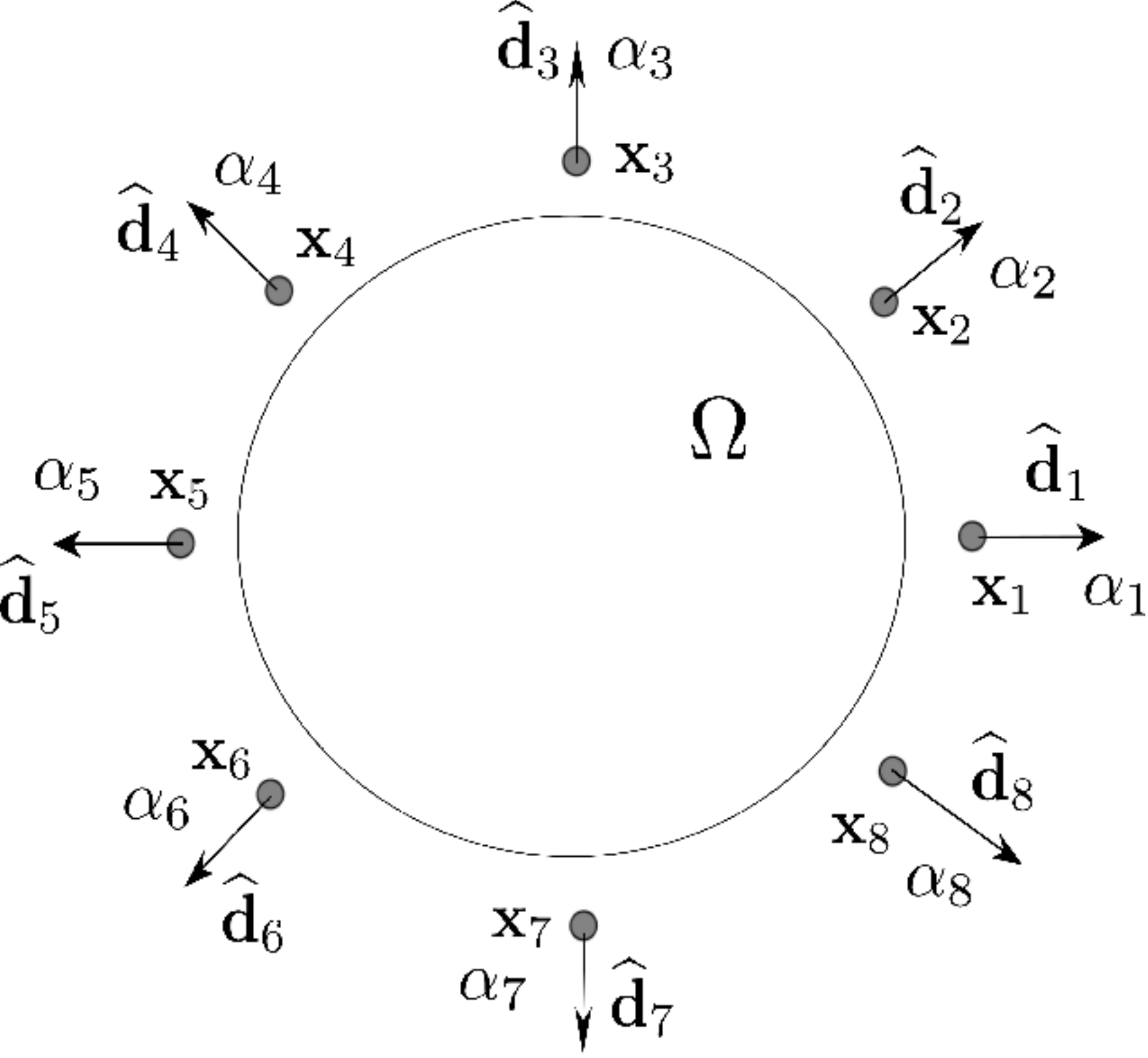}\qquad
\raisebox{0.3\height}{\includegraphics[width=0.18\linewidth]{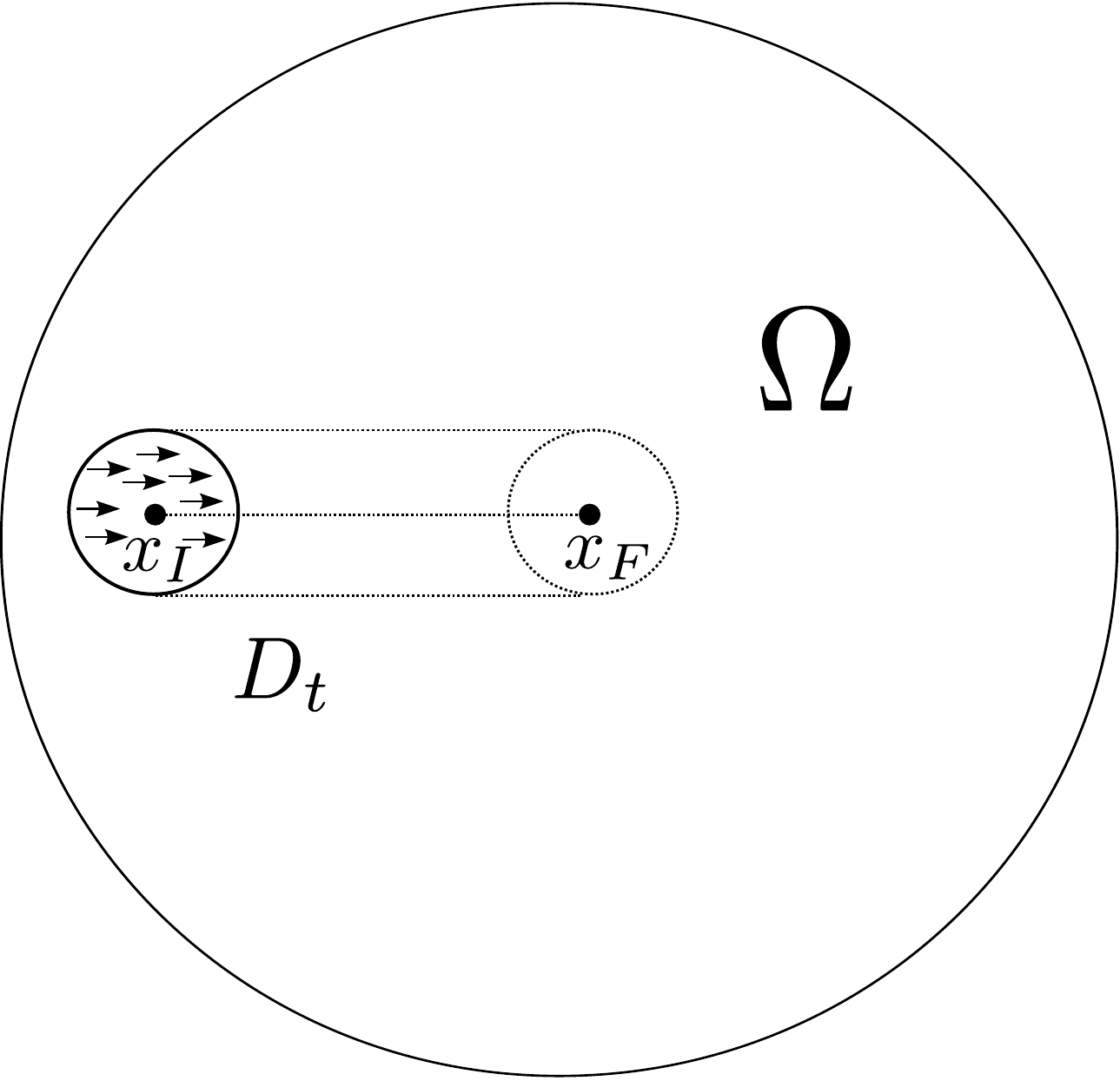}}\qquad
\raisebox{0.3\height}{\includegraphics[width=0.18\linewidth]{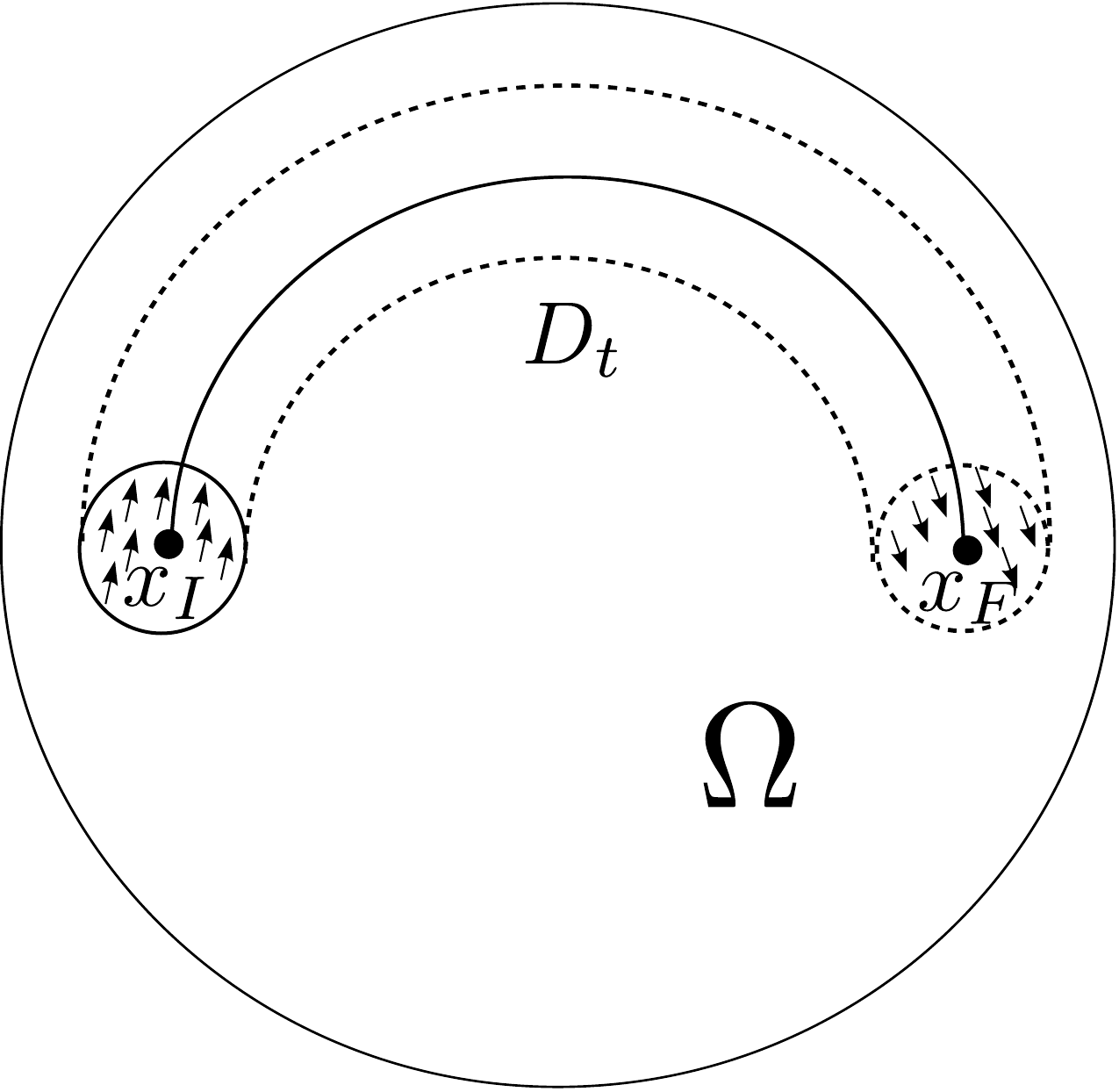}}
\caption{Computational configuration (left) and two different moving subdomains $D_t$ within
$\O\subset\R^2$.  These domains, $D_{1,t}$ (left) and $D_{2,t}$ (right), do not deform, are initialized at $\bx_I$
and travel to their destinations $\bx_F$ along different curves.
The target vector fields $\left\{\bvel_i\right\}_{i=1}^2$, which are tangent to
the curve $\calC$, are  represented by arrows
for the initial configuration (left) and also for the final configuration (right).
}
\label{fig:example1_domains}
\end{figure}
  For each of these configurations we have solved problem \eqref{eq:control_problem_t_disc}
 for   $N=80$  time intervals, $\lambda = 10^{-5}$ and  initial condition $\balpha_{0}= (1,\ldots,1)^\top\in \R^{\nd}$.

 To analyze the discrete optimization problem we consider
 the approximation of $\bvel_1$ on $D_{1,t}$ for two initial guesses:
 constant function $\balphaI^a= {\bf 1}$ and $\balphaI^b$ given by Algorithm~\ref{alg:algorithm} with  $\verb"tol"=10^{-3}$.
 Figure~\ref{fig:example1_ini} (left) shows the initial guess $\balphaI^b$ obtained by Algorithm~\ref{alg:algorithm} with a total of
201 iterations. We recall that, at each time step $n$,  the iterations of the minimization problem in Algorithm~\ref{alg:algorithm}
depends on $\nd$ unknowns.
The solution $\bbalpha(t)=(\bar{\alpha}_i(t))_{i=1}^8$ of problem \eqref{eq:control_problem_t_disc} with initial guesses  given
by  $\balphaI^b$ and $\balphaI^a$ are depicted in Figure~\ref{fig:example1_ini} (right) and
Figure~\ref{fig:example1_one} (left), respectively. We notice that dipoles on the left (dipoles 4, 5 and 6)
have small intensities at initial times. Such a behavior is expected because $D_{1,t}$ is close
 to the boundary of $\Omega$, where the magnetic field generated by these dipoles is large, thus  it is
 difficult for dipoles 4, 5 and 6 to ``push" in the  $\bvel_1$ direction. On the other hand, dipoles on the right
 (dipoles 1, 2 and 8) which can create an attractive field in the  $\bvel_1$ direction, have the
 largest intensities  at initial times.
 \begin{figure}
\centering
\includegraphics[width=0.46\linewidth]{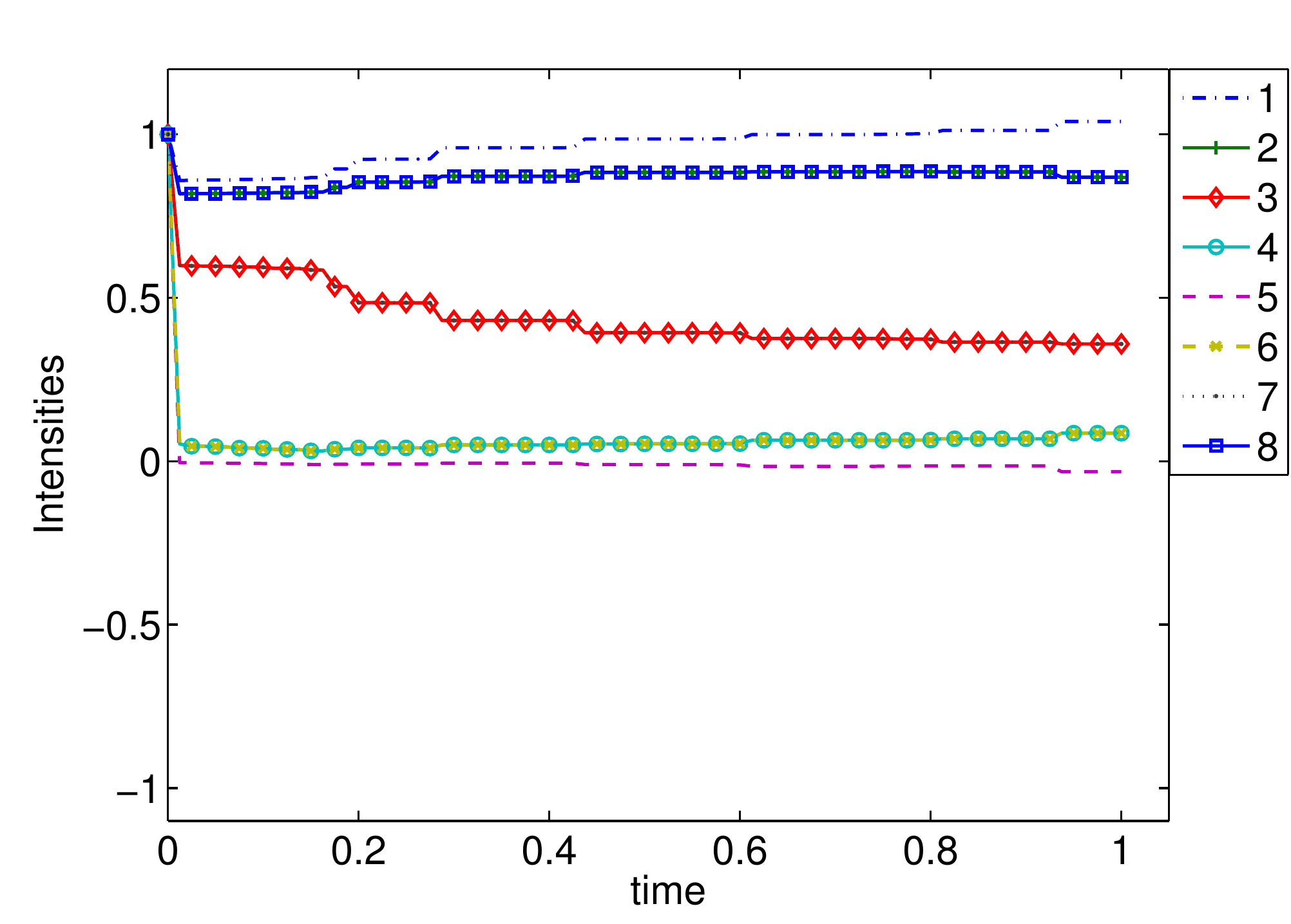}\quad
\includegraphics[width=0.46\linewidth]{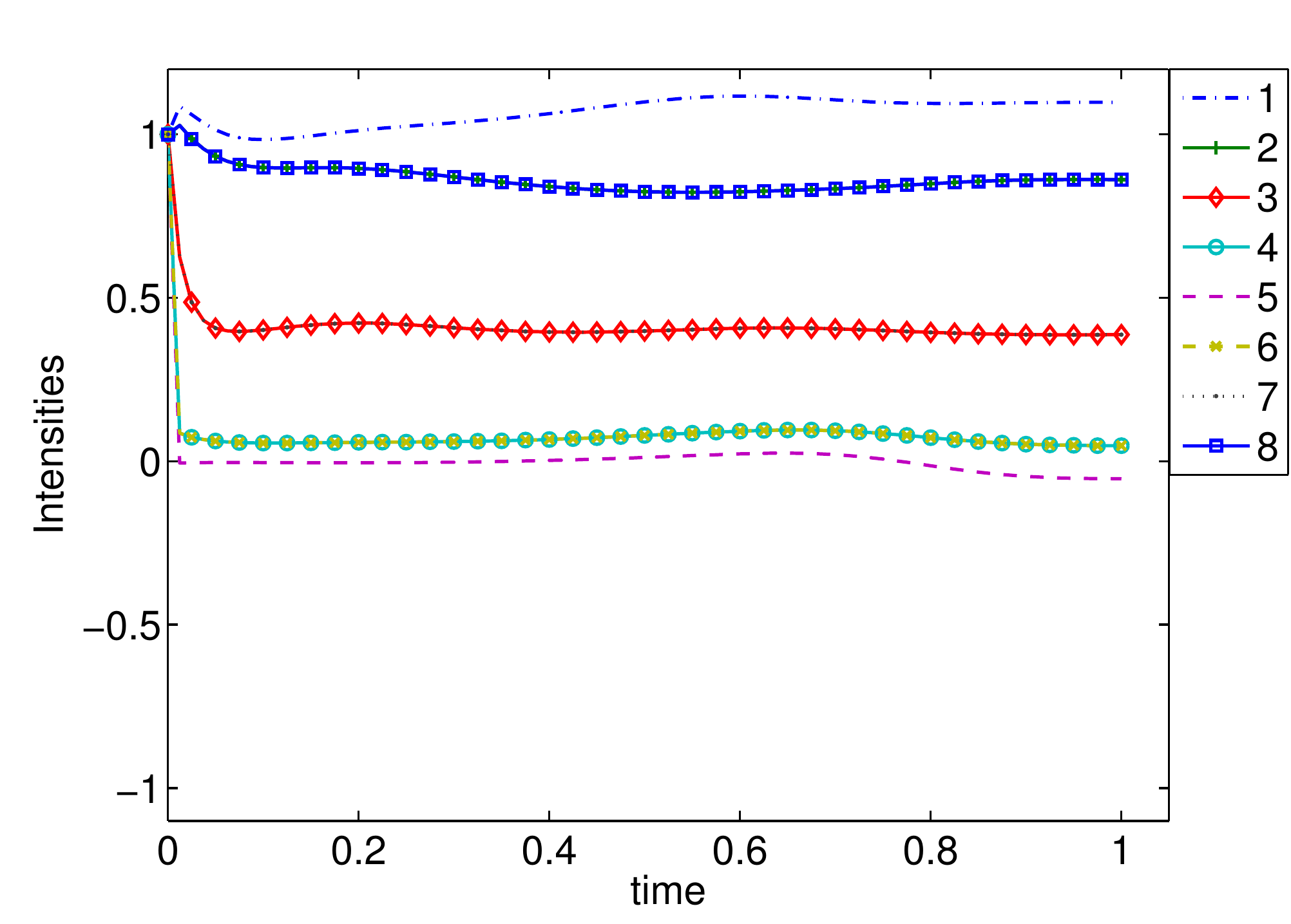}
\caption{Initial guess $\balphaI^b$ computed by Algorithm~\ref{alg:algorithm} (left) and
optimal solution $\bbalpha_\tau=(\bar{\alpha}_{i,\tau})_{i=1}^8$ to problem \eqref{eq:control_problem_t_disc}
with initial guess $\balphaI^b$ (right).
The evolution of the intensities is shown for each dipole $i=1,\ldots, 8$.
}
\label{fig:example1_ini}
\end{figure}

Figure~\ref{fig:example1_one} (right) shows  $\calJ_\tau$ per optimization  iteration (log scale).
Notice that the number of unknowns at each iteration to solve \eqref{eq:control_problem_t_disc} is $N\nd$ .
The minimization algorithm to solve \eqref{eq:control_problem_t_disc} with initial guess $\balphaI^b$ stops after 3440
iterations, whereas for the initial guess $\balphaI^a$  the stopping criteria is
satisfied after 38427 iterations: the convergence of the discrete problem
\eqref{eq:control_problem_t_disc} with initial guess $\balphaI^b$ is faster than the one computed
with $\balphaI^a$. Moreover, the initial configurations lead us to two different local minimizers:
 the limit solution thus depends on the initial data.
 \begin{figure}
\centering
\includegraphics[width=0.46\linewidth]{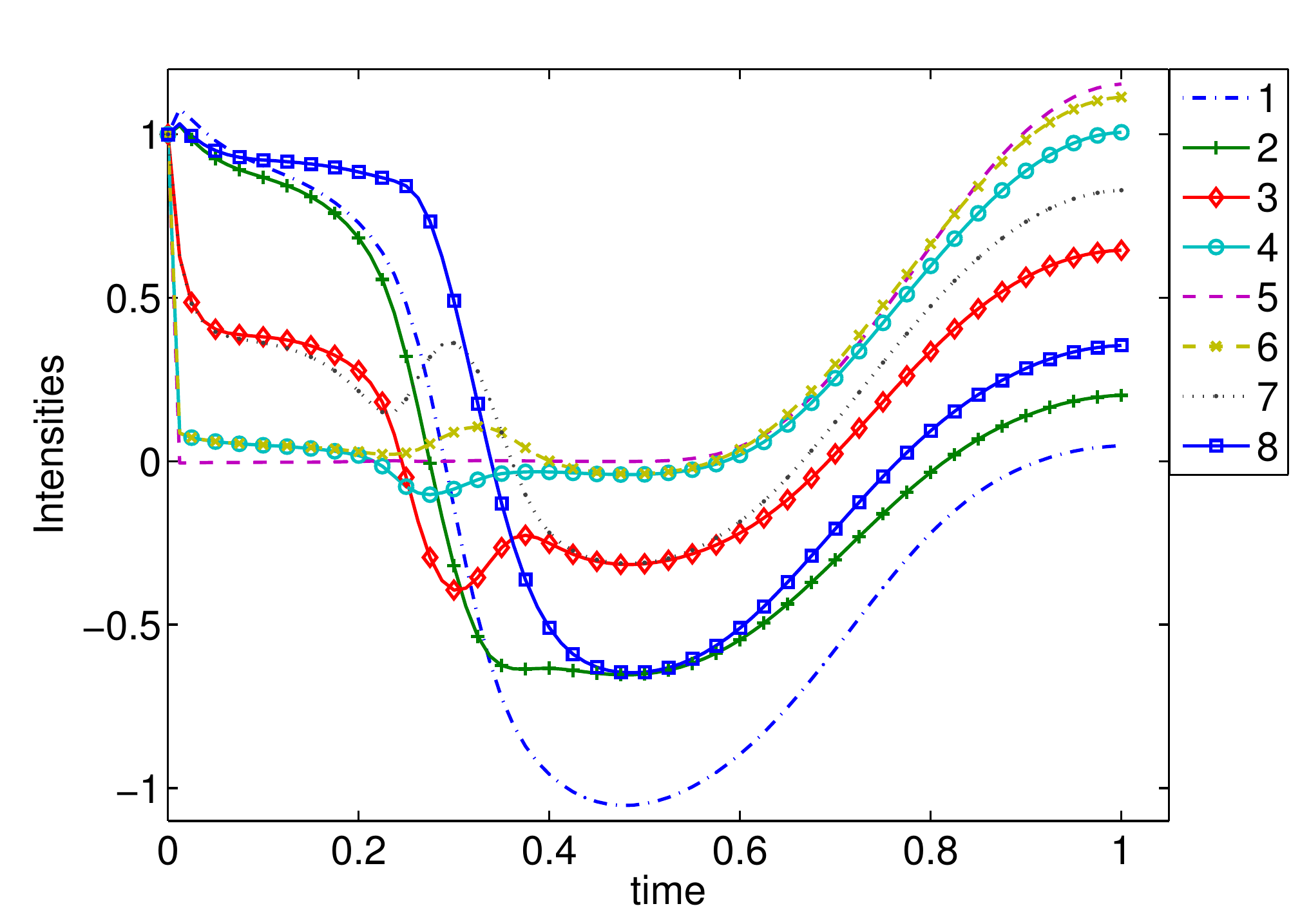}\quad
\includegraphics[width=0.46\linewidth]{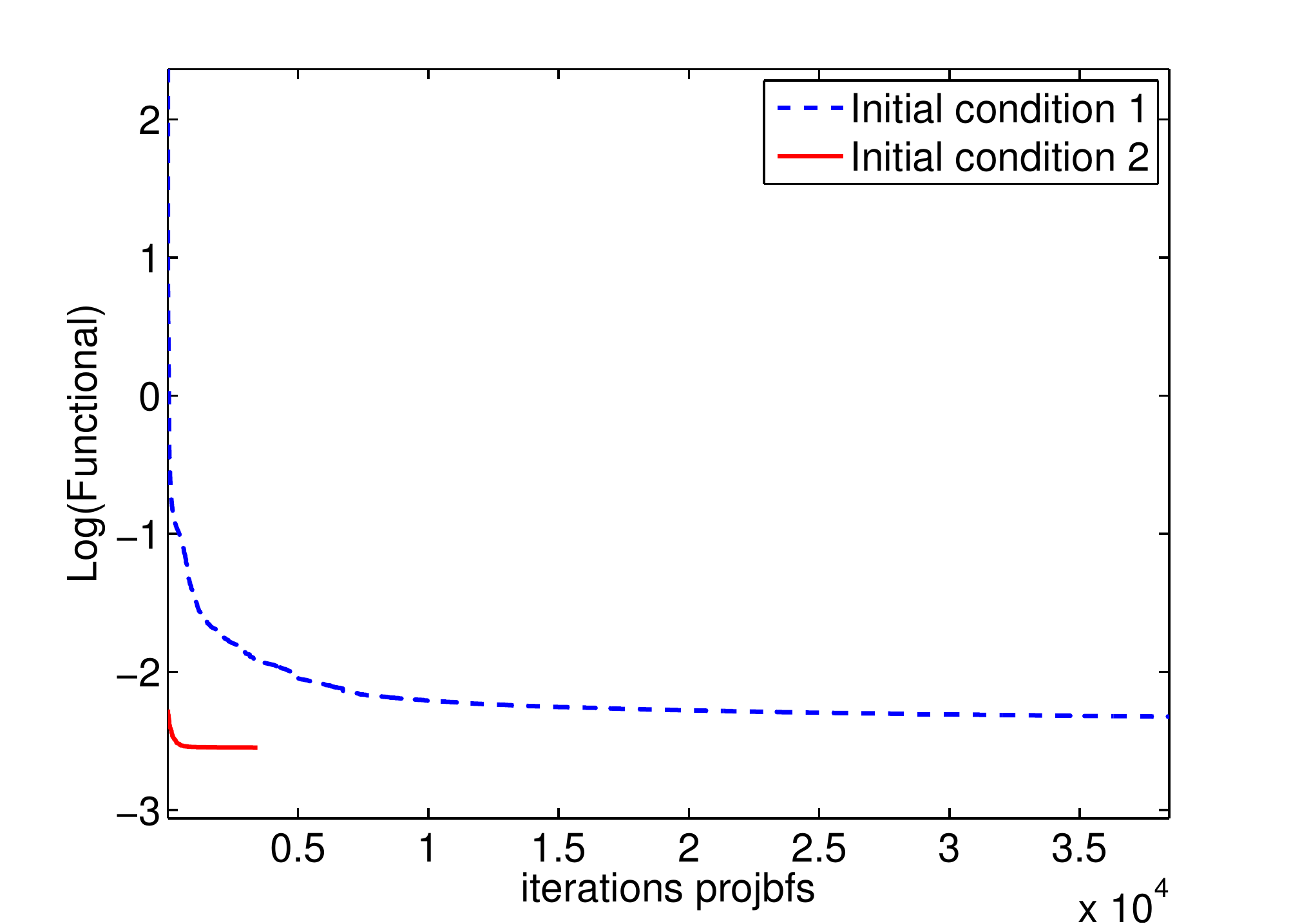}
\caption{Optimal solution  $\bbalpha_\tau=(\bar{\alpha}_{i,\tau})_{i=1}^8$ (left) to problem \eqref{eq:control_problem_t_disc}
 computed with initial guess $\balphaI^a = {\bf 1}$ and $\log(\calJ_\tau)$ computed at each optimization iteration with initial guess
$\balphaI^a$ (dashed line) and $\balphaI^b$ (solid line) (right). The evolution of the intensities is
shown for each dipole $i= 1,\ldots, 8$.
 }
\label{fig:example1_one}
\end{figure}
Figure~\ref{fig:app_ex_ini} shows the approximate field for three time instances
$t= 0.0125, 0.5, 1$ computed with initial guesses $\balphaI^a$ (top) and $\balphaI^b$ (bottom).
Here, the magnitude of magnetic force $|\nabla|\bH|^2|$ and the magnetic force vectors restricted to $D_{1,t}$ are depicted.
From Figures \ref{fig:example1_ini} (right) and \ref{fig:example1_one} (left) we notice a similar behavior
for the intensities at initial times, thus in Figure~\ref{fig:app_ex_ini} we observe an analogous magnetic force
for both solutions at $t= 0.0125$.
In both cases, the magnetic force is close to $\bvel_1$ in $D_{1,t}$ as expected,
whence about constant, whereas it is quite far from constant in the entire domain. %
\begin{figure}
\centering
\includegraphics[width=0.325\linewidth]{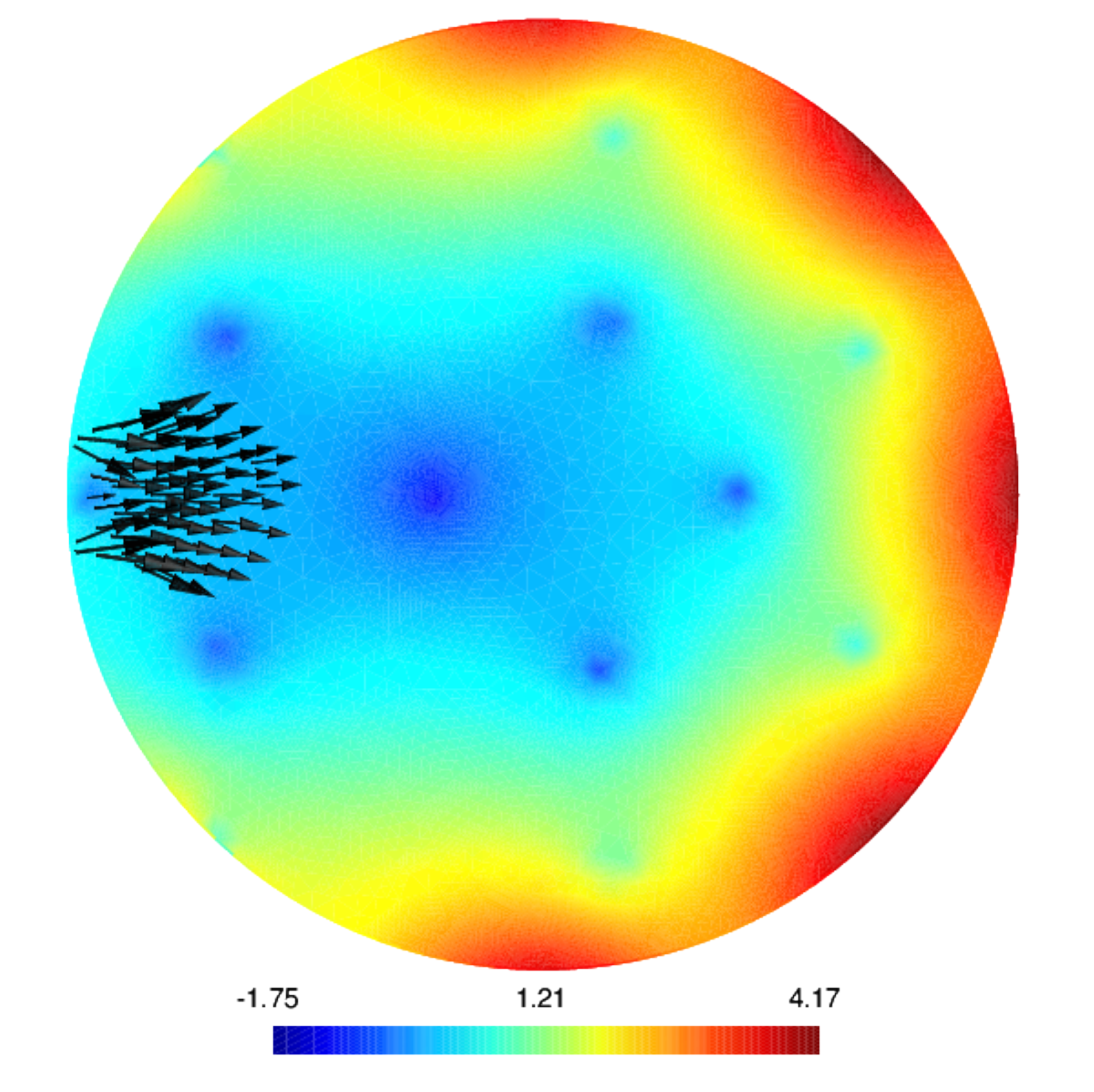}
\includegraphics[width=0.325\linewidth]{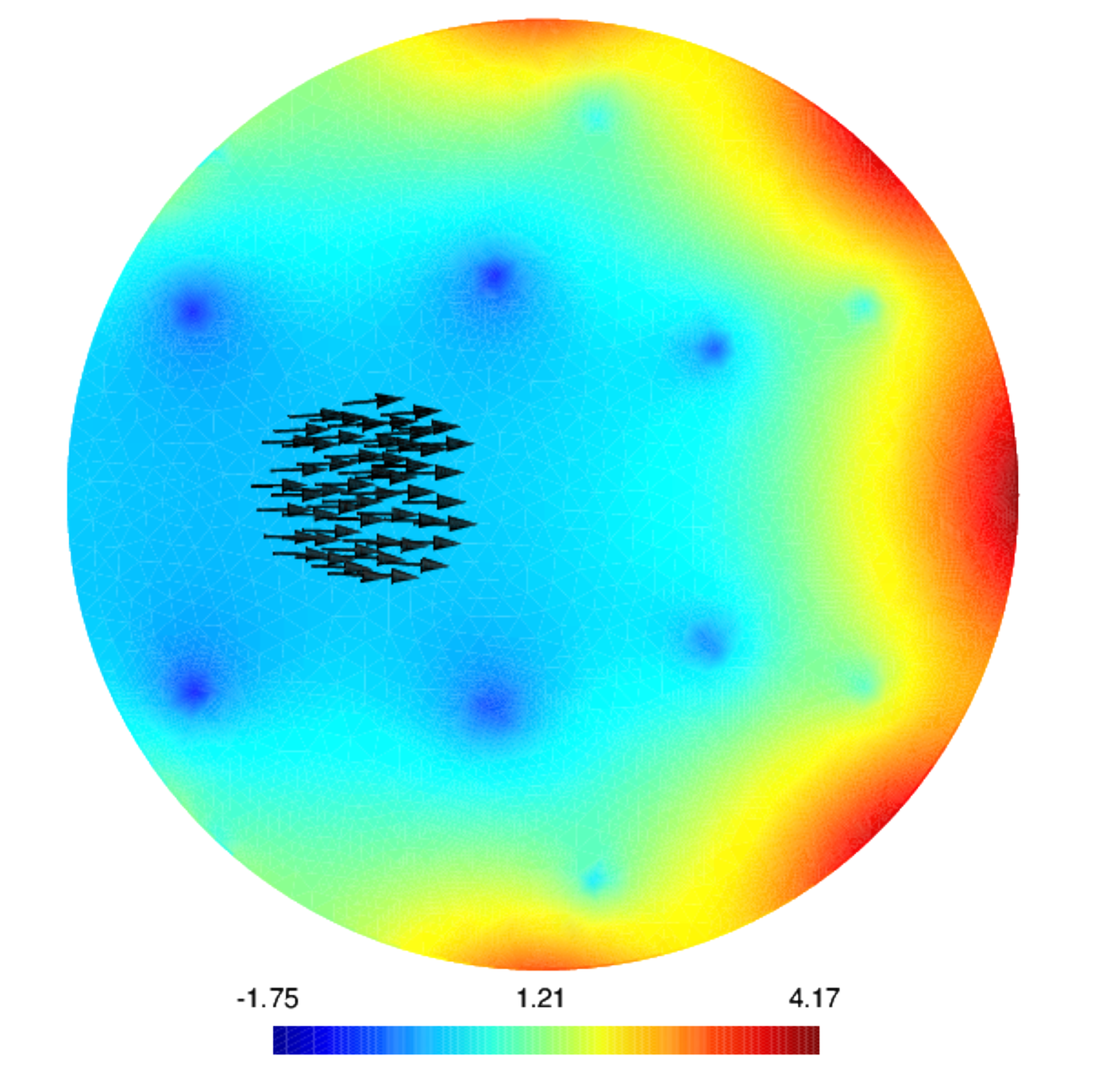}
\includegraphics[width=0.325\linewidth]{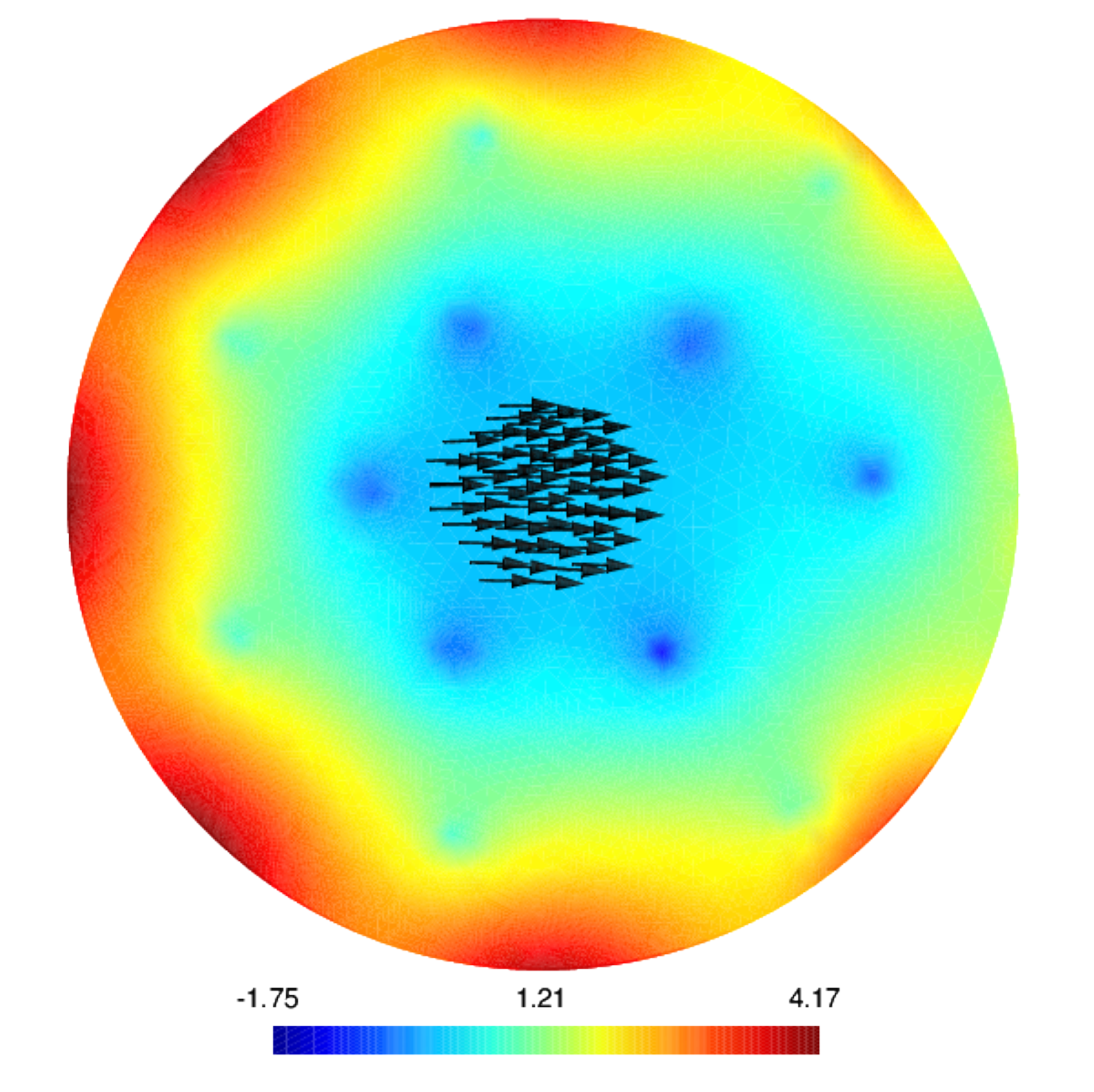}\\
\includegraphics[width=0.325\linewidth]{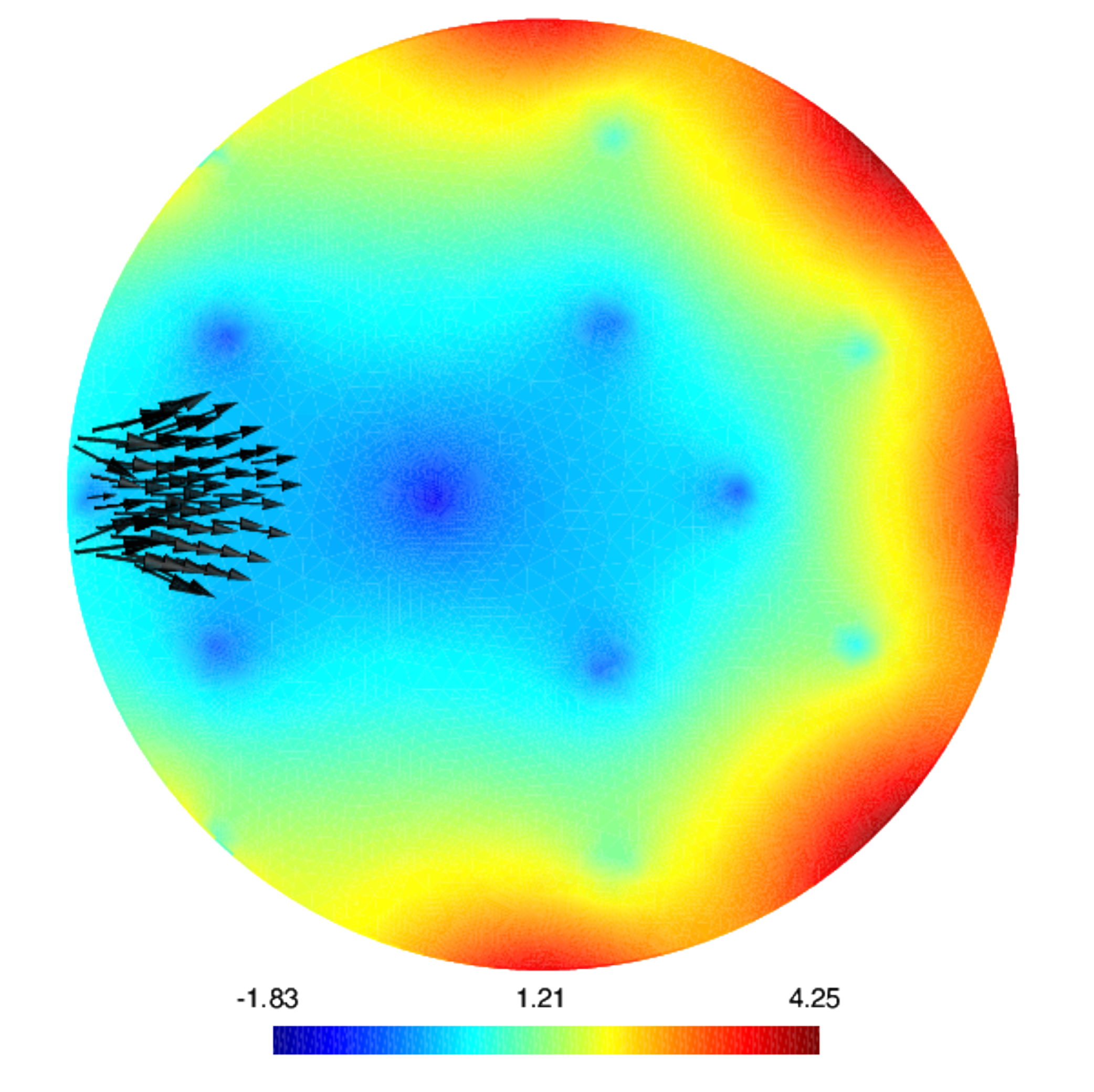}
\includegraphics[width=0.325\linewidth]{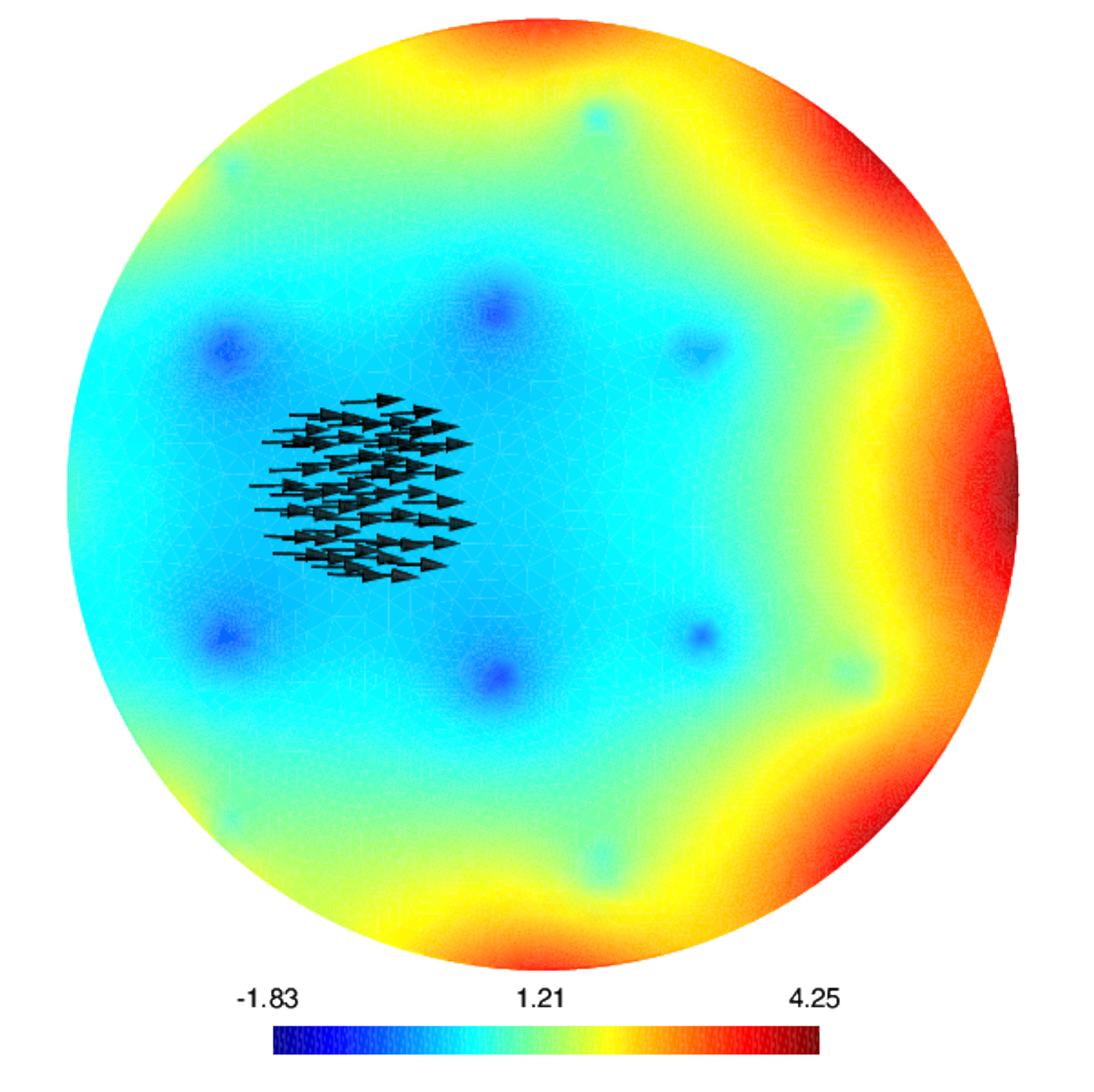}
\includegraphics[width=0.325\linewidth]{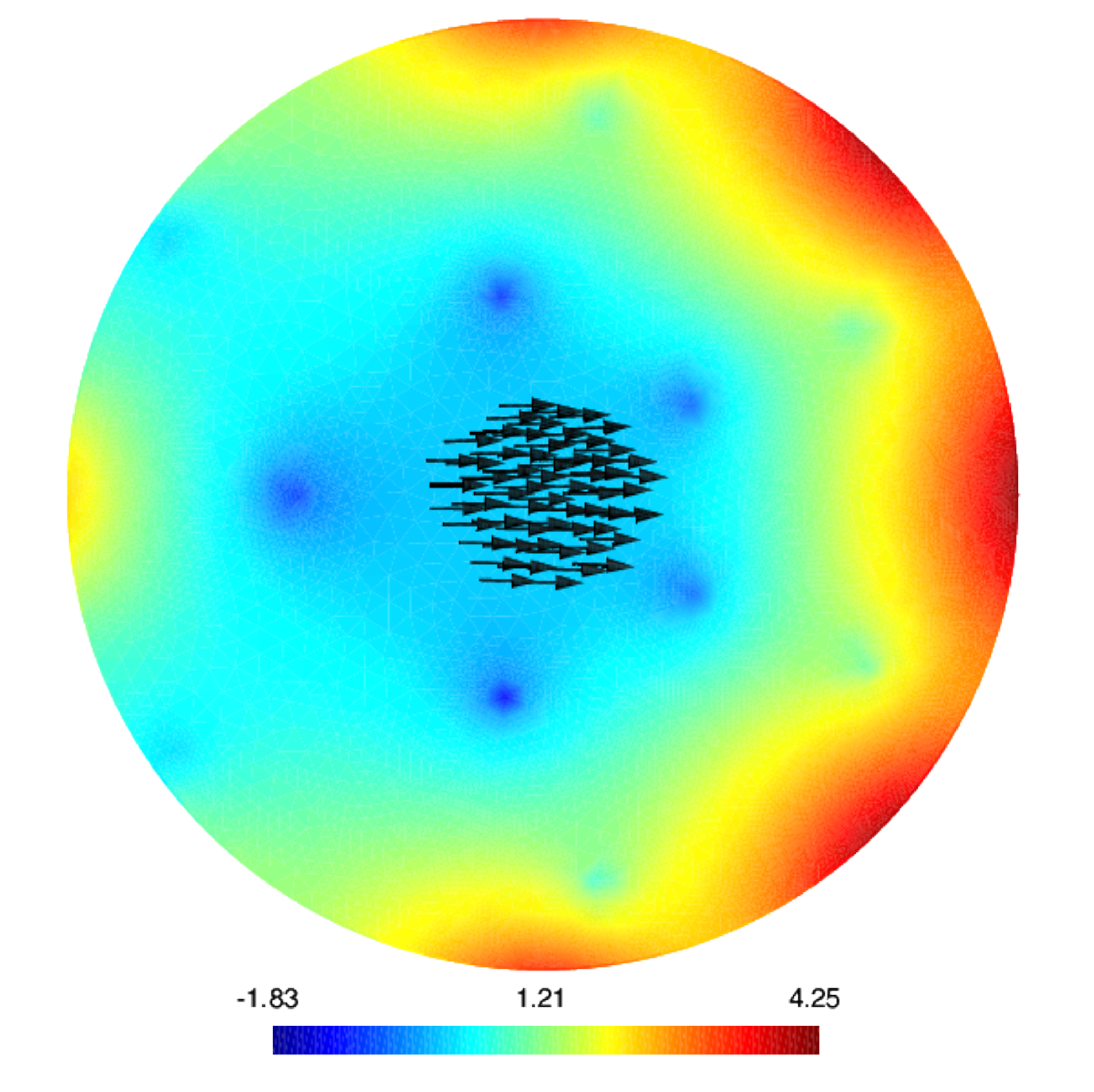}
\caption{Magnetic force solution to the minimization problem with initial guess
$\balphaI^a$ (top) and  $\balphaI^b$ (bottom).
Figures show the force  at three different times
$t=0.0125, 0.5$ and $1$ from left to right, respectively.
For illustrative purposes the magnetic forces are depicted only in $D_{1,t}$
for the same time instances
(directions shown by black arrows).
The magnetic force magnitude  $|\nabla|\bH|^2|$
is shown by the background coloring on a logarithmic scale.}
\label{fig:app_ex_ini}
\end{figure}
 \begin{figure}
\centering
\includegraphics[width=0.46\linewidth]{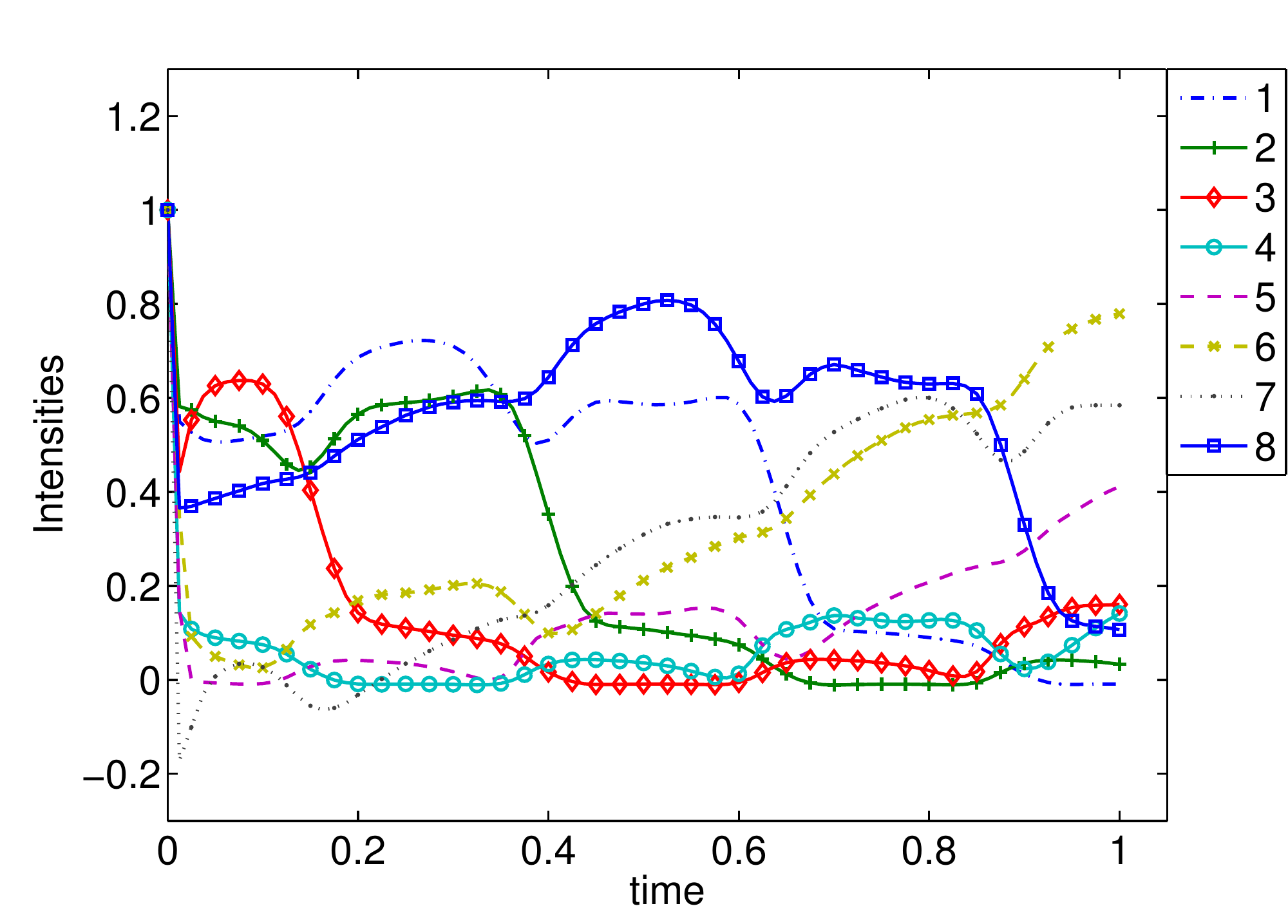}\quad
\includegraphics[width=0.46\linewidth]{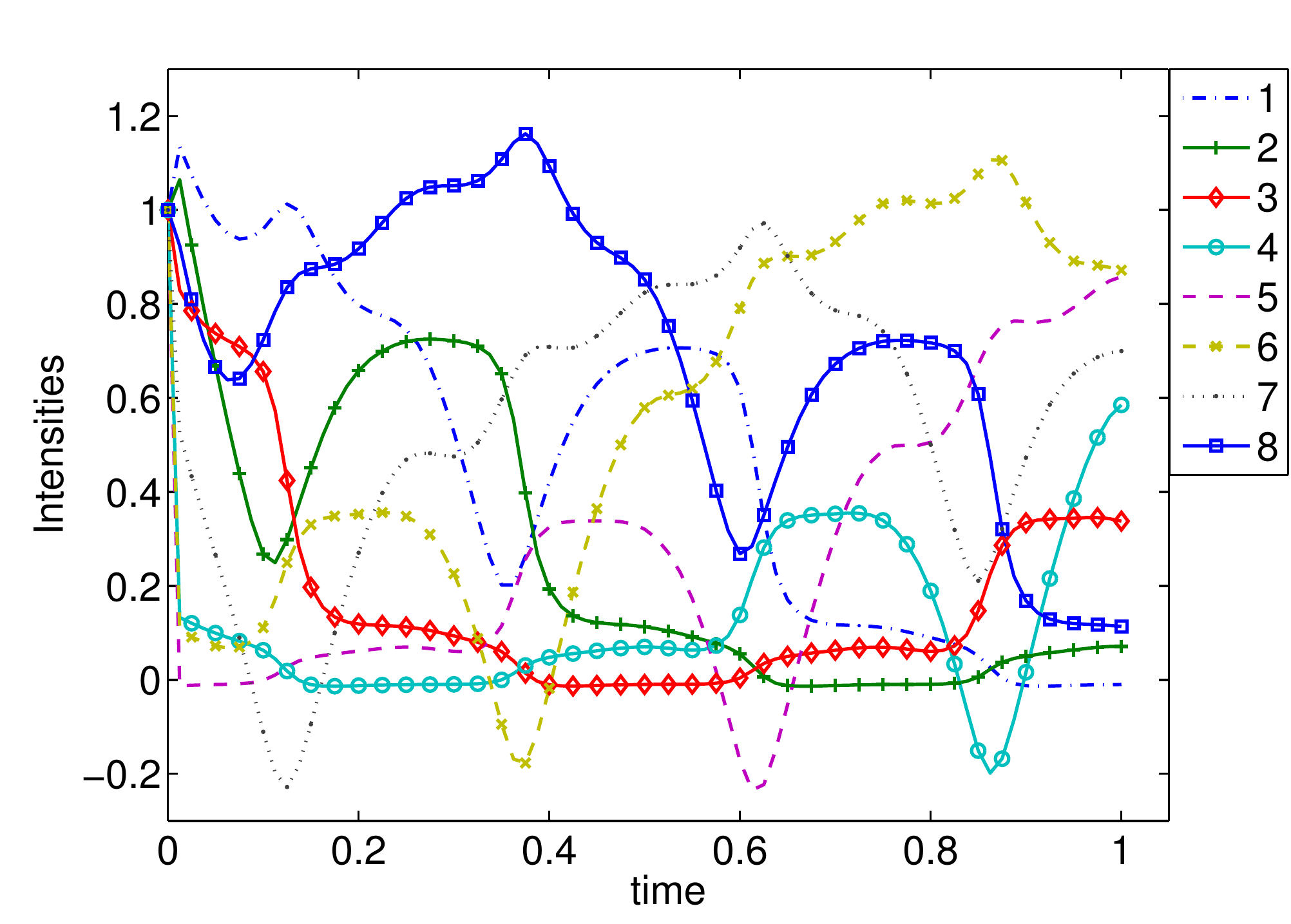}
\caption{Initial guess $\balphaI^b$ computed by Algorithm~\ref{alg:algorithm} (left) and
optimal solution $\bbalpha_\tau=(\bar{\alpha}_{i,\tau})_{i=1}^8$ to problem \eqref{eq:control_problem_t_disc}
with initial guess $\balphaI^b$ (right).
The evolution of the intensities is shown for each dipole $i=1,\ldots, 8$.
}
\label{fig:example2_ini}
\end{figure}

\begin{figure}
\centering
\includegraphics[width=0.325\linewidth]{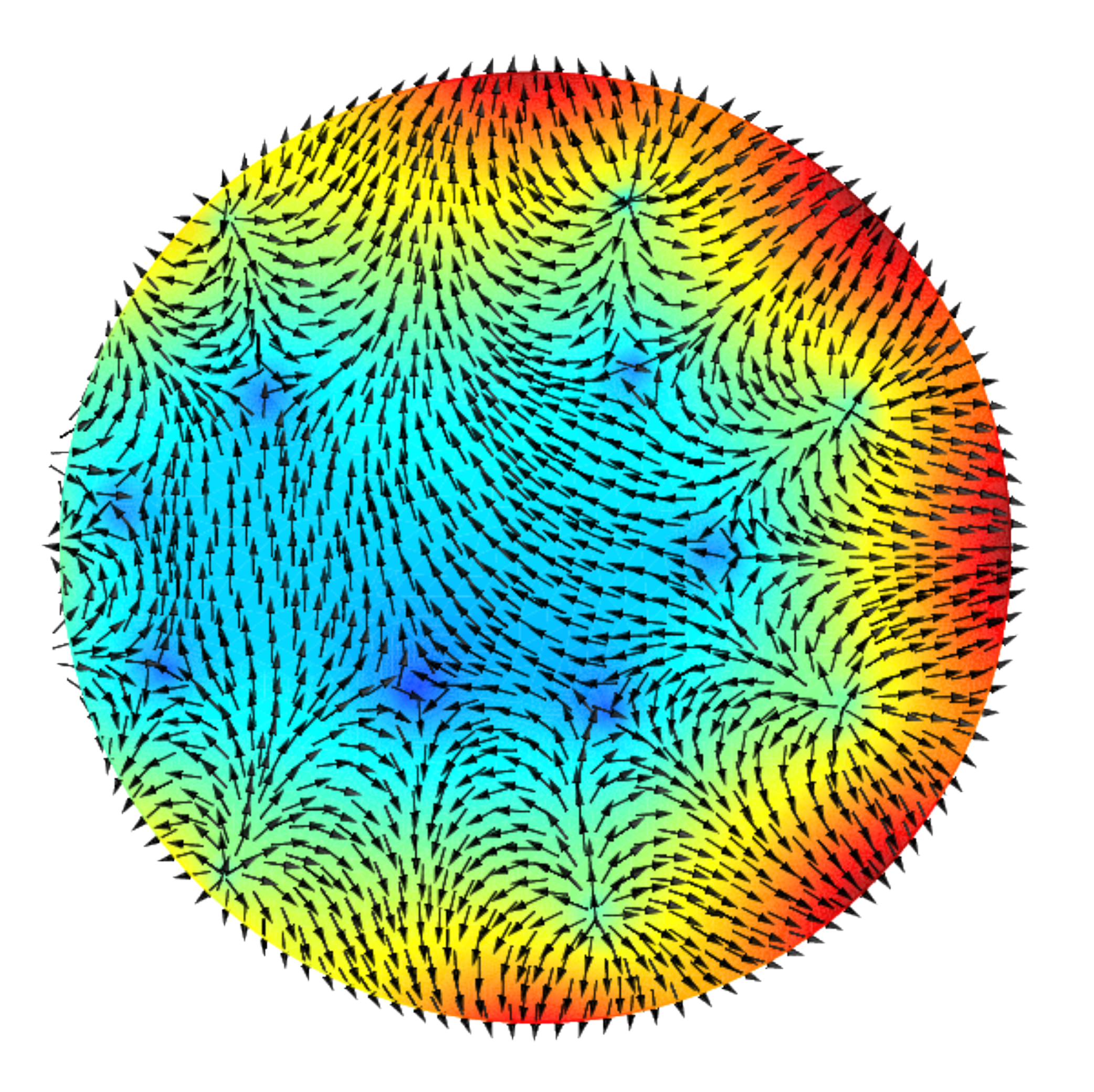}
\includegraphics[width=0.325\linewidth]{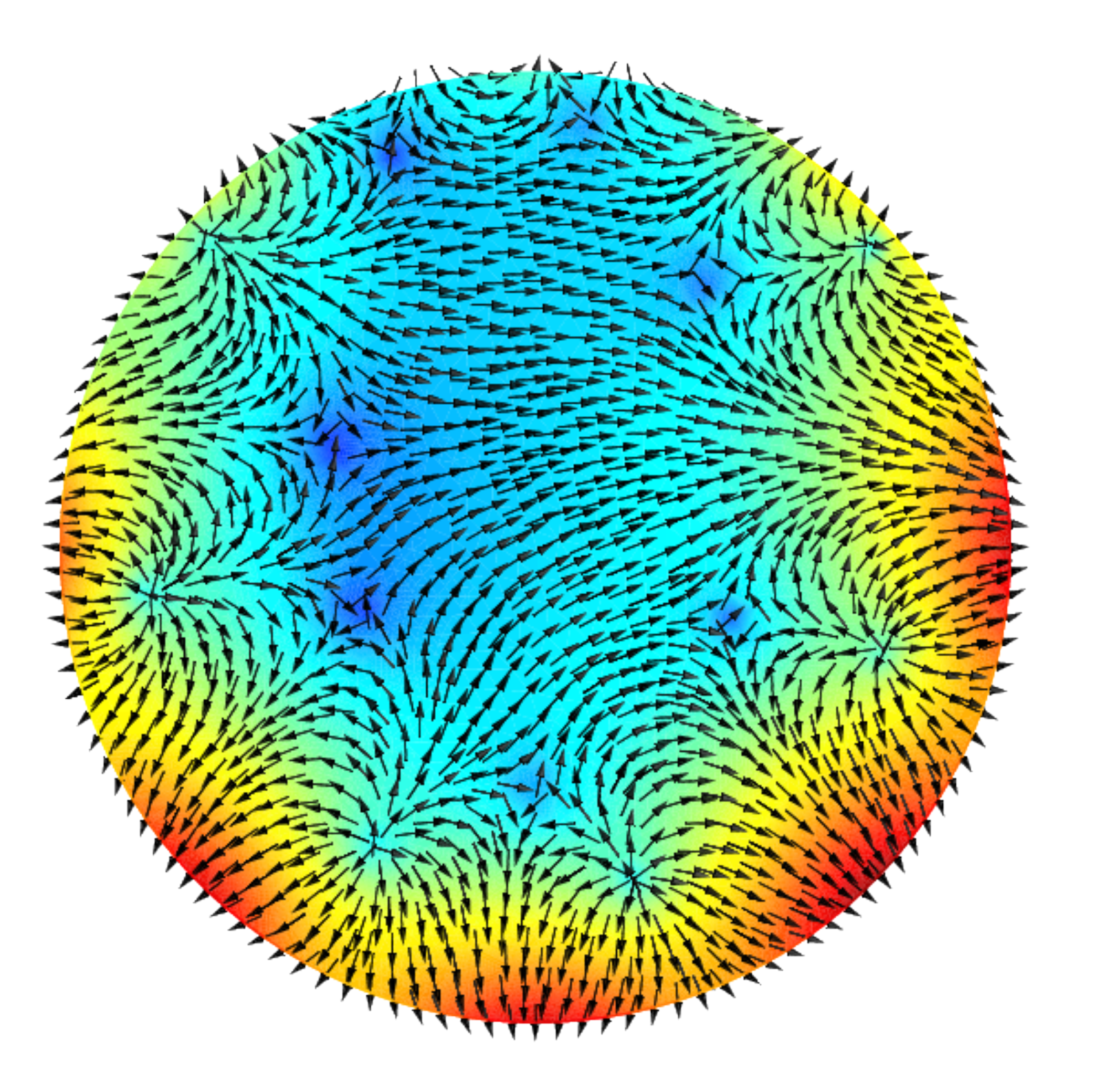}
\includegraphics[width=0.325\linewidth]{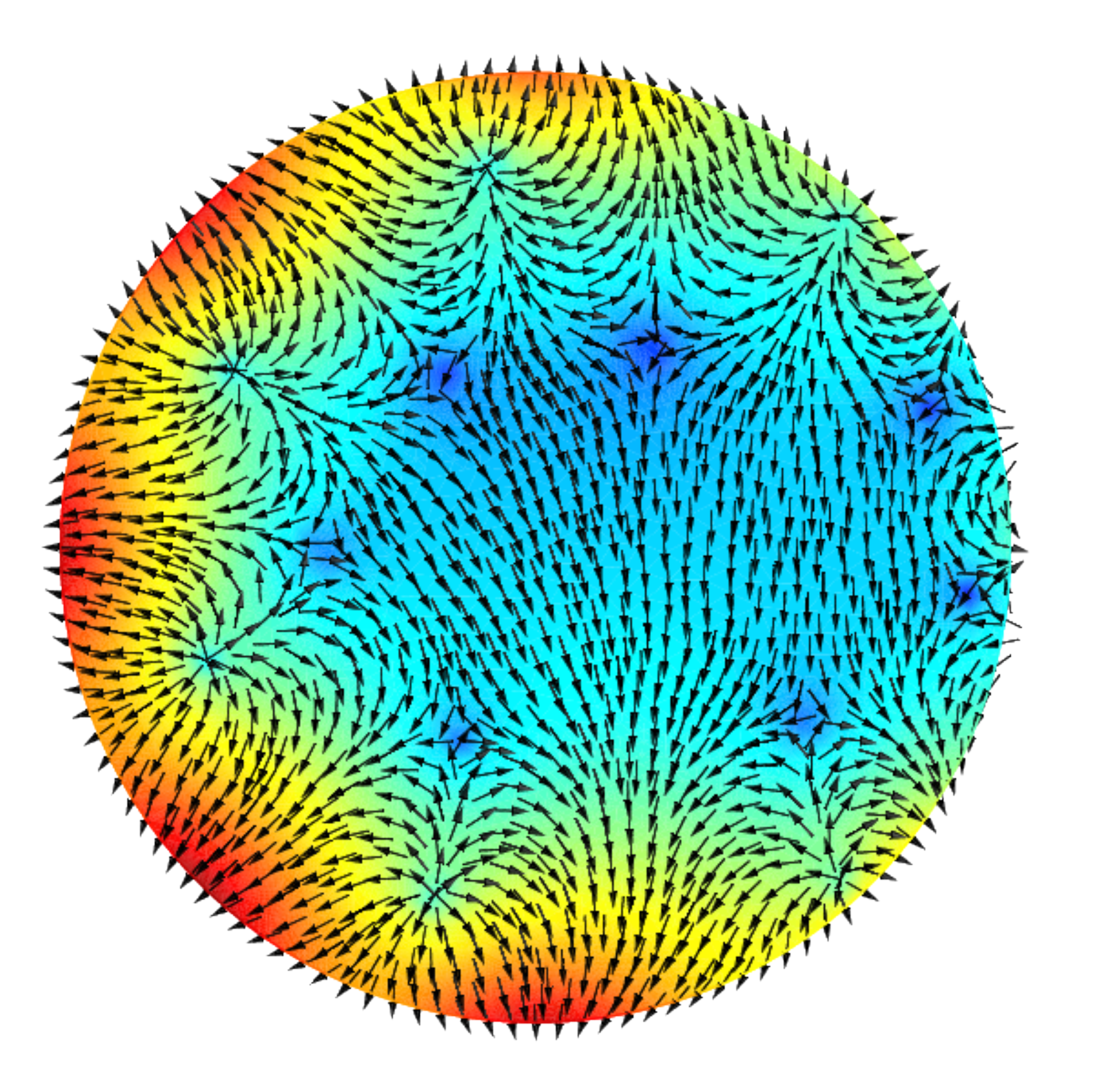}\\
\includegraphics[width=0.325\linewidth]{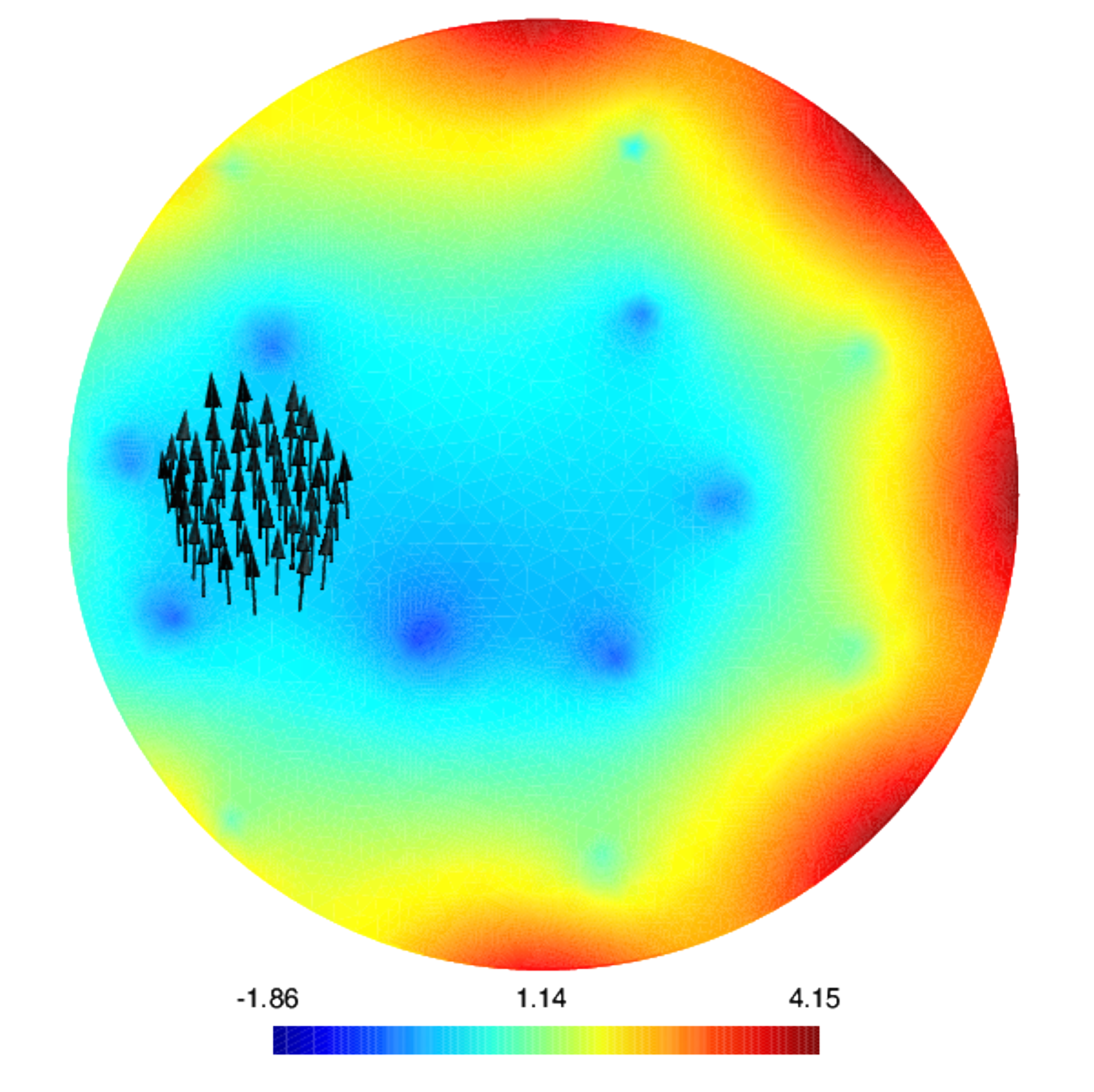}
\includegraphics[width=0.325\linewidth]{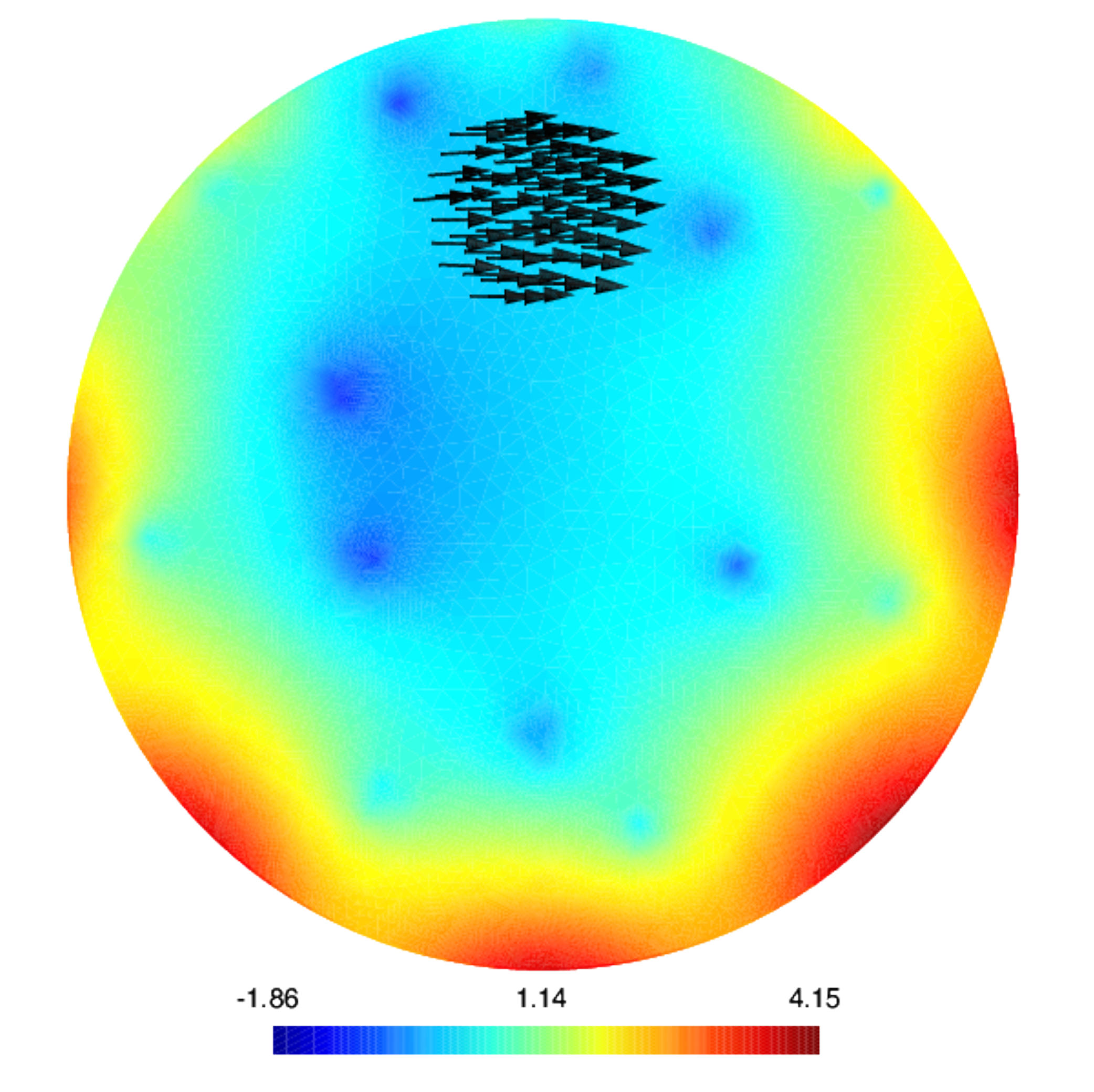}
\includegraphics[width=0.325\linewidth]{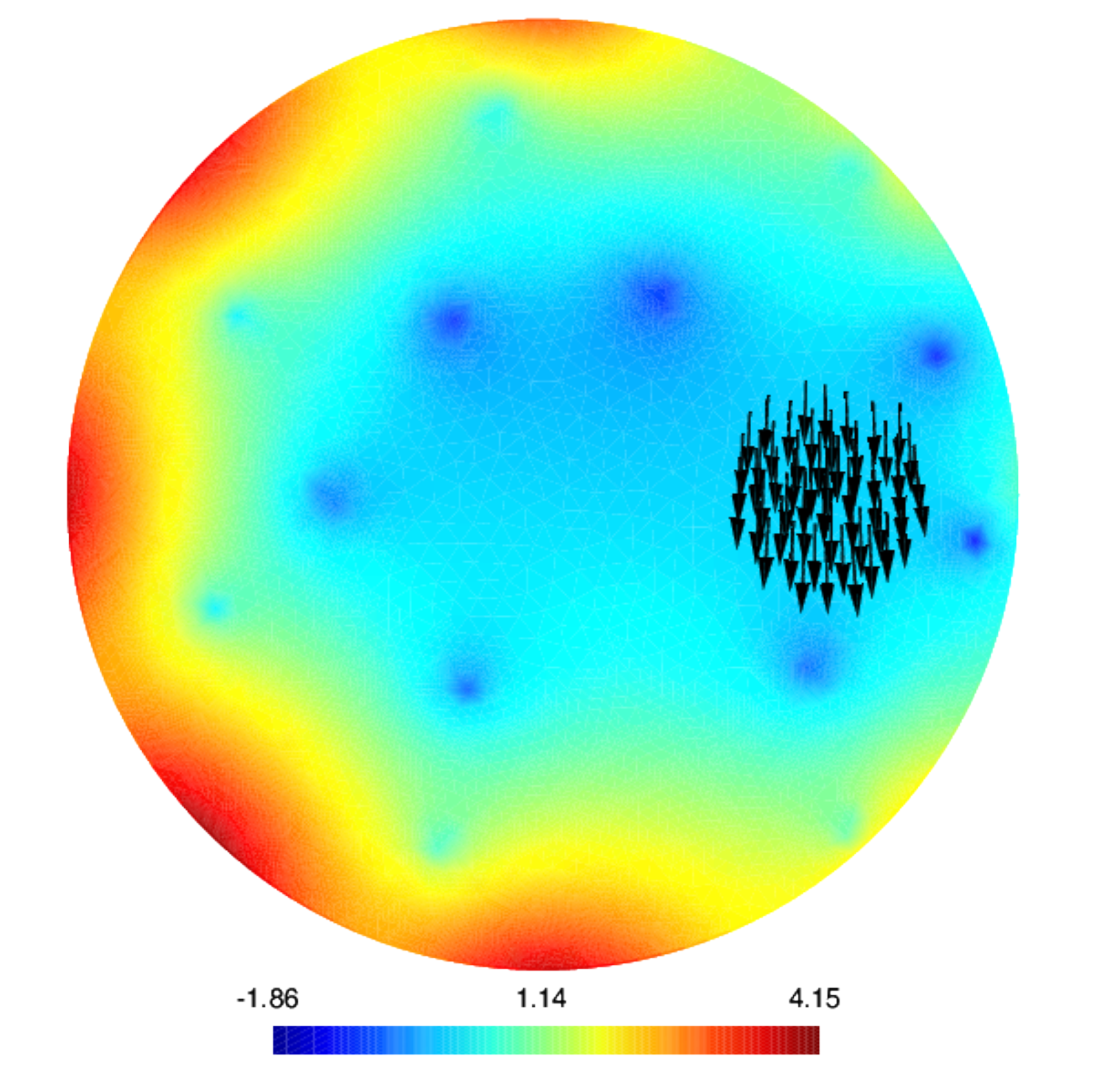}
\caption{Magnetic force solution to the minimization problem with
$D_{2,t}$ and vector field $\bvel_2$ (directions shown by black arrows).
Top figure shows the force  at three different times
$t=0.0125, 0.5$ and $1$ from left to right, respectively.
In the bottom figure  the magnetic forces is defined only on $D_{2,t}$ for the same time instances.
The magnetic force magnitude  $|\nabla|\bH|^2|$ is shown by the background coloring on a logarithmic scale.}
\label{fig:app_ex3}
\end{figure}

Next, we approximate the vector field $\bvel_2$ defined on $D_{2,t}$.
Figure \ref{fig:example2_ini}  shows the initial guess  computed by
Algorithm~\ref{alg:algorithm} and the vector intensities $(\bar{\alpha}_{i,\tau})_{i=1}^8$
solutions to the minimization problem \eqref{eq:control_problem_t_disc}.
The magnetic force  is shown  for three time instances
$t=0.0125, 0.5, 1$ in Figures~\ref{fig:app_ex3}. Here,  the magnitude of magnetic force
(in the background) and the magnetic force directions represented by arrows are depicted. The top figures
 illustrate the normalized magnetic force in $\Omega$, but in the bottom figures the field is restricted to $D_t$.
It is well-known that the magnetic forces on the boundary of the domain $\Omega$ are much higher
than inside $\Omega$. Thus making it difficult  to approximate the magnetic force when $D_t$
is close to the boundary. In fact, as we notice in the previous example, such high forces may lead to
inferior approximation of the vector field $\bvel$. The optimization problem overcomes this situation and
we observe that the optimal intensity $\bbalpha_i$ due to $i^{th}$ dipole, is smaller when $D_t$ is in
close proximity to the $i^{th}$ dipole and $\bvel$ is pointing in a  direction opposite to the dipole position.
This can be seen in Figure~\ref{fig:example1_ini} (right) and Figure~\ref{fig:example1_one} (left), where the
intensities of dipole 5 is small when $t$ is close to 0. We notice a similar behavior (small intensities)
in Figure~\ref{fig:example2_ini} (right) for dipoles
5, 3 and 1, for $t$ close to $0$, $0.5$ and $1$, respectively.
From the previous figures it can be seen that our proposed minimization procedure produces a vector field that is close to the target field.

\subsection{Problem 2: Optimal final time}
\label{sec:example2}
We consider problem~\eqref{eq:control_problem_s_disc} with the curve $\calC$  parameterized by
\[
\brho(s)=\bx_I+s\dfrac{(\bx_F-\bx_I)}{\norm{\bx_F-\bx_I}{}},\qquad s\in[0,0.75]
\]
where the starting  and end points are given by $\bx_I=(0,-0.75)$ and  $\bx_F=(0,0)$, respectively.
For the arc length we consider a uniform discretization with $M=80$
space intervals, namely, the space stepping is $\Ds=0.75/M$.
The upper and lower bounds  characterizing the admissible set $\calU^\Ds_{ad}\times\calV^\Ds_{ad}$ are
given by  $(\balpha_*,\balpha^*,\theta_*,\theta^*)=(-1,1,10^{-10},10)$.
As in the previous problem, we consider the following algorithm in order to obtain an initial guess
for problem \eqref{eq:control_problem_s_disc}:

\begin{algorithm}[!h]
\caption{: Initialization algorithm}
\begin{algorithmic}[1]

\STATE $\textbf{Input:} \, \balpha_0,\, \theta_0,\,\balpha_*,\,\theta_*\balpha^*,\,\theta^*,\, \lambda,\,\Ds,\, \eta,\, \beta,\,
\hD,\, \verb"tol",\, \brho', \, \wbP^n_i,\, m=1,\ldots, M,\, i=1,\ldots,\dd $
\STATE Set $(\mathbf{x}^0,\textrm{y}^0):=(\balpha_0,\theta_0)$
\FOR{$m=1,\ldots,M$}
    \STATE Solve for $(\mathbf{x},\textrm{y})\in \R^{\nd+1}$
    \[
    \displaystyle \underset{(\balpha_*,\theta_*)\leq(\mathbf{x},\textrm{y})\leq(\balpha^*,\theta^*)}{\min_{(\mathbf{x},\textrm{y})\in\R^{\nd+1}}}
    F(\mathbf{x},\textrm{y})
    \]
    \[
        F(\mathbf{x},\textrm{y})=
\dfrac{1}{2\textrm{y}}\sum_{i=1}^\dd\|\mathbf{x}^\top\wbP^n_i\mathbf{x}-\brho'(m\Ds)\textrm{y}\|^2_{\rL^2(\hD)}
+\dfrac{\beta}{\textrm{y}}+\dfrac{\lambda}{2\Ds^2}|\mathbf{x}-\mathbf{x}^{0}|^2
+\dfrac{\eta}{2\Ds^2}|\textrm{y}-\textrm{y}^{0}|^2
\]
 with termination criterion: $|(\mathbf{x}, \textrm{y})-\mbox{Proj}_{[\balpha_*,\balpha^*,\theta_*,\theta^*]}
 ((\mathbf{x}, \textrm{y})-\nabla F(\mathbf{x},\textrm{y}))| < \verb"tol"$.
\STATE $\balphaI(n\Ds)=\mathbf{x}$,\, $\theta_I(n\Ds)=\textrm{y}$
\STATE $\mathbf{x}^0 \gets \mathbf{x}$,\, $\textrm{y}^0 \gets \textrm{y}$
\ENDFOR
\end{algorithmic}
\label{alg:algorithm_2}
\end{algorithm}

We solve problem \eqref{eq:control_problem_s_disc}
 for an initial condition  $(\balpha_{0},\theta_{0})=(10^{-6},\ldots,10^{-6})\in \R^{\nd+1}$,
 $\beta= 10^{-1}$ and two set of cost parameters:  $(\lambda_1,\eta_1)=(10^{-6},10^{-4})$
 and $(\lambda_2, \eta_2)=(10^{-4},10^{-6})$. The initial guess is computed by using
 Algorithm~\ref{alg:algorithm_2}.
Figure~\ref{fig:finalT}  shows the evolution of velocity and intensity in
term of arc length. The intensity plots (left and center) correspond to
$(\lambda_1, \eta_1)$ and$( \lambda_2, \eta_2)$, respectively.
On the other hand, the dotted line in the right figure shows the velocity when
$(\lambda_1 , \eta_1)$ and the solid line corresponds to $(\lambda_2 , \eta_2)$.
The computed values of velocity are $\theta_\Ds(t(s^m))$, $m=1,\ldots,M$.

Since $\theta(\cdot) = \dfrac{ds}{dt}(\cdot)$, for $t \in [0,T_F]$,
we compute the final time $T_F$ by solving
\[
s_F=\int_0^{T_F}\theta_\Ds(\tau) d\tau=\sum_{i=1}^M\int_{t^{i-1}}^{t^i}\left(\theta_\Ds^{i-1} +\left(\dfrac{\tau-t^{i-1}}{t^i-t^{i-1}}\right)(\theta_\Ds^i-\theta_\Ds^{i-1})\right)d\tau,
\]
where $t^i=t(s^i), i=0,\ldots,M$. Because of the dependence on $t^i$, this equation has to be solved recursively:
\begin{align}\label{eq:t_i}
s^1=\int_{0}^{t^1}\left(\theta_\Ds^{0} +\dfrac{\tau}{t^1}(\theta_\Ds^1-\theta_\Ds^{0})\right)d\tau&\longrightarrow t^1=\dfrac{2s^1}{\theta_\Ds^1+\theta_\Ds^0}\nonumber\\
s^2=s^1+\int_{t^{1}}^{t^2}\left(\theta_\Ds^{1} +\left(\dfrac{\tau-t^{1}}{t^2-t^{1}}\right)
(\theta_\Ds^2-\theta_\Ds^{1})\right)d\tau&\longrightarrow t^2=\dfrac{2(s^2-s^1)}{\theta_\Ds^2+\theta_\Ds^1}+t^1\nonumber\\
\vdots&\phantom{\longrightarrow} \vdots\nonumber\\
s^i=s^{i-1}+\int_{t^{i-1}}^{t^i}\left(\theta_\Ds^{i-1} +\left(\dfrac{\tau-t^{i-1}}{t^i-t^{i-1}}\right)(\theta_\Ds^i-\theta_\Ds^{i-1})\right)d\tau&\longrightarrow t^i=\dfrac{2(s^{i}-s^{i-1})}{\theta_\Ds^{i}+\theta_\Ds^{i-1}}+t^{i-1}
\end{align}
for $ i=1,\ldots,M$.
Therefore, the final time $T_F$ is given by
\[
T_F=t^M=2\sum_{i=1}^M\dfrac{s^{i}-s^{i-1}}{\theta_\Ds^{i}+\theta_\Ds^{i-1}}=2\Ds\sum_{i=1}^M\dfrac{1}{\theta_\Ds^{i}+\theta_\Ds^{i-1}}.
\]
\begin{figure}
\centering
\includegraphics[width=0.33\linewidth]{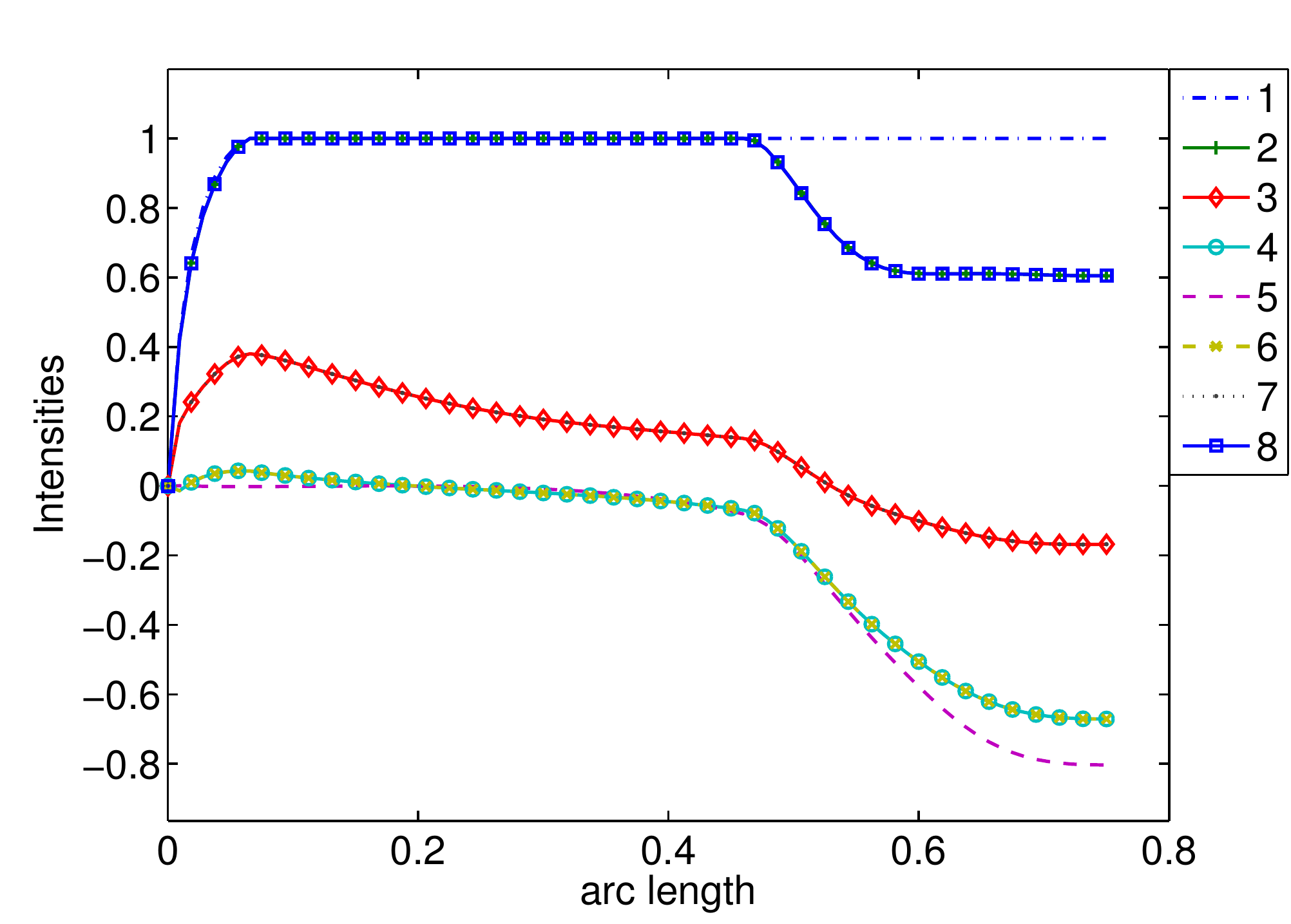}\!
\includegraphics[width=0.33\linewidth]{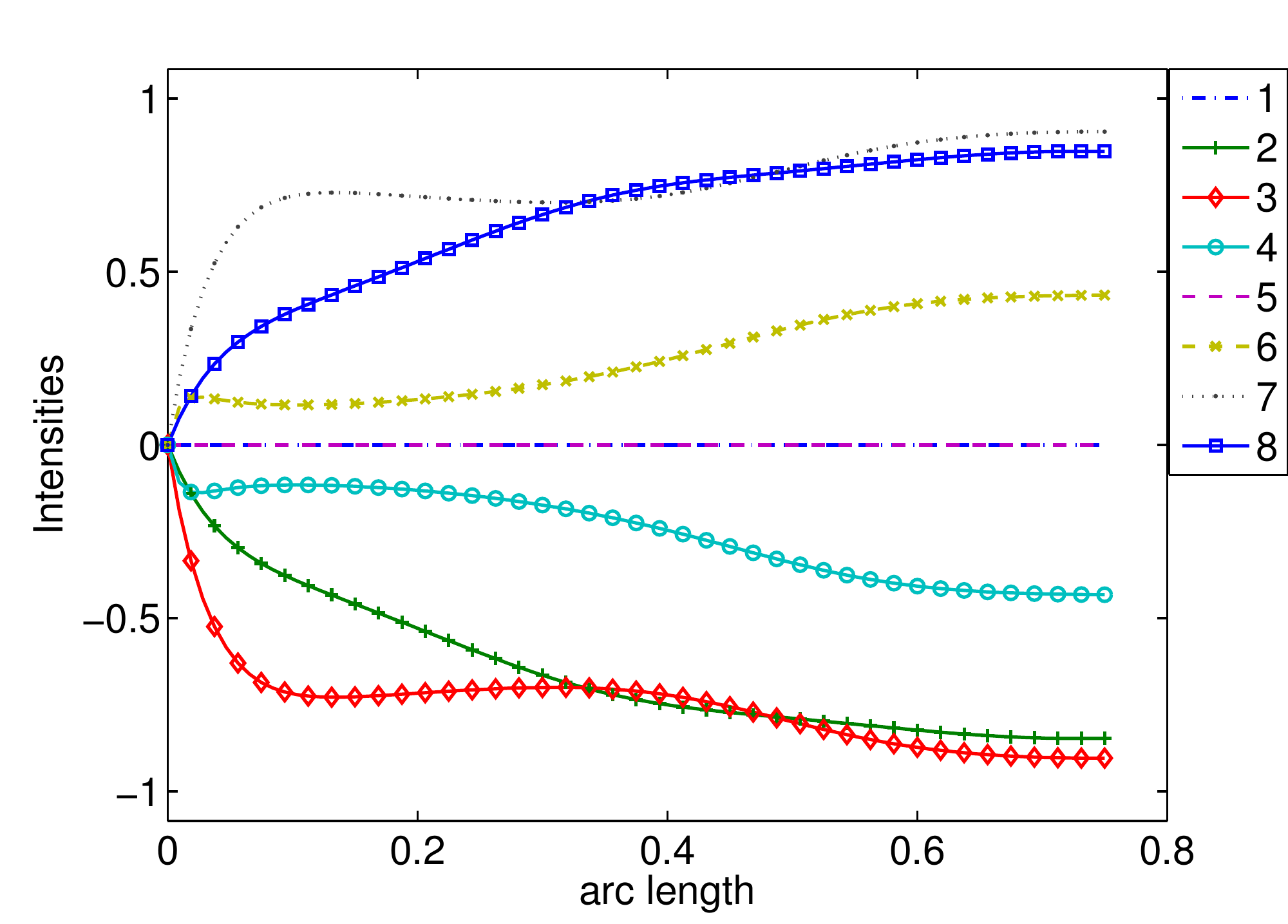}\!
\includegraphics[width=0.33\linewidth]{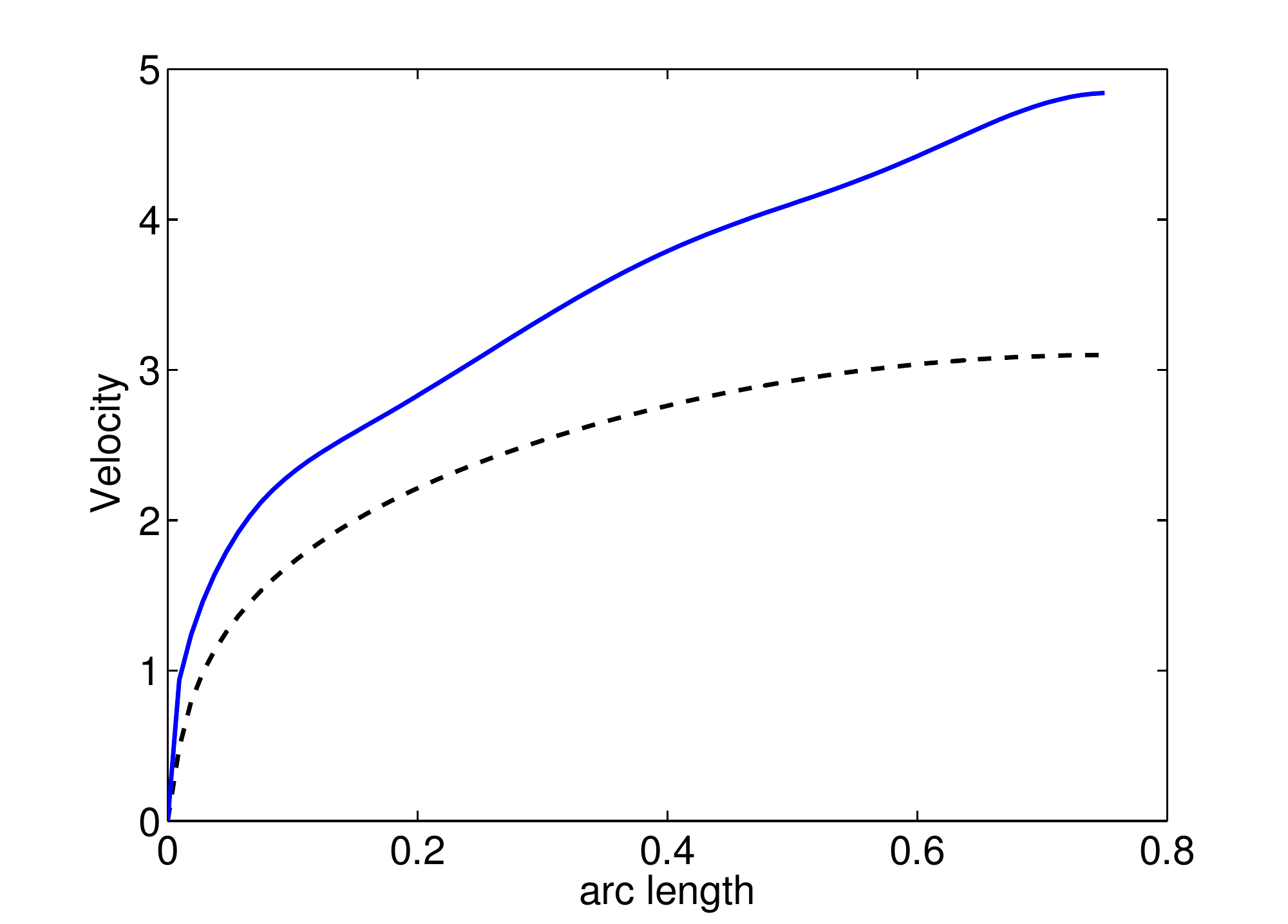}
\caption{Intensity and velocity solutions to problem \eqref{eq:control_problem_s_disc} with $\beta=10^{-1}$.
Evolution of the intensities $\{\balpha^i_\Ds\}_{i=1}^8$  for parameters $(\lambda_1,\eta_1)$ (left)
and $(\lambda_2,\eta_2)$ (center), and  velocity $\theta_\Ds$ (right) for both parameters.
The dotted line shows the velocity when
$(\lambda_1 , \eta_1)$ and the solid line corresponds to $(\lambda_2 , \eta_2)$.
The horizontal axis represents the arc length $s$ of the curve $\calC$ with $s_F=0.75$.}
\label{fig:finalT}
\end{figure}
%%
%\begin{figure}
%\centering
%\includegraphics[width=0.33\linewidth]{imag/int_vel_ex1_t_20k}\!
%\includegraphics[width=0.33\linewidth]{imag/int_vel_ex2_t_20k}\!
%\includegraphics[width=0.33\linewidth]{imag/velo_ex1_20k_ex2_20k_t}
%\caption{Intensity and velocity solutions to problem \eqref{eq:control_problem_s_disc} with $\beta=10^{-1}$.
%Evolution of the intensities $\{\balpha^i_\Ds\}_{i=1}^8$  for parameters $(\lambda_1,\eta_1)$ (left)
%and $(\lambda_2,\eta_2)$ (center), and  velocity $\theta_\Ds$ (right) for both parameters.
%The dotted line shows the velocity when
%$(\lambda_1 , \eta_1)$ and the solid line corresponds to $(\lambda_2 , \eta_2)$.
%The horizontal axis represents the time $t$ computed by \eqref{eq:t_i}.
%}
%\label{fig:finals}
%\end{figure}
%%
Using the aforementioned formules  we deduce that $T_F = 0.485$
 when $(\lambda_1, \eta_1)$ and $T_F =0.337$ when $(\lambda_2,\eta_2)$.
% In Figure~\ref{fig:finals} we illustrate the behavior of intensities with respect to time t.

In order to increase
the force in direction $\brho'=(1,0)^\top$ (cf.~\eqref{eq:F_s}), in principle, it is sufficient
to have only the dipole number one (see Figure~\ref{fig:domain})  with maximum intensity
$\bbalpha_1$ (cf. Figure~\ref{fig:push}, left). Indeed it is clear
from Figure~\ref{fig:finalT}~(left) that the constraint $\bbalpha_1 = \balpha^*$, is
active for certain time instances.
However, even though we can increase the magnetic force with only one dipole,
it is safe to conclude that, for this example, even if we set $\bbalpha = \balpha^*$
for all times, we do not achieve a uniform constant vector field (which is our goal).
Then the remaining dipoles, which have nonzero intensities as well, contribute to attain this.
On the contrary, Figure~\ref{fig:finalT}~(center) shows a different behavior.
Here the dipole $\bbalpha_1$ on the right has the same intensity as the dipole $\bbalpha_5$
on the left  (close to 0), whereas the dipoles with the
largest intensities are 3, 7 (at the center of $\O$) and 2, 8 (on the right of $\O$).
Notice that the penalization $\lambda_1=10^{-6}$ of the cost functional leads to fast increasing
values of dipoles intensities 1,2 and 8, which is not the case
when $\lambda_2=10^{-4}$.

\subsection{Application: Transport of a passive scalar}
Magnetic drug targeting is an important application of ferrofluids where drugs,
with ferromagnetic particles in suspension, are injected into the blood stream. The external
magnetic field thus concentrates the drug to the most relevant areas, for example, solid tumors
(see, for instance, \cite{Lubbe1996}).
We assume a concentration of magnetic nanoparticles confined in a domain $\widetilde{\Omega}\subset \R^\dd, \dd=2,3$.
Let $c$ be the drug concentration and $\bH$ the magnetic field, then the evolution
 of  $c$ by the applied magnetic field is given by the following
  advection-diffusion model \cite{GR2005}:
\begin{align}\label{eq:state_strong_1}
\dfrac{\partial c}{\partial t}+\di \left( -A\nabla c   +c \bv{u} +
  \gamma_1 cf(\bH)\right)=0 \quad \mbox{in } \widetilde{\Omega}\times(0,T)&\\
c=0\quad \mbox{on } \partial\widetilde{\Omega}\times(0,T)\qquad c(x, 0) = c_0
\quad \mbox{in } \widetilde{\Omega}&\label{eq:state_strong_2}\\
 \curl \bH=\bv{0} \quad \mbox{in } \widetilde{\Omega}\qquad   \di \del{\mu\bH}=0
 \quad \mbox{in } \widetilde{\Omega}&\label{eq:state_strong_max}
\end{align}
where  $A$ is a diffusion coefficient matrix,  $\bv{u}$ is a fixed velocity vector
and $f$ is the  \textit{Kelvin force} depending on $\bH$ (cf~\eqref{eq:force_1}).
If the  magnetic susceptibility $\chi$ is independent of $\bH$, then
$f(\bH)=\gamma_2 \mu_0  \chi \nabla|\bH|^2$ where $\gamma_1$ and
$\gamma_2$ are constitutive constants with different units and $\mu_0$
denotes the magnetic permeability.

Under the principle of magnetic drug delivery, we aim to move an initial concentration
$c_0$ of drugs from one subdomain to another (desired location) using the magnetic
force while minimizing the spreading. In this example we focus on ``magnetic injection"
of the concentration away from the boundary.
The two fundamental units that determine
the evolution of concentration $c$ are transport and diffusion (cf.~\eqref{eq:state_strong_1}).
Given the variability of the magnetic force in $\O$ (see, Figures~\ref{fig:app_ex3} (top)),
the major challenges are: to generate the appropriate magnetic
force
to move the concentration to a desired location and to control the
spreading due to the diffusion in \eqref{eq:state_strong_1}.

Indeed, we can overcome the first of these challenges by using the ``almost uniform" magnetic force
 generated using magnetic dipoles in the previous two examples. Recall that $D_t$ moves
 along a pre-specified curve $\calC$. In fact, in our computations we notice that the magnetic
 force in $D_t$ helps in minimizing the spread of $c$ as well.

To fix ideas, we set $\widetilde{\Omega}:=\Omega\subset \R^2$ be a ball of unit radius centered at $(0,0)$ and the dipole
configuration is the same as in the previous examples. We assume that $c_0$ is as in Figure~\ref{fig:app_ex4} (left),
and lies inside a ball centered at $(-0.75,0)$ with radius $0.2$ (see Figure~\ref{fig:app_ex4} (left)). We also set final time $T=1$.
For simplicity we assume $\gamma_1=\gamma_2=\mu_0 =\chi=1$, $\bv{u}=\bv{0}$. In order to
further reduce the spread we choose a small diffusion coefficient, in particular, we set
$A=\varepsilon\mathbb{I}$, with $\varepsilon=10^{-5}$.

We consider piecewise linear functions on simplicial meshes to approximate \eqref{eq:state_strong_1}--\eqref{eq:state_strong_2}. However, it is well--known that
the standard finite element method yields solution oscillations when $\varepsilon\ll |f(\bH)|$.
A possible remedy is to add an artificial term to stabilize the numerical scheme.
We use the so--called SUPG technique  (see, for instance, \cite{BH1982}).
Let $f(\bH) = \nabla|\bH|^2$ as computed  in Section~\ref{sec:example1} by solving \eqref{eq:control_problem_t_disc}
$\bvel = \bvel_1=(1,0)$. Then we solve \eqref{eq:state_strong_1}-\eqref{eq:state_strong_2} for $c$.

 Figure~\ref{fig:app_ex4} shows the evolution of the
concentration for three times instances.
\begin{figure}
\centering
\includegraphics[width=0.325\linewidth]{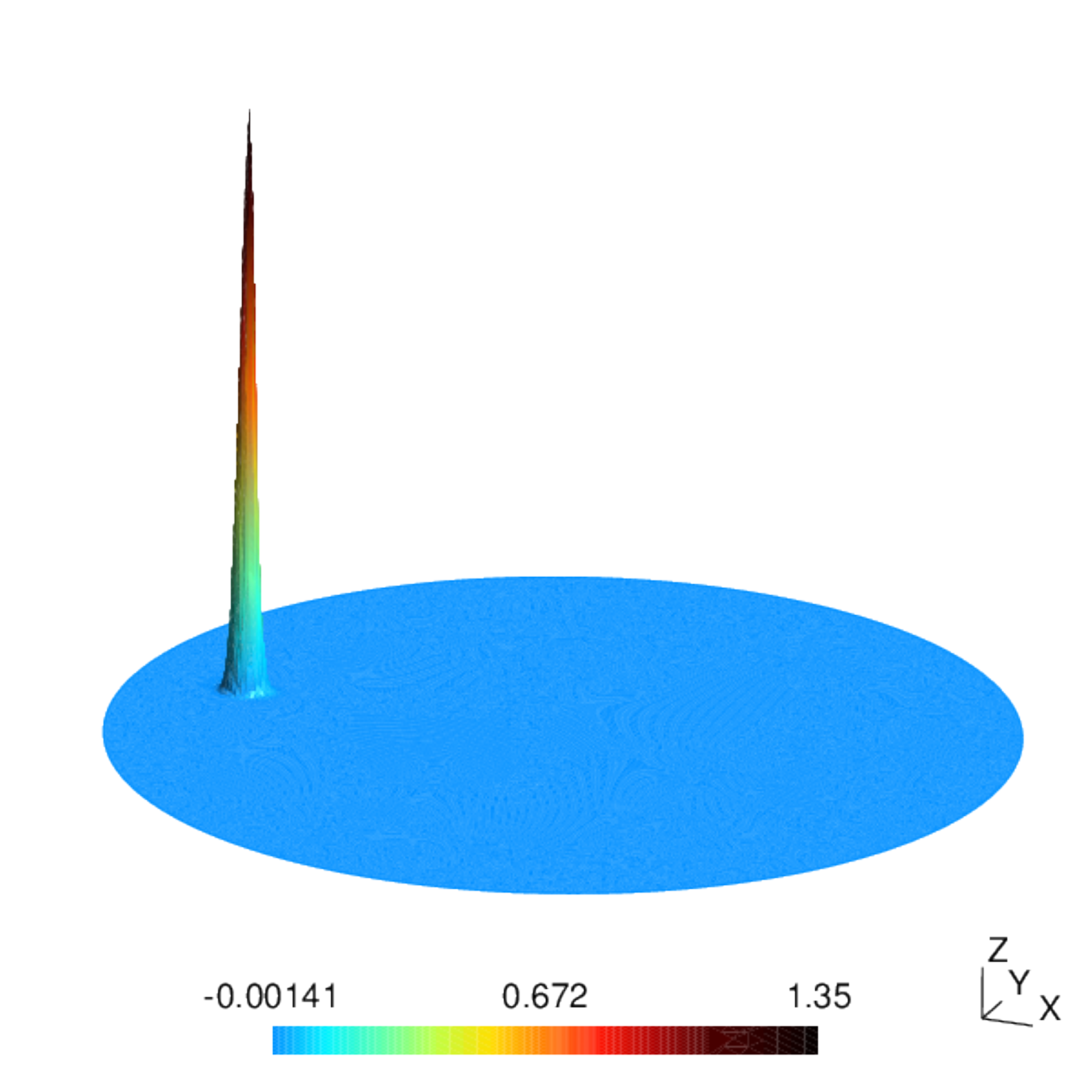}
\includegraphics[width=0.325\linewidth]{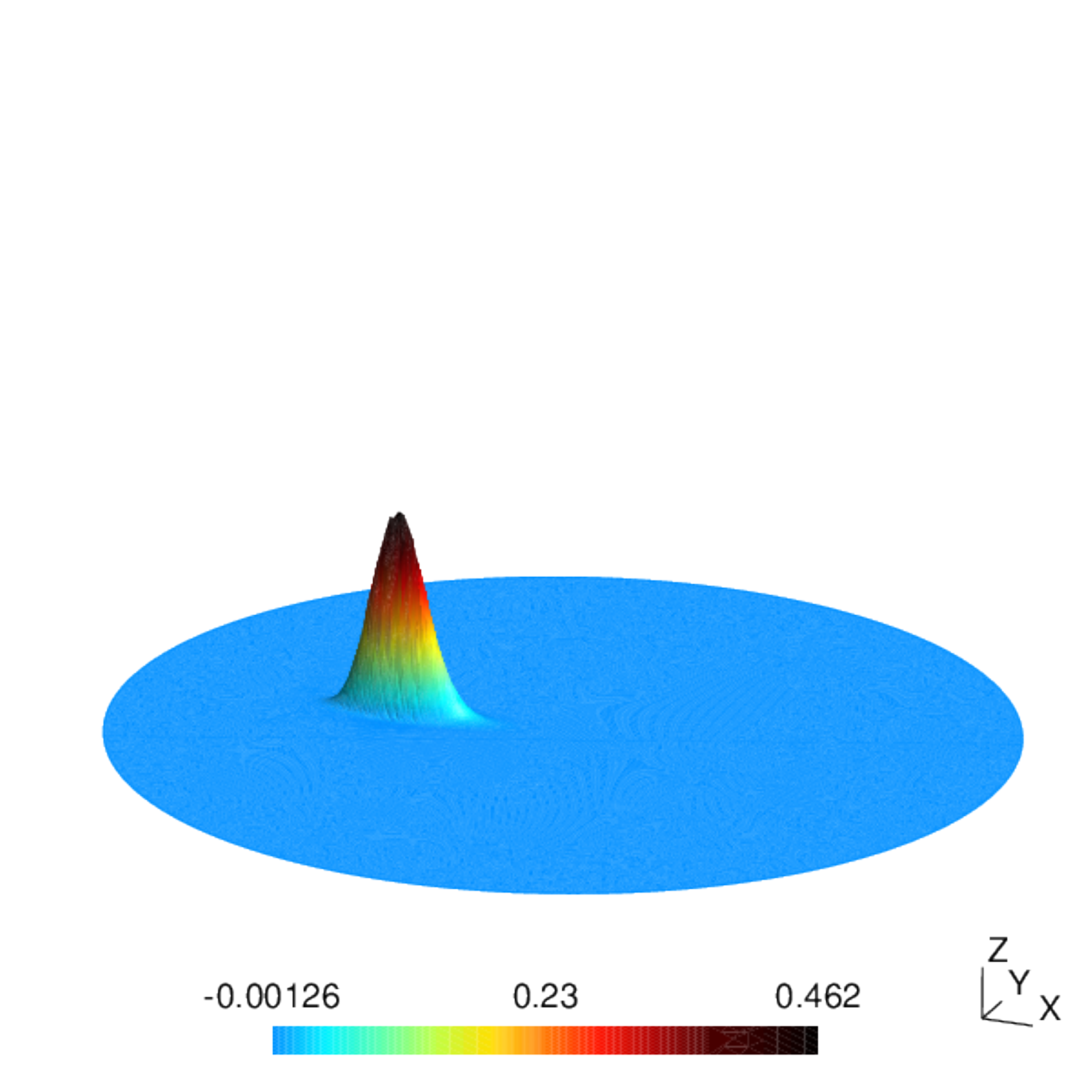}
\includegraphics[width=0.325\linewidth]{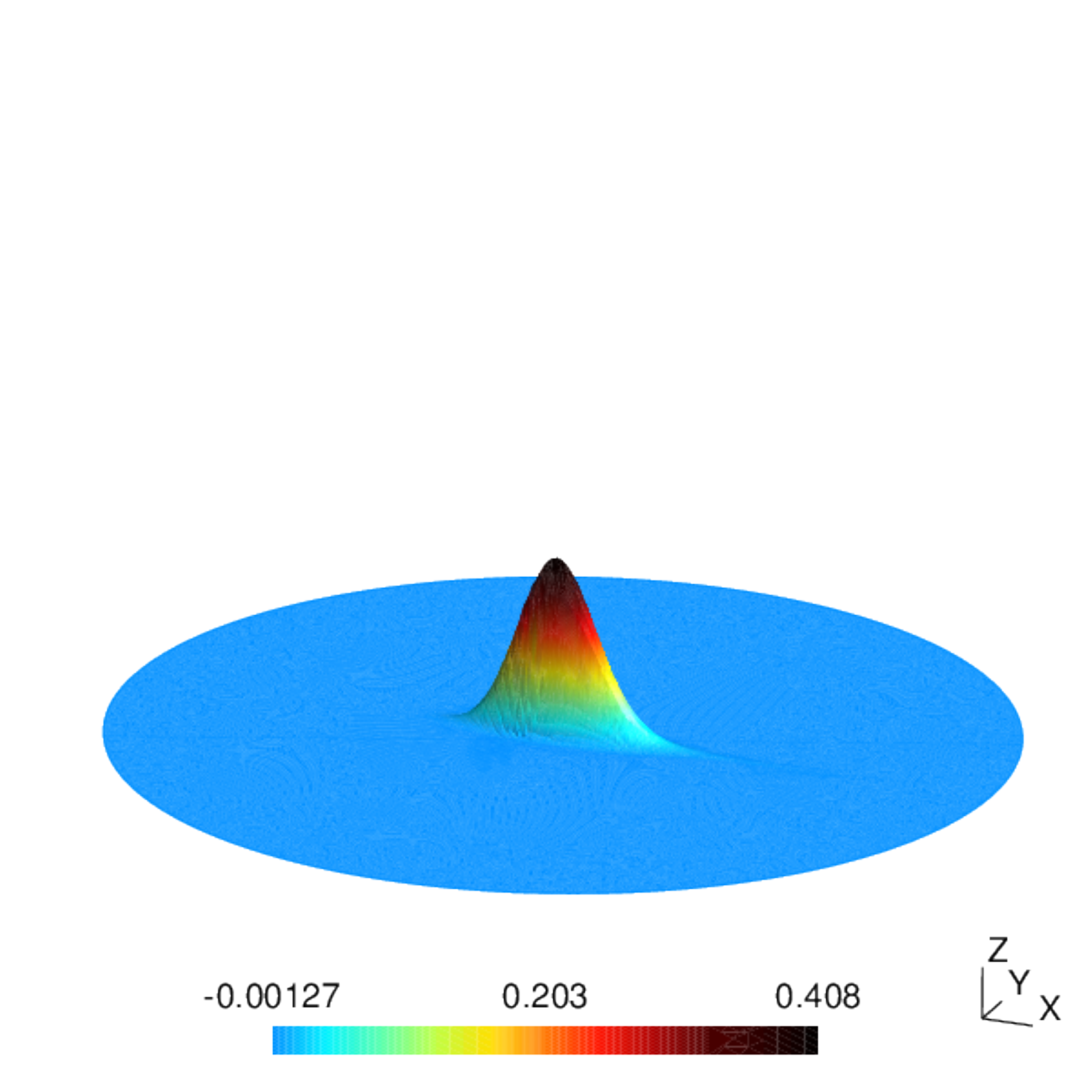}\\
\includegraphics[width=0.325\linewidth]{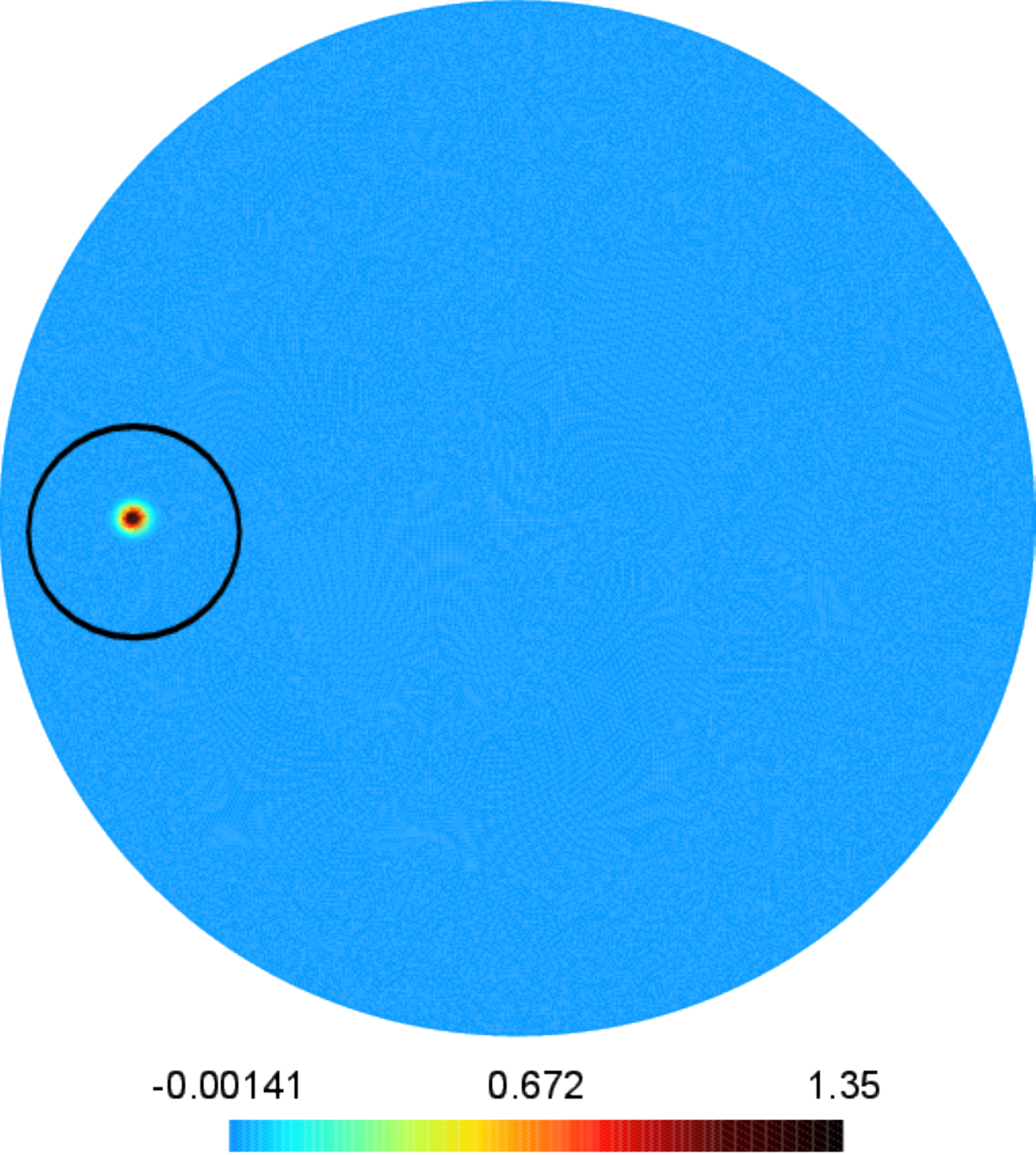}
\includegraphics[width=0.325\linewidth]{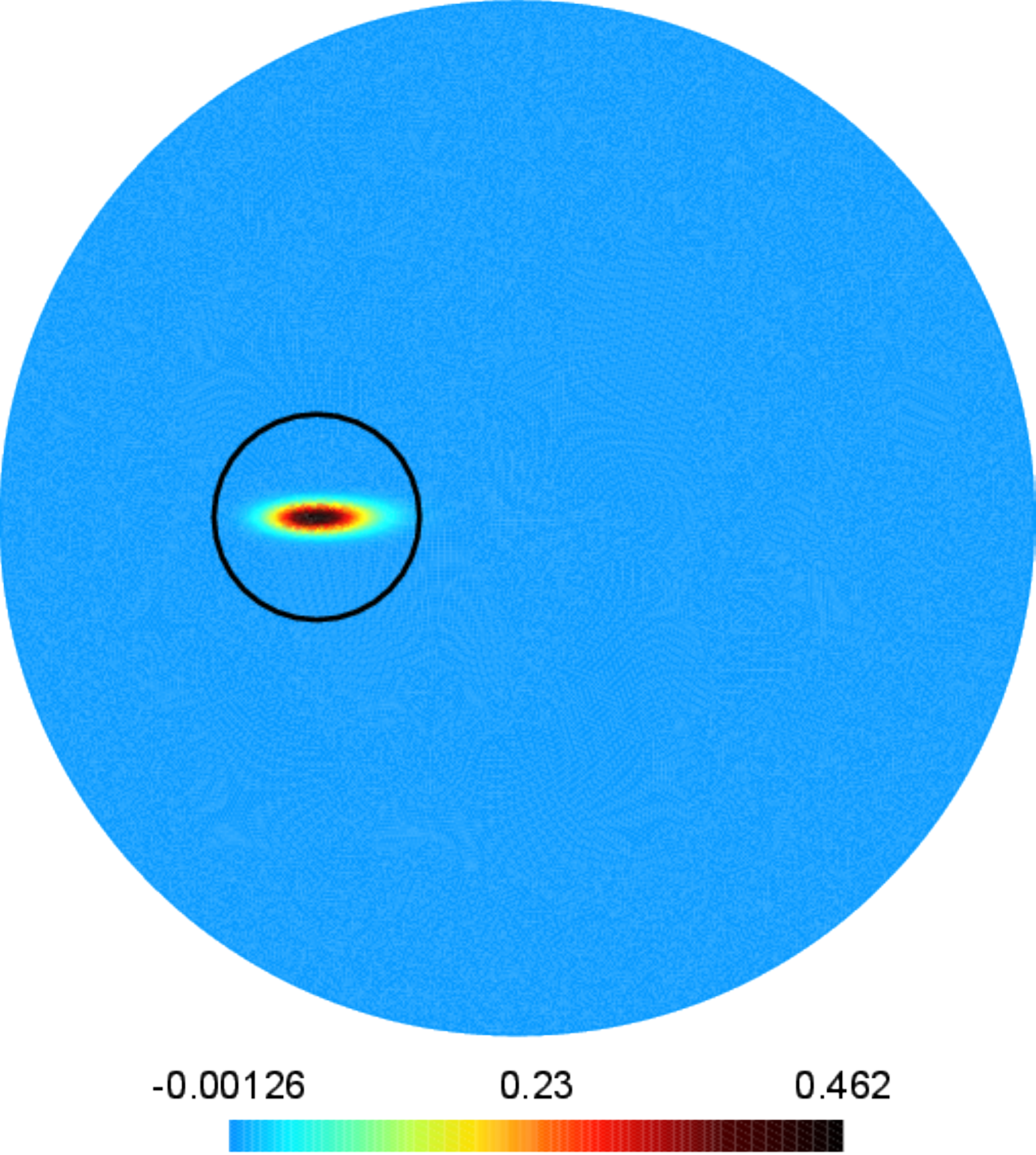}
\includegraphics[width=0.325\linewidth]{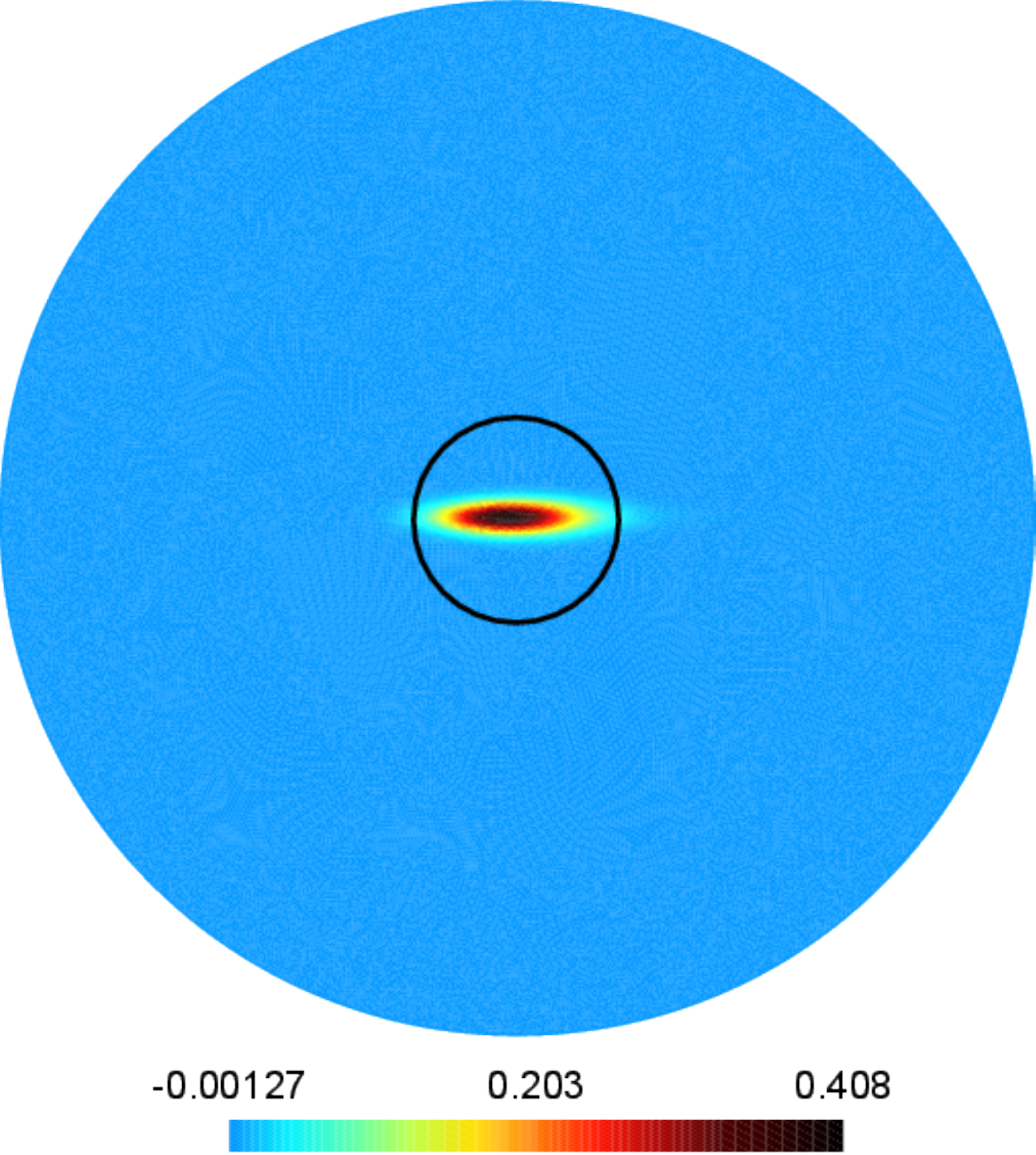}
\caption{ Evolution of the concentration in $\O$ and moving domain $D_{1,t}$ (circle) at times $t=0,0.5, 1$ s. $\eps=10^{-5}$ with two different views.}
%Evolution of the concentration in $\O$ at times $t=0,0.5, 1$ s.}
\label{fig:app_ex4}
\end{figure}

From Figure~\ref{fig:app_ex4}, we notice that  most part of the concentration $c$ is confined in $D_{1,t}$
(denoted by the smaller circle) for all times. Indeed, our approach minimizes  spreading and prevents
concentration from reaching $\partial\Omega$.
 Otherwise, part of the concentration could  be transported to the boundary where the closest (active)
 dipole is positioned, which is not our goal. Figure~\ref{fig:app_ex4}
 illustrate that the concentration moves from the initial configuration and reaches at center of $\O$.
%

%----------------------- Reference --------------------------------------
\bibliographystyle{plain}

\bibliography{Ref}

\end{document}